\documentclass[10pt,reqno]{amsart}

\usepackage[margin=10pt,font=small,labelfont=bf,justification=centering]{caption}
\usepackage{lipsum}
\usepackage{caption}
\usepackage{subcaption}

\usepackage{amscd}
\usepackage{amsthm}
\usepackage{amssymb}

\usepackage{enumerate,array,mathrsfs, a4wide}
\usepackage{graphicx}
\usepackage{sseq}
\usepackage{xcolor}

\usepackage{rotating}
\usepackage{bbm}
\usepackage[all,cmtip]{xy}

\usepackage[latin1]{inputenc}
\usepackage{dsfont}

\usepackage{hyperref}
\usepackage{accents}
\newcommand{\ubar}[1]{\underaccent{\bar}{#1}}

\newtheorem{theorem}{Theorem}[section]

\newtheorem{lemma}[theorem]{Lemma}
\newtheorem{proposition}[theorem]{Proposition}
\newtheorem{corollary}[theorem]{Corollary}

\newtheorem{definition}[theorem]{Definition}

\newtheorem{question}[theorem]{Question}

\newtheorem{remark}[theorem]{Remark}

\newcommand{\ow}{\omega}

\newcommand{\C}{{\mathbb{C}}}
\newcommand{\R}{{\mathbb{R}}}
\newcommand{\Q}{{\mathbb{Q}}}
\renewcommand{\P}{{\mathbb{P}}}
\newcommand{\Z}{{\mathbb{Z}}}
\newcommand{\N}{{\mathbb{N}}}

\renewcommand{\epsilon}{\varepsilon}
\renewcommand{\phi}{\varphi}
\renewcommand{\theta}{\vartheta}

\newcommand{\w}{\wedge}

\newcommand{\coker}{\mathop {coker}}

\newcommand{\Cont}{\mathop {Cont}}
\newcommand{\diag}{\mathop {diag}}

\newcommand{\id}{\mathop {Id}}
\newcommand{\im}{\mathop {Im}}
\newcommand{\ind}{\mathop {ind}}
\newcommand{\lcm}{\mathop {lcm}}

\newcommand{\OB}{\mathop {OB}}

\newcommand{\rk}{\mathop {rk}}

\newcommand{\sig}{\mathop {sig}}

\newcommand{\SO}{\mathop {SO}}

\newcommand{\tr}{\mathop{Tr}}

\newcommand{\Sp}{\mathop{Sp} }
\newcommand{\U}{\mathop U}

\begin{document}
\title{Brieskorn manifolds in contact topology}

\author{Myeonggi Kwon and Otto van Koert}




\begin{abstract}
In this survey, we give an overview of Brieskorn manifolds and varieties, and their role in contact topology. We discuss open books, fillings and invariants such as contact and symplectic homology. We also present some new results involving exotic contact structures, invariants and orderability.
The main tool for the required computations is a version of the Morse-Bott spectral sequence. We provide a proof for the particular version that is useful for us.
\end{abstract}
\maketitle

\section{A brief historical overview and introduction}
Brieskorn varieties are affine varieties of the form
$$
V_\epsilon(a_0,\ldots,a_n)=\{ (z_0,\ldots,z_n)\in \C^{n+1} ~|~\sum_{j=0}^n z_j^{a_j}=\epsilon \}
,
$$
and are a natural generalization of Fermat varieties.
They became popular after it was observed by Hirzebruch that links of singular Brieskorn varieties at $0$, meaning sets of the form
$$
\Sigma(a_0,\ldots,a_n)=\{ (z_0,\ldots,z_n)\in \C^{n+1}~|~\sum_{j=0}^n |z_j|^2=1 \text{ and }
\sum_{j=0}^n z_j^{a_j}=0
\}
,
$$
can sometimes be homeomorphic, but not diffeomorphic to spheres.
Some essential ingredients for the necessary computations were worked out by Pham, \cite{Pham}, and a rather complete picture was worked out by Brieskorn, \cite{Brieskorn:difftop}.
The above links $\Sigma(a_0,\ldots,a_n)$ are now known as Brieskorn manifolds.

In contact geometry, Brieskorn manifolds were first recognized as contact manifolds by Abe-Erbacher, Lutz-Meckert and Sasaki-Hsu around 1975-1976, \cite{Abe:Brieskorn,LM,Sasaki:Brieskorn}.
In particular, their work showed that at least some exotic spheres admit contact structures.
Around the same time, Thomas \cite{Thomas:almost_regular} also used Brieskorn manifolds to establish existence results for contact structures on simply-connected $5$-manifolds.
With hindsight, the role of Brieskorn manifolds in early constructions of contact $5$-manifolds can be clarified;
a classification result for simply-connected $5$-manifolds by Smale, \cite{Smale:5-mfd}, combined with a homology computation, see Section~\ref{sec:homology_Brieskorn}, shows that every simply-connected spin $5$-manifold is actually the connected sum of Brieskorn manifolds.

Since Brieskorn manifolds also include many spheres with the standard smooth structure, these manifolds can be used to show that there are non-standard contact structures on $S^{2n-1}$ following Eliashberg, \cite{eliashberg:nonstd}.
The main ingredient is a result of Eliashberg-Gromov-McDuff stating that aspherical symplectic fillings for the standard contact sphere $(S^{2n-1},\xi_0)$ are diffeomorphic to $D^{2n}$.
Smooth Brieskorn varieties, natural symplectic fillings for Brieskorn manifolds, are symplectically aspherical, but typically have a lot of homology. Hence the contact structure on the boundary, a Brieskorn manifold, is non-standard provided none of the exponents equals $1$.

Later on, Ustilovsky, \cite{Ustilovsky:contact}, showed that there are infinitely many non-isomorphic contact structures on $S^{4n+1}$ with the same formal homotopy class of almost contact structures as the standard contact sphere.

Because Brieskorn manifolds have a rich and rather understandable structure, they are still of interest, and in this note we will highlight some of their more recent applications to contact and symplectic topology, including some new results on orderability and exotic contact structures.
\subsection{Exotic contact structures and applications to contact topology}
In this survey, we will start by describing classical results concerning homology and exotic spheres, and apply these results to obtain some symplectic invariants of Brieskorn varieties, namely symplectic homology.
These invariants are in some cases invariants of the contact structure on the boundary, the Brieskorn manifold, and we will use this to describe a multitude of exotic contact structures on various kinds of manifolds.

To state the first result, we recall that symplectic homology, denoted by $SH$, is a Floer type invariant of symplectic manifolds with contact type boundary.
There are various versions of symplectic homology, including equivariant symplectic homology, denoted by $SH^{S^1}$.
The symplectic manifolds we will consider will be so-called Liouville domains with vanishing first Chern class. 
This means that there is an exact symplectic form and symplectic homology can be equipped with an integer grading.
We will briefly review symplectic homology in Section~\ref{sec:blackbox_SH}.
The mean Euler characteristic of the $+$-part of equivariant symplectic homology of a Liouville domain $(W,d\lambda)$ is then
$$
\chi_m(W,d\lambda)=\lim_{N\to \infty} \frac{1}{N}\sum_{k=-N}^N(-1)^k \rk SH^{+,S^1}_k(W,d\lambda).
$$
This limit is not always defined, and contains less information than the homology itself.
Also, by definition the mean Euler characteristic depends on the Liouville filling, and not only on the contact type boundary.
In Lemma~\ref{lemma:mec_invariant} we will give sufficient conditions on the contact manifold that guarantee that the mean Euler characteristic is an invariant of the contact structure on the boundary.
The precise conditions, which we will refer to as a contact manifold with \emph{convenient dynamics}, are somewhat technical.
As a particular case, we point out that a contact form with a periodic Reeb flow for which the mean index is non-zero is an example of a contact manifold with convenient dynamics.

In view of this, consider $\Xi_{nice}(S^5)$, the monoid of contact structures on $S^5$ that are convex fillable by simply-connected Liouville manifolds with vanishing first Chern class and that admit convenient dynamics.
The (boundary) connected sum is the monoid operation.
Then 
\begin{theorem}
Given any rational number $x$, there is a Stein fillable contact structure $\xi_x$ on $S^5$ whose mean Euler characteristic equals $\chi_m(W_x,\omega_x)=x$.
Furthermore, the map
\[
\begin{split}
\tilde \chi_m: (\Xi_{nice}(S^5),\#) & \longrightarrow (\Q,+) \\
(\xi,W) & \longmapsto \chi_m(W,\omega)-\frac{1}{2}.
\end{split}
\]
is a surjective monoid homomorphism.
\end{theorem}
This theorem is proved by computing this invariant for Brieskorn manifolds and applying a simple connected sum formula.
We have several other results about exotic contact structures and invariants, but due to their more technical nature we will refer to Sections~\ref{sec:examples},~\ref{sec:MEC},~\ref{sec:MEC_Brieskorn} and \ref{sec:not_full_invariants} for their statements.

We also investigate an invariant of the filling, namely symplectic homology.
\begin{theorem}
Let $(\Sigma(a_0,\ldots,a_n),\alpha_a)$ be a Brieskorn manifold with filling $V_\epsilon(a_0,\ldots,a_n)$, and suppose that none of the exponents $a_i$ equals $1$.
Then $SH_*(V_\epsilon(a_0,\ldots,a_n)\, )$ does not vanish.
Furthermore, $(\Sigma(a_0,\ldots,a_n),\alpha_a)$ is orderable.
\end{theorem}
Non-vanishing of symplectic homology follows from the fact that the Milnor number is positive. This guarantees the existence of Lagrangian spheres which in turn imply non-vanishing of symplectic homology. In particular, the same holds true for links of \emph{isolated singularities} with positive Milnor number, see Theorem~\ref{thm:SH_isolated_sing}.
The last statement on orderability is related to non-vanishing of symplectic homology and depends on recent work due to Albers and Merry, \cite{AM:orderability}, see Theorem~\ref{thm:orderable_albers_merry} for a more precise formulation.
We remind the reader that the standard contact sphere, which is realized by any Brieskorn manifold for which at least one of the $a_i$'s equal to $1$, is not orderable.
In fact, the filling of $\Sigma(1,a_1,\ldots,a_n)$ by the Brieskorn variety $V_0(1,a_1,\ldots,a_n)$, is the filling by a ball, which has vanishing symplectic homology.

Interestingly, the Brieskorn manifolds satisfying the above conditions, include contact manifolds that are diffeomorphic to spheres in dimension greater than $3$, so we obtain
\begin{corollary}
For $n\geq 3$, there are infinitely many non-contactomorphic contact structures on $S^{2n-1}$ that are orderable.
\end{corollary}

We also give explicit examples where equivariant symplectic homology with $\Q$-coefficients cannot distinguish certain contact structures on $S^2\times S^3$.
This is the family of Brieskorn manifolds $\Sigma(2k,2,2,2)$, where we may regard $\Sigma(0,2,2,2):=\OB(T^*S^2,\id)$ as a subcritically fillable structure on $S^2\times S^3$ (this notation is explained in Sections~\ref{sec:contact_open_book} and \ref{sec:monodromy_Brieskorn}).
To distinguish two of them, we use the fact that $\Sigma(2k,2,2,2)$ for $k>0$ does not embed into a subcritical Stein manifold.
\begin{proposition}
On $S^2\times S^3$, there are Stein-fillable contact structures $\xi_1$ and $\xi_2$ that are not contactomorphic.
However, the cylindrical contact homology/equivariant symplectic homology groups with $\Q$-coefficients of these contact structures are isomorphic. 
\end{proposition}
For non-simply-connected contact manifolds, a related, and much stronger result is already known, namely that there are contact manifolds that are not even diffeomorphic, yet with symplectomorphic symplectizations, \cite{Courte:symplectizations}.

\section{Links of isolated singularities and contact structures}
Let $p:(\C^n,0)\to (\C,0)$ be a holomorphic function with an isolated singularity at $0$.
Define the singular variety $V_0(p):=p^{-1}(0)$.
The {\bf link} of the isolated singularity is defined as 
$$
L_{0,\delta}(p):=V_0(p)\cap S^{2n-1}_\delta.
$$
Put differently, if we define the function $f$ on $\C^n$ by putting $f(z)=|z|^2$, then $L_{0,\delta}(p)$ is the level set $f|_{V_0(p)}^{-1}(\delta^2)$.
As $0$ is an isolated singularity on $V_0(p)$, we find a small regular value $\delta^2$ of $f|_{V_0(p)}$ and we see that the link is a smooth manifold.

\begin{figure}[htp]
\def\svgwidth{0.35\textwidth}%
\begingroup\endlinechar=-1
\resizebox{0.35\textwidth}{!}{%
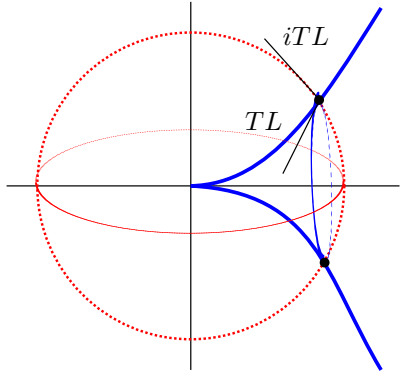%
}\endgroup
\caption{The cone represents the singular variety, and the intersection with the sphere is the link.
The associated contact structure is $\xi=T\cap iTL$.}
\label{fig:link}
\end{figure}
\subsubsection{Some basic definitions concerning contact manifolds}
By a {\bf contact structure} on an $2n-1$-dimensional manifold $Y$, we mean a field of hyperplanes $\xi^{2n-2} \subset TY$ that are maximally non-integrable.
We will always work with contact structures that are globally defined as the kernel of a smooth $1$-form $\alpha$, a so-called {\bf contact form}.
The condition that $\xi$ is maximally non-integrable is then equivalent to $\alpha\w d\alpha^{n-1}\neq 0$.
By a {\bf contactomorphism} between $(Y^{2n-1}_1,\xi_1)$ and $(Y^{2n-1}_2,\xi_2)$, we mean a diffeomorphism $\psi: Y_1 \to Y_2$ with the property $T\psi \xi_1 =\xi_2$.

A basic theorem tells us that deformations of contact structures on a compact manifold are always contactomorphic. More precisely.
\begin{theorem}[Gray stability theorem, \cite{Gr}]
Let $\xi_t$, $t \in [0,1]$ be a smooth family of contact structures on a closed manifold $Y$. Then there is an isotopy $\psi_t$ of $Y$ such that
$$
T \psi_t(\xi_0) = \xi_t \;\; \text{for each} \;\; t \in [0,1].
$$
In particular, $(Y, \xi_0)$ is contactomorphic to $(Y, \xi_1)$.
\end{theorem}

\subsubsection{Contact structures on links}
The simplest way to define a contact structure on a link of a singularity is to observe that, away from $0$, $f|_{V_0(p)}$ defines a strictly plurisubharmonic function on the variety $V_0(p)$.
Here we call a function $f$ {\bf strictly plurisubharmonic} if $-d(df\circ i)(\cdot,i \cdot)$ defines a Riemannian metric on $V$.
The link is a regular level set of this function, namely $L_{0,\delta}(p)=f|_{V_0(p)}^{-1}(\delta^2)$, so this link carries a contact structure.
To see this, put $\beta=-df|_{V_0(p)}\circ i$.
Then $\alpha:=\beta|_{L_{0,\delta}(p)}$ is a contact form with contact structure $\xi=\ker \alpha$.
Put differently, the contact planes are the complex tangencies to the link  $L_{0,\delta}(p)$, so $\xi=TL_{0,\delta}(p) \cap iTL_{0,\delta}(p)$.

Links of holomorphic functions are examples of particularly nice contact manifolds.
They are symplectically fillable.
To make this precise, we will recall three different forms of fillings.
\begin{definition}
A {\bf Liouville domain}, or {\bf compact Liouville manifold}, is a compact, exact symplectic manifold $(W,\omega=d\lambda)$ with boundary and a globally defined Liouville vector field $X$, defined by $i_X\omega=\lambda$, such that $X$ points outward along the boundary of $W$.

A {\bf Weinstein domain} $(W,\omega=d\lambda)$ is a Liouville domain for which the corresponding Liouville vector field is gradient-like for some Morse function $f:W\to\R$.

A {\bf Stein domain} is a compact, complex manifold with a strictly plurisubharmonic function $f$ such that its boundary is a regular level set of $f$.
\end{definition}
By the above construction with plurisubharmonic functions the boundary of a Stein domain is a contact manifold. We call the resulting contact manifold Stein fillable.

In fact, all of the above domains carry a contact structure on the boundary: in the two remaining cases we simply insert the Liouville vector field in the symplectic form $\omega$ obtaining a Liouville form $\lambda$. The restriction of $\lambda$ to the boundary is a contact form.
We call the resulting contact manifolds Liouville or Weinstein fillable depending on the type of the domain.

One can show that a Stein domain is always a Weinstein domain. A Liouville domain is the weakest of these three forms of fillability, and by a theorem of Eliashberg one can deform a Weinstein domain into a Stein domain. The monograph \cite{CE:Stein_book} explains this in detail.

\begin{remark}
Links of holomorphic functions as defined above, are always Stein fillable. Indeed, consider the variety $V_\epsilon(p):=p^{-1}(\epsilon)$. If $\epsilon>0$ is sufficiently small, the perturbed link $\tilde L_{0,\delta}(p)=V_\epsilon(p)\cap S^{2n-1}_\delta$ carries a contact structure that is isomorphic to the one constructed earlier by Gray stability.
\end{remark}

\subsection{Weighted homogeneous polynomials}
A particularly nice class of links of singularities is formed by weighted homogeneous polynomials.
\begin{definition}
A function $p:\C^{n+1} \to \C$ is called {\bf weighted homogeneous} with weights $(w_0,\ldots,w_n,d)$, which are all assumed to be positive integers, if for all $\lambda\in \C$ we have $p(\lambda^{w_0}z_0,\ldots,\lambda^{w_n}z_n)=\lambda^d p(z_0,\ldots,z_n)$.
\end{definition}
Usually, the weights are normalized by requiring $\gcd_i w_i=1$.
Note that the singular variety $V_0(p)$ carries a $\C^*$-action:
\begin{eqnarray*}
V_0(p)\times \C^* & \longrightarrow & V_0(p) \\
(z_0,\ldots,z_n;\lambda) & \longrightarrow & (\lambda^{w_0}z_0,\ldots,\lambda^{w_n}z_n).
\end{eqnarray*}

Links of weighted homogeneous polynomials carry a natural contact structure by the above argument. The contact form provided by this argument is not the best one though, and the one provided by the following proposition provides more insight into the structure.
\begin{proposition}
\label{prop:BW_contact_form}
Assume that $p$ is a weighted homogeneous polynomial with an isolated singularity at $0$.
Then the manifold $L_{0,\delta}(p)$ carries the contact form
$$
\alpha_1=\frac{i}{2}\sum_{j=0}^n \frac{1}{w_j}(z_jd\bar z_j -\bar z_j d z_j).
$$
whose Reeb flow is periodic,
$$
Fl^R_t(z_0,\ldots,z_n)=(e^{iw_0t}z_0,\ldots,e^{iw_nt}z_n).
$$
Furthermore, the contact manifolds $(L_{0,\delta}(p),\ker \alpha)$ and  $(L_{0,\delta}(p),\ker \alpha_1)$ are contactomorphic.
\end{proposition}

\begin{proof}
The proof is essentially backwards, since we start by guessing the contact form and Reeb field, and reconstruct the Liouville vector field from there.
Once one obtains the Liouville vector field, the contact and Reeb conditions are easily checked.
Define a symplectic form on $\C^n$ by $\omega:=i\sum_j \frac{1}{w_j} dz_j \w d\bar z_j$.
A primitive is given by the $1$-form $\alpha=\frac{i}{2}\sum_j \frac{1}{w_j}(z_jd\bar z_j -\bar z_j d z_j)$.
We shall show that $\alpha$ restricts to a contact form on the link of $p$.
Let $R$ denote the vector field generating the $S^1$-action induced by the $\C^*$-action on $V_0(p)$, so
$$
R=i\sum_{j=0}^n \left( w_jz_j\frac{\partial}{\partial z_j}
-w_j\bar z_j\frac{\partial}{\partial \bar z_j}
\right)
.
$$
This vector field is going to be the Reeb vector field, but this is not clear at the moment.
We claim that $R$ is tangent to $L_{0,\delta}(p)$. To see this observe that $R$ lies in $\ker dp$ (and in the kernel $\ker d\bar p$) as well as in the kernel of $df$, where
$$
f=|z|^2=z\bar z.
$$ 
Now define the vector field $X$ by the linear equation $i_X\omega|_{V_0(p)}=\alpha|_{V_0(p)} $.
Since $d\alpha=\omega$, it follows that $X$ is Liouville on $V_0(p)$.
We claim that $X$ is transverse to $L_{0,\delta}(p)$. To see why, note first of all that $L_{0,\delta}(p)=f^{-1}(\delta^2)\cap V_0(p)$, so we only need to check that $X$ is transverse to level sets of $f$ on $V_0(p)$.
Since
$$
i_R \omega=-\frac{1}{2}\sum_j \left(  z_j d\bar z_j+\bar z_j d z_j \right) =-\frac{1}{2}df,
$$
we have $df(X)=-2\omega(R,X)$.
On the other hand, if $z\neq 0$, then $\omega(X,R)=\alpha(R)\neq 0$ as a short computation shows. It follows that $X$ is a Liouville vector field that is transverse to $L_{0,\delta}(p)$.
Hence $\alpha|_{V_0(p)}=i_X\omega|_{V_0(p)}$ is a contact form after restricting to a regular level set of $f$.
It is now also clear that $R$ is the Reeb vector field for $\alpha$, since $d\alpha(R,\ldots)|_{L_{0,\delta}(p)}=\omega(R,\ldots)|_{L_{0,\delta}(p)}=0$. On $L_{0,\delta}(p)$, we also have $\alpha(R)=1$. Finally, the flow of $R$ is given by the above formula.

We now come to the claim that the contact forms $\alpha$ and $\alpha_1$ have contactomorphic contact structures.
The main ingredient is the Gray stability theorem.
First of all, we may assume that $\frac{1}{w_j} \geq 2$ for all $j$ by rescaling $\alpha$. 
Then define a family of almost complex structures $J_s$, $s \in [0,1]$, on $\C^{n+1}$ by
$$
J_s : = \sum_{j=0}^n \left\{ c_j(s) \frac{\partial}{\partial y_j} \otimes dx_j - \frac{1}{c_j(s)} \frac{\partial}{\partial x_j} \otimes dy_j \right\}
$$ 
where $c_j(s)$'s are nonzero real numbers satisfying $c_j(0)=1$, $c_j(1) + \frac{1}{c_j(1)} = \frac{1}{w_j}$ for all $j$. Such choice of real number $c_j(1)$ can be made because of the assumption $\frac{1}{w_j} \geq 2$.
By a direct computation we see that the function $f(z) = |z|^2$ is strictly plurisubharmonic with respect to $J_s$ for all $s$ and hence  we obtain a family of contact forms $\tilde \alpha_s : = - df \circ J_s|_{L_{0,\delta}(p)}$ on $L_{0,\delta}(p)$ with $\tilde \alpha_0 = \alpha$ and $d \tilde \alpha_1 = d\alpha = \omega$ with $\omega$ above.

We claim that $\alpha_1$ is isotopic to $\tilde \alpha_1$ which completes the proof. Note that the equation $ \iota_{X_1} \ow|_{V_0(p)}  = \alpha_1|_{V_0(p)}$ defines a Liouville vector field $X_1$ on $V_0(p) \backslash \{0\}$ and it is positively transverse to $L_{0,\delta}(p)$, i.e., $X_1(f) >0$.
As $X$ is also a Liouville vector field for $\omega$ that is positively transverse to $L_{0,\delta}(p)$, their convex sum $X_t : = t X_1 + (1-t)X$, $t \in [0,1]$, gives a family of Liouville vector field positively transverse to $L_{0,\delta}(p)$.
The family of contact forms $\iota_{X_t}\omega$ gives the desired isotopy, so Gray stability gives the conclusion.
\end{proof}

The advantage of the above contact form is that the Reeb flow is periodic, so we get a locally free circle action on the link $L_{0,\delta}(p)$. If the defining polynomial $p$ is homogeneous, then the resulting action is free, and the quotient of $Q(p):=L_{0,\delta}(p)/S^1$  carries the structure of a symplectic manifold.
In general, the action has non-trivial, but finite isotropy and the quotient $Q(p):=L_{0,\delta}(p)/S^1$  will be a symplectic orbifold.
\begin{proposition}
Links of weighted homogeneous polynomials are prequantization bundles over symplectic orbifolds. If the defining polynomial is homogeneous, then $L_{0,\delta}(p)$ is a prequantization bundle over a symplectic manifold.
\end{proposition}
Let us say more precisely which symplectic orbifold or manifold we get.
The link $L_{0,\delta}(p)$ lies inside a sphere $S^{2n+1}$. The circle action induced by the Reeb flow on $L_{0,\delta}(p)$ extends to the sphere, and the quotient is a weighted projective space $\C P^n(w)$. The quotient of the link $L_{0,\delta}(p)$ is a hypersurface in weighted projective space, given by the zeroset of the polynomial $p$.

In particular, we see that the quotient is a K\"ahler orbifold.
On the other hand, we have an explicit nice model for the symplectization as well, namely by removing $0$ from the variety $V_0(p)$, so the symplectization is also a K\"ahler manifold.
Contact manifolds with this property are called {\bf Sasaki}, and form an area of research by themselves, \cite{BG:Sasakian_geometry}.
We will not pursue this line of thought, though.

\subsection{Brieskorn polynomials}
We now consider an even more special class of singularities, namely weighted homogeneous polynomials of the form
$$
p(z_0,\ldots,z_n)=\sum_j z_j^{a_j},
$$
where $a_0,\ldots,a_n$ are positive integers.
We will call these polynomials {\bf Brieskorn-Pham} or just {\bf Brieskorn} polynomials.
If the exponents satisfy $a_i>1$ for all $i$, then a Brieskorn polynomial has an isolated singularity at $0$, so its link is a contact manifold.
If one of the exponents $a_i$ of the Brieskorn polynomial equals $1$, then we are looking at the link of a smooth point, which is the standard contact sphere. 
\begin{definition}
The link of a Brieskorn polynomial is called a {\bf Brieskorn manifold}.
\end{definition}
Since we will also look the variety $p^{-1}(0)$ and deformations, it is useful to reserve the word {\bf Brieskorn variety} for sets of the form
$$
V_\epsilon(a_0,\ldots,a_n):=\{ (z_0,\ldots,z_n) ~|~\sum_j z_j^{a_j}=\epsilon \} ,
$$
where $\epsilon$ is a (small) deformation parameter.
For $\epsilon\neq 0$, this variety is non-singular. If all exponents $a_j>1$, and $\epsilon=0$, then we obtain the singular variety $V_0(a_0,\ldots,a_n)$ with an isolated singularity at $0$.

Instead of taking the usual contact form on a link induced by the plurisubharmonic function $|z|^2$, we will use Proposition~\ref{prop:BW_contact_form} to obtain a nicer contact form.
Indeed, Brieskorn polynomials are weighted homogeneous with weights $(\frac{\lcm_j a_j}{a_0},\ldots,\frac{\lcm_j a_j}{a_n},\lcm_j a_j)$.
For later convenience, we rescale the contact form from Proposition~\ref{prop:BW_contact_form} and always take the contact form
$$
\alpha=\frac{i}{2}\sum_j a_j(z_jd\bar z_j -\bar z_j d z_j).
$$
As mentioned before, the Reeb flow of this contact form is periodic, and this makes Brieskorn manifolds tractable.

To simplify notation, we will often write
$$
V_\epsilon(a):=V_\epsilon(a_0,\ldots,a_n), \quad
\text{and} \quad
\Sigma(a):=\Sigma(a_0,\ldots,a_n).
$$

\section{Topology of Brieskorn manifolds}
As a starter, we point out that Brieskorn manifolds in dimension $1$ are links in the sense of knot theory.
In fact, $\Sigma(a,b)$ is a torus link with $\gcd(a,b)$ components. If $a$ and $b$ are relatively prime, then $\Sigma(a,b)$ is an $(a,b)$ torus knot.
Some other Brieskorn manifolds can also be explicitly identified with well-known manifolds.
\begin{lemma}
\label{lemma:A1-singularity}
The Brieskorn manifold $\Sigma(2,\ldots,2)$ is contactomorphic to $(ST^*S^n,\lambda_{can}|_{ST^*S^n})$, and the smoothed Brieskorn variety $V_\epsilon(2,\ldots,2)$ is symplectomorphic to $(T^*S^n,d\lambda_{can})$.
\end{lemma}
\begin{proof}
A computation shows that the map
\[
\begin{split}
\Sigma(2,\ldots,2) & \longrightarrow (ST^*S^n\subset T^*\R^{n+1},\lambda_{can}|_{ST^*S^n} ) \\
q+ip & \longmapsto (|q|^{-1}q,|q|p)
\end{split}
\]
is a contactomorphism and that the map
\[
\begin{split}
V_1(2,\ldots,2) & \longrightarrow (T^*S^n\subset T^*\R^{n+1},d \lambda_{can} ) \\
q+ip & \longmapsto (|q|^{-1}q,|q| p)
\end{split}
\]
is a symplectomorphism.
\end{proof}

In the next section, we will dive into homology computations involving Brieskorn manifolds and varieties.
The essential point is that the computation of the homology groups can be reduced to a combinatorial problem. 
Furthermore, there are explicit formulas for this data.
These formulas can be rather complicated though.

A few words about the organization of the proofs. 
If the argument is short, we give an immediate proof.
Many results will not be directly used though.
We will omit the proofs of such results, and instead just refer to the literature.
This is indicated by a direct \qed.
The proof of the formula for the Betti numbers of Brieskorn manifolds is diverted to Section~\ref{sec:proofs_top}.

\subsection{Homology of Brieskorn varieties and manifolds}
We start with some facts, and give arguments using Lefschetz fibrations later in Section~\ref{sec:lefschetz}.
These arguments are not the most efficient, but they illustrate the geometric structure nicely.
\begin{proposition}[Milnor, Theorem 6.5, page 57 in \cite{Milnor:singular_points}] 
\label{prop:homotopy_type_smoothed_var}
The Brieskorn variety $V_{\epsilon}(a_0,\ldots,a_n)$ is homotopy equivalent to a wedge of $\prod_i (a_i-1)$ spheres.
\end{proposition}
The homotopy equivalence from the proposition can also be seen with Pham's description of Brieskorn varieties in terms of joins.
\begin{definition}
Suppose that $\{ A_i \}_{i=1}^N$ is a collection of topological spaces.
The {\bf join} of the $A_i$'s, denoted $A_1*\ldots *A_N$ is the quotient space
$$
\{ (a_1,t_1,\ldots, a_N,t_N)\in A_1\times [0,1]\times\ldots A_N\times [0,1]~|~\sum_{i=1}^N t_i=1 \} /\sim
$$
where $\prod_{i=1}^N(a_i,t_i)\sim \prod_{i=1}^N(a_i',t_i')$ if and only if $t_i=t_i'$ for all $i$ and $x_j=x_j'$ if $t_j>0$.
\end{definition}
The construction of Pham provides a nice basis of $H_n(V_{\epsilon}(a);\Z)$.
For this consider the action of $G(a)=\Z_{a_0}\oplus \ldots \oplus \Z_{a_n}$ on $V_{\epsilon}(a_0,\ldots,a_n)$ given by multiplication with $a_i$-th roots of unity.
We will denote these roots of unity by $\zeta_{a_i}$.
One can show, see \cite[Section 12]{HM}, that the set
$$
U(a)=\{ (z_0,\ldots,z_n)\in \C^{n+1} ~|~ \sum_j z_j^{a_j}=1,~z_j^{a_j}\in \R \text{ with }z_j^{a_j}\geq 0 \text{ for }j=0,\ldots,n \}
$$
is a deformation retract of $V_1(a)$.
Furthermore, there is a deformation that is compatible with the group action.
By rescaling, one can identify $U(a)$ with the join $\Z_{a_0}*\ldots *\Z_{a_n}$. In fact, we get a simplicial structure.
To see this, define an $n$-simplex as the set
$$
\Delta_n:=\{ (x_0,\ldots,x_n) \in [0,1]^{n+1}\subset \R^{n+1}~|~ \sum_j x_j=1\}
.
$$
Write $W:=\{ (y_0,\ldots,y_n)\in \R^{n+1}~|~\sum_j y_j=0 \}$, so $\Delta_n\times W\cong T^*\Delta_n$.
Then consider the map
\begin{equation}
\label{eq:param_Brieskorn}
\begin{split}
\psi:\Delta_n \times W & \longrightarrow V_1(a_0,\ldots,a_n) \\
(x;y) & \longmapsto ( (x_0+iy_0)^{\frac{1}{a_0} },\ldots,(x_n+iy_n)^{\frac{1}{a_n} })
\end{split}
\end{equation}
The image of $\Delta_n\times \{ 0 \}$ is precisely the set $e$, defined by
$$
e=\{ (z_0,\ldots,z_n)\in \C^{n+1} ~|~ \sum_j z_j^{a_j}=\epsilon,~z_j\in \R_{\geq 0} \text{ for }j=0,\ldots,n \}
,
$$
see Figure~\ref{fig:simplex}.
To get the other simplices of $U(a)$ we act with $G(a)$. This gives a basis of homology:
\begin{equation}
\label{eq:basis_homology}
e_{i_0 \ldots i_n}:=(\zeta_{a_0}^{i_0},\ldots,\zeta_{a_n}^{i_n})\cdot (1-\zeta_{a_0})\ldots (1-\zeta_{a_n})\cdot e \quad \text{for }0\leq i_j\leq a_j-2.
\end{equation}

For example, we obtain the zero section of $T^*S^n=V_1(2,\ldots,2)$ as
$$
\{ ( \pm x_0^{1/2},\ldots, \pm x_n^{1/2})~|~(x_0,\ldots,x_n)\in \Delta_n \}
$$
In general, the Lagrangian simplices $(\zeta_{a_0}^{i_0},\ldots,\zeta_{a_n}^{i_n})\cdot \psi(x,0)$ have matching boundaries, but several simplices can be glued along a single edge, and $U(a)$ is not a smooth Lagrangian as a result.
\begin{figure}[htp]
\def\svgwidth{0.65\textwidth}%
\begingroup\endlinechar=-1
\resizebox{0.65\textwidth}{!}{%
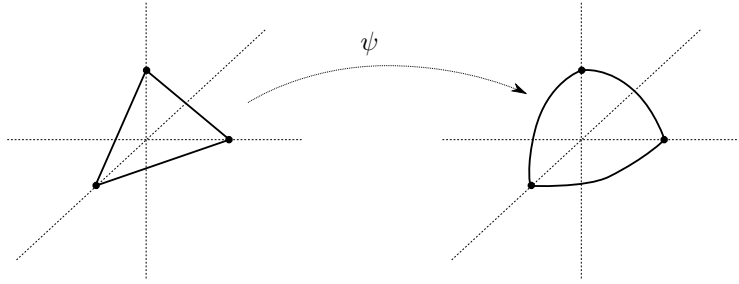%
}\endgroup
\caption{A Lagrangian $n$-simplex in a Brieskorn variety}
\label{fig:simplex}
\end{figure}

\begin{proposition}[Pham, \cite{Pham}]
\label{prop:intersection_form}
With respect to the basis \eqref{eq:basis_homology}, the intersection form $S_{V_{\epsilon}(a)}$ is given by
\[
e_{i_0 \ldots i_n} \cdot e_{j_0 \ldots j_n}=
\begin{cases}
(-1)^{n(n+1)/2}(1+(-1)^{n}) & \text{if }j_k=i_k \text{ for all }k\\
(-1)^{n(n+1)/2}(-1)^{\sum_k(j_k-i_k)} & \text {if } i_k\leq j_k\leq i_k+1 \text{ for all }k  \\
0 & \text{for those cases that cannot be obtained from the}\\
&   \text{previous one by symmetry.}
\end{cases}
\]

\end{proposition}
This result as well as many other related ones can also be found in \cite[Section 3]{GZ}.

\begin{figure}[htp]
\def\svgwidth{0.25\textwidth}%
\begingroup\endlinechar=-1
\resizebox{0.25\textwidth}{!}{%
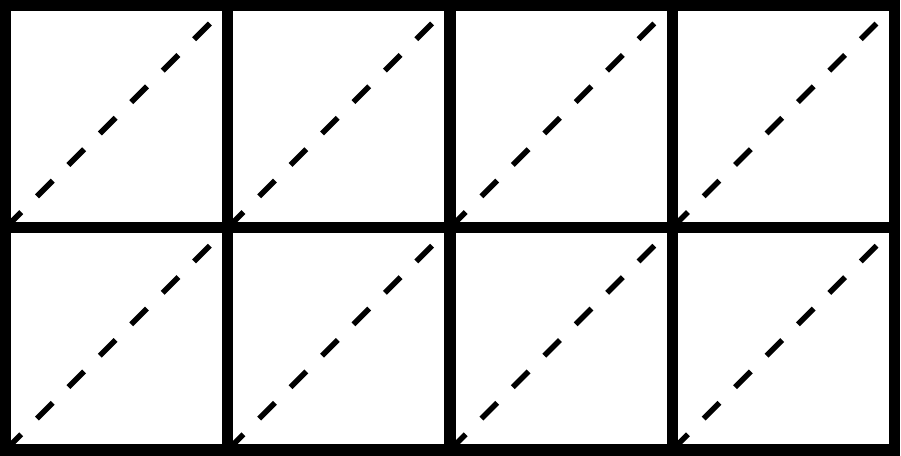%
}\endgroup
\caption{Intersection form for $V_\epsilon(a_0,a_1,2)$ with $a_0=6$ and $a_1=4$. Dashed line indicates intersection number $-1$: the non-dashed line indicates intersection number $+1$. Vertices all have self-intersection $-2$}
\label{fig:intersection_form}
\end{figure}

\begin{proposition}
A Brieskorn manifold $\Sigma(a)^{2n-1}$ is $(n-2)$-connected, so $\pi_i(\Sigma(a)^{2n-1}\,)=0$ for $i\leq n-2$.
\end{proposition}
To see this, we observe that $V_\epsilon(a)$ is a Stein manifold, which we can see as the completion of $V_\epsilon(a) \cap B_R$, where $B_R$ is a ball of some large radius $R$.
Since $V_\epsilon(a)$ has the homotopy type of an $n$-dimensional cell complex, we can obtain $V_\epsilon(a) \cap B_R$ from $\Sigma(a)$ by attaching handles of index at least $n$.
These handle attachments do not affect the homotopy groups $\pi_k$ for $k=0,\ldots,n-2$. Since $\pi_k(V_\epsilon(a)\,)$ vanishes for $k=0,\ldots,n-1$, the claim follows.

\subsection{Methods to compute homology}
Since smooth Brieskorn varieties are highly-connected manifolds, one can compute the homology of the ``boundary'' $\Sigma(a)$ by working out the cokernel of the intersection form of the filling $V_{\epsilon}(a)$.
If we denote the intersection form by $S_{V_{\epsilon}(a)}$, we find
$$
H_{n-1}(\Sigma(a);\Z)\cong \coker S_{V_{\epsilon}(a)}.
$$
However, it is more efficient to use the Milnor fibration structure (or open book decomposition in more contact geometric language).

A complete algorithm to get the above homology group was found by Randell, \cite{Randell:Brieskorn}. 
We will describe it below, but let us first look at the most interesting case, namely that in which $\Sigma(a)$ is homeomorphic to a sphere.

\subsection{Homotopy spheres}
Milnor and Brieskorn found the following appealing criterion to detect spheres.
Given a Brieskorn manifold $\Sigma(a)$, define a graph $\Gamma$ with $n+1$ vertices by the following procedure:
\begin{itemize}
\item label the $j$-th vertex by $a_j$.
\item draw an edge between the $i$-th and $j$-th vertex if $\gcd(a_i,a_j)>1$.
\end{itemize}

\begin{proposition}[Milnor, Brieskorn,\cite{HM}, page 100]
If $n\geq 3$, then $\Sigma(a)$ is homeomorphic to a sphere if and only if the graph $\Gamma$ satisfies one of the following two conditions.
\begin{itemize}
\item{} the graph $\Gamma$ has at least two isolated points.
\item{} the graph $\Gamma$ has an isolated point and a connected component $C$ with an odd number of points such that for $a_i,a_j \in C$ with $i\neq j$, one has $\gcd(a_i,a_j)=2$.
\end{itemize}
\end{proposition}
See \cite{HM}, page 100.
\begin{flushright}
\qed
\end{flushright}
\begin{figure}[htp]
\def\svgwidth{0.35\textwidth}%
\begingroup\endlinechar=-1
\resizebox{0.35\textwidth}{!}{%
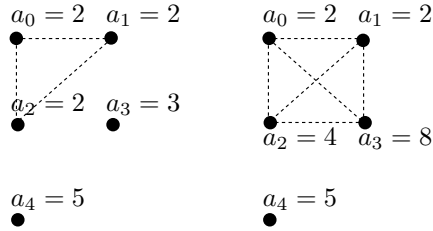%
}\endgroup
\caption{Brieskorn graph for a sphere $\Sigma(2,2,2,3,5)$ and a space that is not a sphere, $\Sigma(2,2,4,8,5)$}
\label{fig:graph_Brieskorn}
\end{figure}

As a special case, we see that $\Sigma(a)$ is homeomorphic to a sphere if all exponents are pairwise relatively prime.

\subsection{Randell's algorithm}
\label{sec:homology_Brieskorn}
To give general results for the homology of Brieskorn manifolds, we need some notation. 
The notation is meant to keep track of embeddings of Brieskorn manifolds into Brieskorn manifolds of higher dimension.
We will explore this more in Section~\ref{sec:decompositions}.

For $I:=\{0,1,\ldots,n\}$ and any subset $I_s$ with $s$ elements of the form $I_s=\{ i_1,\ldots, i_s\}$ define the Brieskorn submanifolds
\begin{equation}
\label{eq:K_notation}
\begin{split}
K(I_s):=&\Sigma(a_{i_1},\ldots,a_{i_s} ) \\
K(I):=&\Sigma(a_0,\ldots,a_n).
\end{split}
\end{equation}
Note that $K(I_s)$ has dimension $2s-3$.
Let $\kappa(I_s):=\rk H_{s-2}(K(I_s)\,)$.
The following formula can be found in \cite{Randell:Brieskorn}; we will present an argument following Milnor in Section~\ref{sec:homology_via_Milnor_fibration}.
\begin{equation}
\label{eq:rk_middle_dim_hom}
\kappa(I_s\,)=\sum_{I_t \subset I_s} (-1)^{s-t}\frac{\prod_{i\in I_t}a_i}{\lcm_{j\in I_t}a_j}.
\end{equation}
\subsubsection{Torsion in $H_{n-1}$}
Define
\[
k(I_s\,):=
\begin{cases}
\kappa(I_s\,) & \text{ if }n+1-s\text{ is odd} \\
0, &\text{otherwise.}
\end{cases}
\]
The function $C$ will send index subsets of $I$ to some integer. We prescribe the value on the empty set and use induction to get the other values:
\begin{equation}
\label{eq:Randell_torsion_C}
\begin{split}
C(\emptyset )&=\gcd_{i\in I} a_i \\
C(I_s)&=\frac{\gcd_{i\in I-I_s}a_i}{\prod_{I_t\subsetneq I_s} C(I_t)}.
\end{split}
\end{equation}
Define 
$$
d_j:=\prod_{\substack{I_s\\k(I_s\,)\geq j}} C(I_s) \quad \text{ and }\quad
r:=\max_{I_s\subset I} k(I_s\,).
$$

\begin{theorem}[Randell, \cite{Randell:Brieskorn}]
The homology group $H_{n-1}(\Sigma(a_0,\ldots,a_n);\Z )$ is isomorphic to
\begin{equation}
\label{eq:homology_Brieskorn}
H_{n-1}(\Sigma(a_0,\ldots,a_n);\Z )\cong
\Z^{\kappa(I)}\oplus \Z_{d_1}\oplus \ldots \oplus \Z_{d_r}.
\end{equation}
\end{theorem}
\begin{proof}
We will explain how to obtain the ranks of these groups in Section~\ref{sec:homology_via_Milnor_fibration}.
\end{proof}

\begin{remark}
Note that simply-connected, spin $5$-manifolds are classified by their second homology group $H_2(M)$, see \cite{Smale:5-mfd}.
Here a spin structure for $M$ consists of a principal fiber bundle $P_{Spin(n)}(M)\to M$ with structure group $Spin(n)$ together with an equivariant double covering map $p:P_{Spin(n)}(M) \to P_{SO(n)}(M)$, where $P_{SO(n)}(M)$ is the principal fiber bundle with structure group $SO(n)$ that is associated with $TM$.
A spin structure exists if and only if the second Stiefel-Whitney class $w_2(M)$ vanishes.

Since Brieskorn manifolds are spin as we shall see in the proof of Corollary~\ref{cor:spin5}, Brieskorn $5$-manifolds are determined up to diffeomorphism by Randell's algorithm.
\end{remark}

\begin{corollary}
\label{cor:spin5}
Every simply-connected spin $5$-manifold can be written as the connected sum of Brieskorn manifolds.
Furthermore, since connected sums can be defined in the contact category, every simply-connected spin $5$-manifold carries a Stein fillable contact structure.
\end{corollary}

\begin{proof}
By Smale's theorem, simply-connected spin $5$-manifolds admit a prime decomposition of the form
$$
M^5\cong \#_{m} S^2\times S^3 \# M_{q_1}\# \ldots \# M_{q_\ell}, 
$$
where the $q_i$ are powers of primes and we use the following conventions:
\begin{itemize}
\item the empty connected sum is $S^5$
\item the manifold $M_{k}$ is a spin manifold with $H_2(M_k;\Z)\cong \Z_k \oplus \Z_k$. We may think of $M_\infty$ as $S^2\times S^3$.
\end{itemize}
We will show that each of the $M_k$ have models in the form of Brieskorn manifolds.

First observe that Brieskorn manifolds are indeed spin.
Indeed, $w_2(\Sigma)=c_1(\xi_{\Sigma})\mod 2$, and we know that $c_1(\xi_\Sigma)=0$ since the Brieskorn variety is a hypersurface in $\C^{n+1}$.

Since we already know that $\Sigma(1,2,2,2)\cong S^5$ and $\Sigma(2,2,2,2)\cong ST^*S^3 \cong S^2\times S^3$, we only need to find Brieskorn models for those $5$-manifolds with pure torsion in homology.
We claim that $H_2(\Sigma(p,3,3,3);\Z)\cong \Z_p \oplus \Z_p$ for $p$ relatively prime to $3$ and $H_2(\Sigma(q,4,4,2);\Z)\cong \Z_q \oplus \Z_q$ for $q$ relatively prime to $2$. 
To check this, we can just apply Randell's formula \eqref{eq:homology_Brieskorn}.
We will do the case of $\Sigma(p,3,3,3)$.

The set $I=\{ 0,1,2,3 \}$.
Formula~\eqref{eq:Randell_torsion_C} gives
\[
\begin{split}
C(\emptyset)=1, 
\quad C(\{0\} )=3, \quad C(\{1\} )=C(\{2\} )=C(\{3\} )=1 \\
C(\{0,1,2\})=C(\{0,1,3\})=C(\{0,2,3\})=1,\quad C(\{1,2,3\})=p
\quad C(\{0,1,2,3\})=1.
\end{split}
\]

We conclude that
$$
r=\max_{I_s \subset I} k(I_s)=2.
$$
We also have
$$
d_1=\prod_{\substack{I_s\\k(I_s\,)\geq 1}}C(I_s)=C(\{ 1,2,3 \})=p,
\text{ and }
d_2=\prod_{\substack{I_s\\k(I_s\,)\geq 2}}C(I_s)=C(\{ 1,2,3 \})=p,
$$
so the claim follows.
As we now have realized all prime manifolds from Smale's theorem by Brieskorn manifolds, we obtain the statement of the corollary.
\end{proof}

\subsubsection{Equivariant homology}
Using the Gysin sequence, we can compute the equivariant homology $H^{S^1}_*(\Sigma(a_0,\ldots,a_n);\Q)$: we can think of this homology as being isomorphic to the rational homology of $\Sigma(a_0,\ldots,a_n)/S^1$ seen as a hypersurface in a weighted projective space.
Here is the upshot.
\begin{theorem}[Randell, \cite{Randell:Brieskorn}]
\label{thm:equivariant_homology_Brieskorn}
\[
\begin{split}
H^{S^1}_*(K(I)=\Sigma(a_0,\ldots,a_n);\Q)
\cong&
\left\{
\begin{array}{ll}
\Q, &\text{ if }*\text{ even; }0\leq * \leq \dim K(I) -1 \\
0, &\text{ otherwise}
\end{array}
\right\} \\
&
\oplus
\left\{
\begin{array}{ll}
\Q^{\kappa(I)}, &\text{ if }*=\frac{1}{2}(\dim K(I) -1) \\
0, &\text{ otherwise}
\end{array}
\right\}
.
\end{split}
\]
\end{theorem}
This result follows from Formula~\eqref{eq:rk_middle_dim_hom}, and the Gysin sequence with rational coefficients.
\begin{flushright}
\qed
\end{flushright}

\subsection{Standard spheres and exotic spheres}
With Milnor's criterion we have an effective method to see when a Brieskorn manifold is homeomorphic to a sphere.
Let us call such manifolds Brieskorn spheres.
Interestingly, it turns out that many Brieskorn spheres are not diffeomorphic to the standard sphere.
The technology to prove this, is completely classical by now, but the arguments are still rather lengthy, so we refer to the literature.
A particularly good reference is the book of Hirzebruch and Mayer, \cite{HM}.
First we describe some restrictions.
\begin{theorem}[Hirzebruch-Mayer, section 14.8]
Let $\Sigma(a)$ be a Brieskorn manifold that is homeomorphic to a sphere.
If $n=1,3,7$, then $\Sigma(a)$ is diffeomorphic to the standard sphere.
\end{theorem}
\begin{flushright}
\qed
\end{flushright}

\begin{theorem}[Hirzebruch-Mayer, Satz 14.4]
Smooth Brieskorn varieties are parallelizable.
\end{theorem}
\begin{flushright}
\qed
\end{flushright}
An immediate corollary is that Brieskorn spheres are always {\bf boundary parallelizable}, meaning that they are the boundary of parallelizable manifolds.
Therefore, Brieskorn spheres cannot, in general, generate the full group of exotic spheres.
Now let $bP_{n+1}$ denote the group of boundary parallelizable homotopy spheres of dimension $n$; its unit element is the standard sphere and the group operation is the connected sum.
We have the following result, which can be found in \cite{HM}, page 107-108.
\begin{theorem}
For $k>1$, the group $bP_{4k}$ is a cyclic group of order
$$
2^{2k-2}(2^{2k-1}-1)\cdot Numerator(\frac{4B_k}{k}),
$$
where $B_k$ is the $k$-th Bernoulli number, which we number by the convention $B_1=\frac{1}{6},B_2=\frac{1}{30}$, and so on.
Furthermore, a generator $g_{4k}$ of the group $bP_{4k}$ is given by  the boundary of the plumbing of copies of $T^*S^{2k}$ on the $E_8$-graph, see Figure~\ref{fig:plumbings}, and a Brieskorn manifold $\Sigma_a$ which is homeomorphic to a sphere represents the class
$$
\frac{\tau}{8}g_{4k},
$$
where $\tau$ denotes the signature of the intersection form of the Brieskorn variety $V_\epsilon(a)$.
\end{theorem}
\begin{flushright}
\qed
\end{flushright}

\begin{definition}
The {\bf Kervaire sphere} is the smooth $4k+1$-manifold obtained as the boundary of the plumbing of two cotangent bundles of spheres along fibers. 
\end{definition}
Put differently, the Kervaire sphere is the link of the so-called $A_2$-singularity in dimension $4k+2$, which is another way of saying that the Kervaire sphere is the Brieskorn manifold
$$
\Sigma(3,2,\ldots,2)=\{ (z_0,z_1,\ldots,z_{2k+1})~|~z_0^3+\sum_{j=1}^{k+1}z_j^2=0 \} \cap S^{4k+3}.
$$
The plumbing graph for this singularity is depicted in Figure~\ref{fig:plumbings} on the right.

\begin{figure}[htp]
\def\svgwidth{0.60\textwidth}%
\begingroup\endlinechar=-1
\resizebox{0.60\textwidth}{!}{%
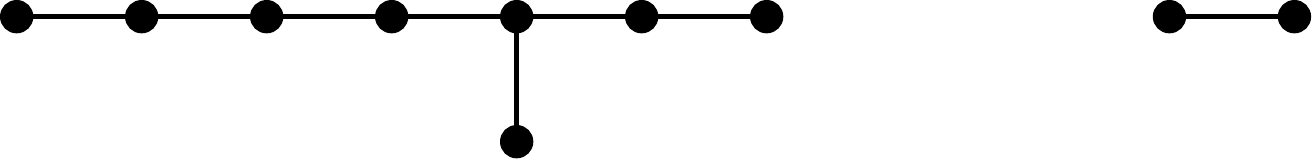%
}\endgroup
\caption{Plumbing graphs for $E_8$ and $A_2$}
\label{fig:plumbings}
\end{figure}

\begin{theorem}
\label{thm:std_and_kervaire}
Let $\Sigma(a_0,\ldots,a_n)$ be a Brieskorn manifold that is homeomorphic to a sphere, and assume that $n\neq 1,3,7$.
If $\det S_{V_{\epsilon}(a_0,\ldots,a_n,2)}=\pm 1 \mod 8$, then $\Sigma(a_0,\ldots,a_n)$ is diffeomorphic to the standard sphere.
If $\det S_{V_{\epsilon}(a_0,\ldots,a_n,2)}=\pm 3 \mod 8$, then $\Sigma(a_0,\ldots,a_n)$ is diffeomorphic to the Kervaire sphere (which may actually be diffeomorphic to the standard sphere depending on the dimension).
\end{theorem}
\begin{flushright}
\qed
\end{flushright}
There is a practical way to compute the determinant of the intersection form needed in the theorem using the Alexander polynomial, which we describe in more detail in Section~\ref{sec:homology_via_Milnor_fibration}.
The result is
$$
\det S_{V_{\epsilon}(a_0,\ldots,a_n,2)}=\pm \prod_{0<k_j<a_j}(1+\zeta_{a_0}^{k_0}\cdot \ldots \cdot \zeta_{a_n}^{k_n}).
$$

\subsection{Some proofs}
\label{sec:proofs_top}
The results listed above are all classical in the sense that they were obtained more than 40 years ago.
The proofs and tricks involved are still very nice to see, so we will review the proofs of those results that we use directly.
The basic references are Milnor, Dimca, Hirzebruch-Mayer and of course Pham and Brieskorn, \cite{Milnor:singular_points,Dimca:singularities_book,HM,Brieskorn:difftop,Pham}.
The literature on Brieskorn manifolds is vast, so alternative proofs may be found elsewhere.

\subsubsection{Homology via Milnor fibration}
\label{sec:homology_via_Milnor_fibration}
We essentially follow Milnor's proof from \cite[Chapter 9]{Milnor:singular_points} combined with some arguments from \cite{Milnor_Orlik:weighted}.
Interestingly, Milnor actually investigates periodic points under the Reeb flow, although he never calls it that way.
Anyway, to see the so-called Milnor fibration or open book structure, note that $\Sigma(a):=\Sigma(a_0,\ldots,a_n)$ embeds into $S^{2n+1}$, and in fact $S^{2n+1}-\Sigma(a)$ fibers over $S^1$: the  argument of the Brieskorn polynomial provides the projection to $S^1$.
Hence we have $pr:S^{2n+1}-\Sigma(a)\to S^1$, so by the clutching construction we get a ``monodromy map'', which we denote by $h$.
The name monodromy will be clarified in Remark~\ref{rem:monodromy}.
We will denote the fiber of $pr$ by $V$.

By Poincar\'e and Alexander duality, we have 
$$
H_{n-1}(\Sigma(a_0,\ldots,a_n)\, )\cong H^n(\Sigma(a_0,\ldots,a_n)\, )\cong H_n(S^{2n+1}-\Sigma(a_0,\ldots,a_n)\,).
$$
Abbreviate $E=S^{2n+1}-\Sigma(a_0,\ldots,a_n)$, and compute $H_n(E)$ with the Wang sequence,
$$
0 \longrightarrow H_{n+1}(E)
\longrightarrow H_n( V )
\stackrel{h_*-\id}{\longrightarrow} H_n( V )
\longrightarrow H_{n}(E)
\longrightarrow 0.
$$
We conclude
$$
H_{n-1}(\Sigma(a_0,\ldots,a_n);\Z)\cong \coker(h_*-\id).
$$

\subsubsection{The Alexander polynomial and a proof of formula~\eqref{eq:rk_middle_dim_hom}}
To compute the full homology we need to find $\coker(h_*|_{H_n}-\id)$, but since we shall only need the free part, it is enough to find $\dim \ker(h_*|_{H_n}-\id)$.
First define the characteristic polynomial
$$
\Delta(t):=\det( t\id-h_*|_{H_n}),
$$
which is the definition of the {\bf Alexander polynomial}.
To rework this Alexander polynomial, we make some observations.

First note that $V$, the fiber of $pr$, is diffeomorphic to the smoothed Brieskorn variety $V_\epsilon(a_0,\ldots,a_n)$. 
This last fact was proved by Milnor, \cite[Theorem 5.11]{Milnor:singular_points}, but one can also see it using the Lefschetz fibration we will describe in Section~\ref{sec:lefschetz}.
For the ``monodromy'' $h$ we have
\begin{lemma}
\label{lemma:observations}
The map $h$ is given by the time-$2\pi$ flow of the Reeb vector field,
\[
\begin{split}
h: V_\epsilon(a_0,\ldots,a_n) & \longrightarrow V_\epsilon(a_0,\ldots,a_n) \\
(z_0,\ldots,z_n) &\longmapsto (e^{i2\pi/a_0}z_0,\ldots,e^{i2\pi/a_n}z_n).
\end{split}
\]
This map is periodic with period $p:=\lcm_j a_j$.
Hence $h_*-t \id$ is diagonalizable over $\C$, and all zeroes of $\Delta$ lie on the unit circle.
Also $\Delta$ is symmetric, meaning that $\Delta(1/t)t^{\deg \Delta}=\pm \Delta(t)$.
\end{lemma}

Define the {\bf Lefschetz number} of the $\ell$-th iterate of $h$ by
$$
L(h^\ell):=
\sum_{k\geq 0} (-1)^k \tr( h^\ell_*|_{H_k(V_\epsilon(a);\Q)} )
.
$$
Here $h^\ell_*|_{H_k(V_\epsilon(a);\Q)}$ denotes the induced map on the $k$-th homology group of $V_\epsilon(a)$.
With the Lefschetz trace formula (or Lefschetz fixed point theorem) one can deduce $L(h^\ell)=\chi(Fix(h^\ell)\,)$, see \cite[Theorem 9.5]{Milnor:singular_points}.
We explain the idea in the case of an isometry $f:X\to X$. Consider a $\delta$ neighborhood of the fixed point locus of $f$, which we call $U$.
The restriction $f|_{U}:U\to U$ induces the identity map on the homology $H_*(U)\cong H_*(Fix(f)\,)$, so in this neighborhood we have $L(f|_{U})=\chi(Fix(f)\, )$.
Outside this neighborhood U the map $f$ has no fixed points.
The Lefschetz fixed point theorem states that any continuous map $f:X\to X$ with $L(f)\neq 0$ has fixed points, so we see that $L(f|_{X-U})=L(f|_{\partial U})=0$. By the Mayer-Vietoris sequence we conclude that $L(f)=L(f|_{U})=\chi(Fix(f)\,)$.

We will write $L(h^\ell)=\chi(Fix(h^\ell)\,)$ more briefly as $\chi_\ell$ from now on.
Note that the fixed point set of $h^\ell$ is a Brieskorn variety.
Indeed, if $e^{2\pi i \ell/a_j}z_j=z_j$, then either $z_j=0$ or $e^{2\pi i \ell/a_j}=1$.
From Proposition~\ref{prop:homotopy_type_smoothed_var} we see that the Euler characteristic $\chi_\ell$ satisfies
\begin{equation}
\label{eq:multiplicative_property}
1-\chi_\ell=\prod_{a_j | \ell}(1-a_j).
\end{equation}

Define integers $s_d$ by
\begin{equation}
\label{eq:chi_decomp}
\chi_i=\sum_{d|i}s_d.
\end{equation}
By the M\"obius inversion theorem, we can write
\[
s_i=\sum_{d|i}\mu(i)\chi_{i/d},
\]
where $\mu: \N \to \Z$ is the M\"obius function, which can be defined as follows.
If $k=p_1 \cdot \ldots \cdot p_r$ is a product of distinct primes, then $\mu(k)=(-1)^r$. If an integer $k$ contains a power of a prime, then $\mu(k)=0$. By convention $\mu(1)=1$.

It now follows from the above formula for $s_i$ that $s_i=0$ if $i$ does not divide the period $\lcm_j a_j$, and that $s_i$ is divisible by $i$, so we will write $s_i=ir_i$.

We will collect the Euler characteristics $\chi_i$ in a so-called Weil zeta function.
To manipulate some identities the following lemma is useful.
\begin{lemma}
Let $A\in Mat_{n\times n}(\C)$.
Then
$$
\det(\id -tA)=\exp\left(
\tr \log (\id-tA)
\right)
:=
\exp\left( 
-\sum_{i=1}^\infty\frac{\tr(A^i) t^i}{i}
\right)
$$
\end{lemma}
This lemma can be proved by using the Jordan normal form for $A$.
Define the {\bf Weil zeta function} by
$$
\zeta_h(t)=\prod_j \det(\id-th_*|_{H_j(V)})^{(-1)^{j+1}}.
$$
This is an unusual way to define this zeta function, but by applying the above lemma to each determinant in the product, we quickly obtain the standard definition as found in Milnor, \cite[Chapter 9]{Milnor:singular_points},
\[
\begin{split}
\zeta_h(t)&=\prod_j \exp\left( 
(-1)^j\sum_{i=1}^\infty \frac{\tr({h_*|_{H_j}}^i)t^i}{i}\right)=\exp\left(
\sum_{i=1}^\infty \frac{L(h^i)t^i}{i}
\right).
\end{split}
\]
Here $\chi_i=L(h^i)$ is the Lefschetz number of the iterate $h^i$.
Using the standard power series for the logarithm, this can be rewritten.
Indeed, the map $h$ is periodic with period $p=\lcm_j a_j$, so we can write
\[
\begin{split}
\sum_{j=1}^{\infty} \frac{\chi_j}{j} t^j
&=
\frac{\chi_1}{1} t^1+\frac{\chi_2}{2} t^2+\ldots +\frac{\chi_p}{p} t^p
+\frac{\chi_1}{p+1} t^{p+1}+\ldots\\
&=\frac{r_1}{1} t^1+\frac{r_1}{2} t^2+\frac{r_2}{1} t^2+\ldots +\frac{r_1}{p+1}t^{p+1}+\frac{r_1}{p+2}t^{p+2}+\frac{r_2}{(p+2)/2}t^{(p+2)/2}+\ldots \\
&=\sum_{d=1}^p \sum_{j=1}^{\infty} \frac{r_d}{j} ({t^d})^j =-\sum_{d=1}^p r_d \log(1-t^d).
\end{split}
\]
Thus we obtain the following lemma.
\begin{lemma}
The Weil zeta function satisfies
$$
\zeta_h(t)=\prod_{d|\lcm_j a_j} (1-t^d)^{-r_d}.
$$
Furthermore
\begin{equation}
\label{eq:Alexander_poly}
\Delta(t)=
\begin{cases}
\pm (t-1)^{-1}\prod_{d|p}(1-t^d)^{r_d} & n \text{ even}\\
\pm (t-1)\prod_{d|p}(1-t^d)^{-r_d} & n \text{ odd}.
\end{cases}
\end{equation}
\end{lemma}
For the last observation, we use the symmetry property of the Alexander polynomial mentioned in Lemma~\ref{lemma:observations}.

Since $h_*$ is periodic, $h_*$ is diagonalizable over $\C$.
This implies that $\dim \ker(\id -h_*|_{H_n})$  equals the multiplicity of the (generalized) eigenvalue $1$ appearing in the characteristic polynomial $\Delta$.
In other words $\kappa=\dim \ker(\id -h_*|_{H_n})$ is equal to the maximal degree for which $t-1$ divides $\Delta$.
From Formula~\eqref{eq:Alexander_poly} for $\Delta$, we directly conclude that
$$
\kappa=\sum_{d|\lcm_i a_i} r_d -1.
$$
Together with the M\"obius inversion formula and the formula for the Milnor number, this is enough information to obtain an explicit formula for $\kappa$.
To obtain the nicer formula~\eqref{eq:rk_middle_dim_hom}, we need to work a little more.

For $j\in \Z_{>0}$, define the element $E_j$ in the group ring $\Q[\C^{*}]$ by
$$
E_j=\frac{1}{j}\left( [1]+[\zeta_j]+\ldots +[\zeta_j^{j-1}]\right) ,
$$
where $\zeta_j=e^{2\pi i/j}$ is a $j$-th root of unity.
In the group ring $\Q[\C^{*}]$ the following identity for these $E_j$'s holds,
$$
E_k E_\ell =E_{\lcm(k,\ell)}.
$$
Define the function $\delta$ by
$$
\delta(a_0,\ldots,a_m)=(-1)^{m+1}\left( 
1-\sum_{j\geq 1} s_j(a_0,\ldots,a_m)E_j
\right)
.
$$
Note that the coefficient of $[1]$ in this expression for $\delta$ equals the multiplicity $\kappa$ we are interested in, up to sign.
The following lemma from \cite{Milnor_Orlik:weighted} shows that $\delta$ has nice multiplicative properties.
\begin{lemma}[Milnor-Orlik]
In $\Q[\C^*]$ the following identity holds true.
\[
\delta(a_0,\ldots,a_m) \cdot \delta(a_0',\ldots,a_{m'}')
=
\delta(a_0,\ldots,a_m,a_0',\ldots,a_{m'}')
\]
\end{lemma}
\begin{proof}
To abbreviate the notation, we write $\chi_j=\chi_j(a_0,\ldots,a_m)$, $\chi_j'=\chi_j(a_0',\ldots,a_{m'}')$ and $\chi_j''=\chi_j(a_0,\ldots,a_m,a_0',\ldots,a_{m'}')$. The variables $s_j$, $s_j'$ and $s_j''$ are defined similarly.
With formula~\eqref{eq:multiplicative_property} for $\chi_j$, we find
$$
(1-\chi_j'')=(1-\chi_j)(1-\chi_j'),
$$
so $\chi_j''=\chi_j+\chi_j'-\chi_j\chi_j'$ holds. By substituting formula~\eqref{eq:chi_decomp} for the $\chi$'s, we find
$$
\sum_{i|j}s_i''=\sum_{i|j}s_i+\sum_{i|j}s_i'-\sum_{k|j}\sum_{\ell|j}s_k s_\ell '.
$$
By an induction argument this can be reworked into
$$
s_j''=s_j +s_j' -\sum_{\lcm ( a,b )=j} s_a s_b'.
$$
By multiplying with $E_j$ and using the above identities, we obtain
\begin{equation}
\label{eq:delta_expression_mult}
1-\sum_n s_n''E_n=(1-\sum_a s_a E_a)(1-\sum_b s_b' E_b),
\end{equation}
which is the identity we were looking for.
\end{proof}
By an induction on the number of terms on the right-hand side of Equation~\eqref{eq:delta_expression_mult}, we find
$$
\delta(a_0,\ldots,a_m)=\prod_i (a_i E_{a_i}-1)=\sum_{I_s \subset I} (-1)^{m+1-s} a_{i_1}\cdot \ldots \cdot a_{i_s}\cdot E_{\lcm_{i\in I_s}a_i}.
$$
The coefficient of $[1]$ on the right-hand side equals $\sum_{I_s \subset I} (-1)^{m+1-s} \frac{a_{i_1}\cdot \ldots \cdot a_{i_s}}{\lcm_{i\in I_s}a_i}$, and this is the multiplicity of $t-1$ in $\Delta$, so we have proved formula~\eqref{eq:rk_middle_dim_hom}.

\section{Covering tricks, open books and Lefschetz fibrations}
\label{sec:decompositions}

\subsection{Covering tricks and nesting of Brieskorn manifolds}
\label{sec:covering_tricks}
By definition, a Brieskorn manifold $\Sigma(a_0,\ldots,a_n)$ is a submanifold of $S^{2n+1}$. Besides the standard embedding, we can also take
\begin{eqnarray*}
\Sigma(a_0,\ldots,a_n) & \longrightarrow & \Sigma(a_0,\ldots,a_n,1)\cong S^{2n+1} \\
(z_0,\ldots,z_n) & \longmapsto & (z_0,\ldots,z_n,0)
\end{eqnarray*}
Similarly, we can embed $\Sigma(a_0,\ldots,a_n)$ as a Brieskorn submanifold in $\Sigma(a_0,\ldots,a_n,a_{n+1})$ by using the same type of map.
We hence consider the relation between $\Sigma(a_0,\ldots,a_n,1)$ and $\Sigma(a_0,\ldots,a_n,a_{n+1})$.
\begin{proposition}
The Brieskorn manifold $\Sigma(a_0,\ldots,a_n,a_{n+1})$ is an $a_{n+1}$-fold branched cover over $\Sigma(a_0,\ldots,a_n,1)$.
Furthermore, $\Sigma(a_0,\ldots,a_n,1)$ is diffeomorphic to a sphere.
\end{proposition}
Intuitively, this is almost obvious. We ``almost'' obtain a branched cover by considering
\[
\begin{split}
(z_0,\ldots,z_n,z_{n+1}) & \longmapsto (z_0,\ldots,z_n,z_{n+1}^{a_{n+1}}).
\end{split}
\]
Unfortunately, this map only sends the Brieskorn variety $V(a_0,\ldots,a_{n},a_{n+1})$ to the Brieskorn variety $V(a_0,\ldots,a_{n},1)$.
It does \emph{not} send the Brieskorn manifold $\Sigma(a_0,\ldots,a_{n},a_{n+1})$ to $\Sigma(a_0,\ldots,a_{n},1)$.
This can be corrected for, see \cite[Section 5]{Orlik:Brieskorn}, but we just need the existence of a branched cover.
We now see that Brieskorn manifolds are stacked in a tower of branched covers of the following form,
\[
\entrymodifiers={+!!<0pt,\fontdimen22\textfont2>}
\xymatrix@R=20pt@C=15pt{
& & \Sigma(a_0,a_1,a_2,a_3) \ar[r] \ar[d] & \Sigma(a_0,a_1,a_2,a_3,1)\cong S^7 \\
& \Sigma(a_0,a_1,a_2) \ar[r] \ar[d] & \Sigma(a_0,a_1,a_2,1)\cong S^5 & \\
\Sigma(a_0,a_1) \ar[r] & \Sigma(a_0,a_1,1)\cong S^3 & & \\
}
\]
The vertical arrows are branched covers, and the horizontal arrows give codimension $2$ embeddings.
Furthermore, these embeddings give bindings of open books as we shall see.

\subsection{Open books}
\label{sec:open_books}

We recall the notion of an open book.
\begin{definition}
  A \textbf{(concrete) open book} on a manifold $M$ is a pair $(B,\Theta)$, where
\begin{enumerate}
\item{} the manifold $B$ is a codimension $2$ submanifold of $M$ with trivial
  normal bundle, and
\item{} the map $\Theta:\, M-B\to S^1$ endows $M - B$ with the structure of a
  fiber bundle over $S^1$ such that $\Theta$ gives the angular
  coordinate of the $D^2$--factor on a neighborhood $B\times D^2$ of
  $B$.
\end{enumerate}
\end{definition}

The set $B$ is called the \textbf{binding} of the open book. A fiber
of $\Theta$ together with the binding is called a \textbf{page} of the
open book.
Alternatively, one can ask for a map $\tilde \Theta:M\to \C$ such that $B=\tilde \Theta^{-1}(0)$, and put $\Theta=\frac{\tilde \Theta}{|\tilde \Theta|}$.
Figure~\ref{fig:open_book} explains the name ``open book''.
\begin{figure}[htp]
\def\svgwidth{0.2\textwidth}%
\begingroup\endlinechar=-1
\resizebox{0.2\textwidth}{!}{%
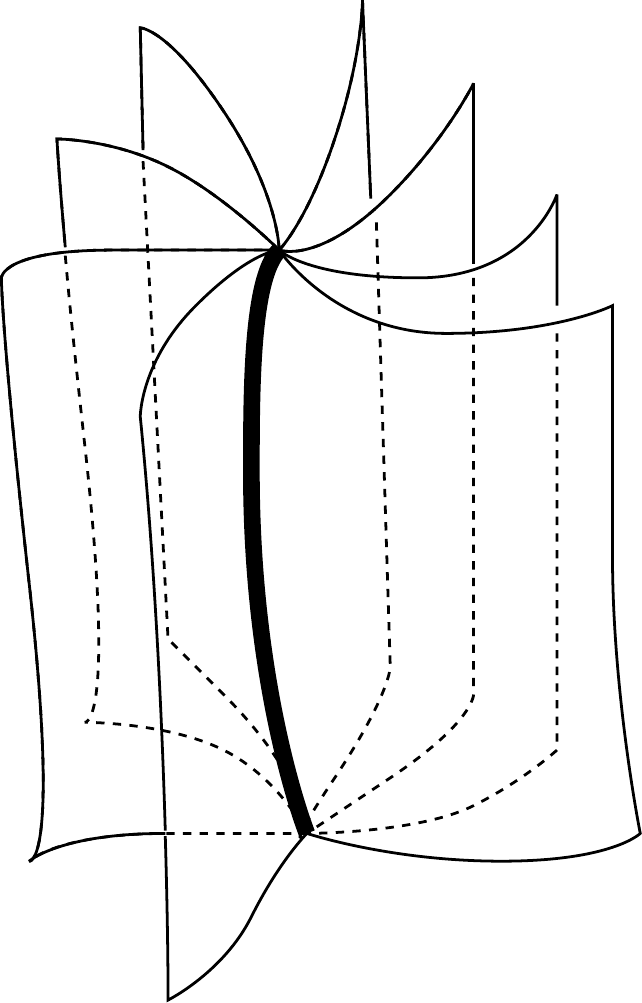%
}\endgroup
\caption{An open book with its binding}
\label{fig:open_book}
\end{figure}

Some historical remarks are in order.
The name ``open book'' was introduced by Winkelnkemper in \cite{Winkelnkemper:openbook}, although the notion had been around for much longer.
For instance, Alexander in \cite{Alexander} had shown that every compact, oriented $3$-manifold admits such an open book without using any special word for this notion.
Even earlier, Poincar\'e and Birkhoff \cite{Birkhoff:dynamical} had defined the notion of a global surface of section, a surface that is transverse to the flow away from the boundary and such that all flow lines return to the surface.

In singularity theory, the notion of open book comes under the name Milnor fibration, due to Milnor's work \cite{Milnor:singular_points}.
Independently, this notion was used
by several Japanese researchers under the name spinnable structure, starting with Tamura, \cite{Tamura:spinnable}; other words describing this notion include fibered knots.

By a result of Quinn, \cite{Quinn}, invoking the $s$-cobordism theorem it is now known that every odd-dimensional manifold admits an open book decomposition.

In contact topology the notion of open book was first used fruitfully by Thurston and Winkelnkemper in \cite{Thurston}. They used open books to show that every compact, orientable $3$-manifold admits a contact structure.
Later on, it was realized by Loi-Piergallini that Stein-fillable contact $3$-manifolds admit open books with a special monodromy, namely the product of right-handed Dehn twists, \cite{Loi:Stein}. We will explain the notions of  monodromy and Dehn twists below in more detail.
Giroux brought open books to the forefront in contact topology in his work \cite{Giroux:ICM2002}. See Theorem~\ref{thm:giroux_open_book} below.

One nice aspect of Brieskorn manifolds is that the above structure can be made very explicit.
Indeed, we obtain an open book by projecting to one of the coordinates.
\begin{lemma}
\label{lemma:Brieskorn_open_book}
Let $\Sigma(a_0,\ldots,a_n)$ be a Brieskorn manifold.
Then the map
\[
\begin{split}
\tilde \Theta:~(z_0,\ldots,z_n) & \longmapsto z_n
\end{split}
\]
defines a open book with binding $B=\{z_n=0\}\cong \Sigma(a_0,\ldots,a_{n-1})$ and page $V_\epsilon(a_0,\ldots,a_{n-1})$.
\end{lemma}

Given a concrete open book $(B,\Theta)$ on a manifold $M$, we can define a monodromy by the following procedure.
Take a Riemannian metric on $M-B$ such that, near $B$, the pages are orthogonal to the vector field $\partial_\theta$ on $B\times (D^2-\{ 0 \})$, where $(r,\theta)$ are polar coordinates on $D^2-\{ 0 \}$.
This uses condition~(ii) of the definition of a concrete open book.
Define the vertical space $Vert$ to be the tangent space to the fibers of the projection $\Theta: M-B \to S^1$.
Form a connection on $T(M-B)$ by declaring the horizontal space $Hor$ to be the orthogonal complement of $Vert$ with respect to the chosen metric.

Now take the path $t\mapsto e^{it}$ in $S^1$, and lift its tangent vector field to a horizontal vector field $X$ on $M-B$.
Near the binding, this horizontal lift $X$ is given by $\partial_\theta$.
The time $2\pi$-flow of $X$, denoted by $h$, defines therefore a diffeomorphism of the fiber $F=\Theta^{-1}(1)$ which is the identity near the binding.
We call this map $h$ the {\bf monodromy} of the open book, and this is well-defined up to conjugation and isotopy.

\begin{remark}
\label{rem:monodromy}
Contrast this with the earlier construction used in Section~\ref{sec:homology_via_Milnor_fibration} or with the following construction.
Consider $M-B$ as a fiber bundle over $S^1$ with fiber $F$.
Remove one fiber to get a fiber bundle over an interval, which is trivial.
Hence we obtain a gluing map, say $\tilde h$ by which we can recover $M-B=F\times I/(\tilde h(x),0)\sim(x,1)$ up diffeomorphism.

This gluing map is similar to the monodromy, but it doesn't fix the boundary.
From the point of view of contact geometry (but even from a topological viewpoint) a lot of information is lost in this construction.
We will discuss how to get the monodromy in the contact setting in the next section.
\end{remark}

Going back to the earlier construction of the monodromy, we get a pair $(F,h)$ which actually determines the concrete open book $(B,\Theta)$ on $M$.
This leads to the following definition.
\begin{definition}
An {\bf abstract open book} is a pair $(F,h)$, where $F$ is a smooth manifold with boundary, and $h:F\to F$ a diffeomorphism that is the identity in a neighborhood of the boundary of $F$.
\end{definition}
As claimed, we can reconstruct $M$ as
$$
M=B\times D^2 \cup_{\partial} F\times I/(h(x),t)\sim(x,t+2\pi)
,
$$
where $B=\partial F$.

\subsection{Contact open books and monodromy}
\label{sec:contact_open_book}
We begin by discussing some orientation issues.
Suppose $M$ is an oriented manifold with a concrete open book $(B,\Theta)$.
We will give $S^1$ its natural orientation as the boundary of the unit disk in the complex plane, and this gives each page $F$ an induced orientation by the convention that the orientation of $M-B$ as a bundle over $S^1$ matches the one coming from $M$; a different choice of orientation of $S^1$ will affect this orientation on the page and the other orientations discussed below. We will stick to the above natural choice.

In the following, we will have a symplectic form $\omega$ on each page.
If the orientation of a page $F$ induced by the open book structure matches the orientation as a symplectic manifold, we call the symplectic form $\omega$ on $F$ {\bf positive}.
We shall orient the binding $B$ as the boundary of a page $F$ using the outward normal.
If this orientation matches the one coming from a contact form $\alpha$, meaning the orientation from the volume form $\alpha\w d\alpha^n$, then we say that $\alpha$ induces a {\bf positive contact structure}.
We now come to Giroux' definition of a contact structure that is supported by an open book.
\begin{definition}
  A positive contact structure $\xi$ on an oriented manifold $M$ is
  said to be \textbf{carried by an open book} $(B, \Theta)$ if $\xi$
  admits a defining contact form $\alpha$ satisfying the following conditions.
  \begin{enumerate}
  \item{} the form $\alpha$ induces a positive contact structure on $B$, and
  \item{} the form $d\alpha$ induces a positive symplectic structure on each
    fiber of $\Theta$.
  \end{enumerate}
  A contact form $\alpha$ satisfying these conditions is said to be
  \textbf{adapted} to $(B, \Theta)$.
\end{definition}
Now let $(M,\xi)$ be a contact manifold carried by an open book $(B,\Theta)$ with adapted contact form $\alpha$.
Define a connection by splitting $T(M-B)$ into $Vert\oplus Hor$, where 
\begin{itemize}
\item the bundle $Vert=TF$, the tangent space to the fiber.
\item the bundle $Hor=\{ X\in T(M-B)~|~i_X d\alpha=0 \}$, so the horizontal space is spanned by the Reeb field.
\end{itemize}
We claim that this is a symplectic connection, meaning that the holonomy along a path is a symplectomorphism.
To see why, observe that with the Cartan formula $\mathcal L_X d\alpha=0$, so the flow of $X$ induces a symplectomorphism between fibers.

Now take the path $t\mapsto e^{it}$ in $S^1$, and compute the holonomy of this path by lifting the tangent vector $\partial_\theta$ to a horizontal vector field $X$ on $M-B$.
Denote the time $2\pi$-flow of this vector field by $\tilde \Psi$.
\begin{proposition}
The map $\tilde \Psi:\Theta^{-1}(1)\cong F\to \Theta^{-1}(1)$ is a symplectomorphism of $(F,d\alpha)$.
\end{proposition}

However, $\tilde \Psi$ is \emph{not} the monodromy as defined in the previous section, since it does not fix the binding.
Still, it is possible to associate a {\bf symplectic monodromy} with this open book: 
by removing a small neighborhood of the binding $\nu_M(B)$, we can arrange the situation such that $X$ is standard on the boundary of this neighborhood $\partial \nu_M(B)$.
By this we mean that $X=\partial_\theta$ when using polar coordinates for the $D^2$-factor of $B\times D^2$.
If we put $\Psi=Fl^X_{2\pi}$, then $\Psi$ is the identity on a neighborhood of the boundary of $F-\nu_F(B)$, and we have a monodromy in the sense we defined earlier, which is also symplectic.

\begin{figure}[htp]
\def\svgwidth{0.5\textwidth}%
\begingroup\endlinechar=-1
\resizebox{0.5\textwidth}{!}{%
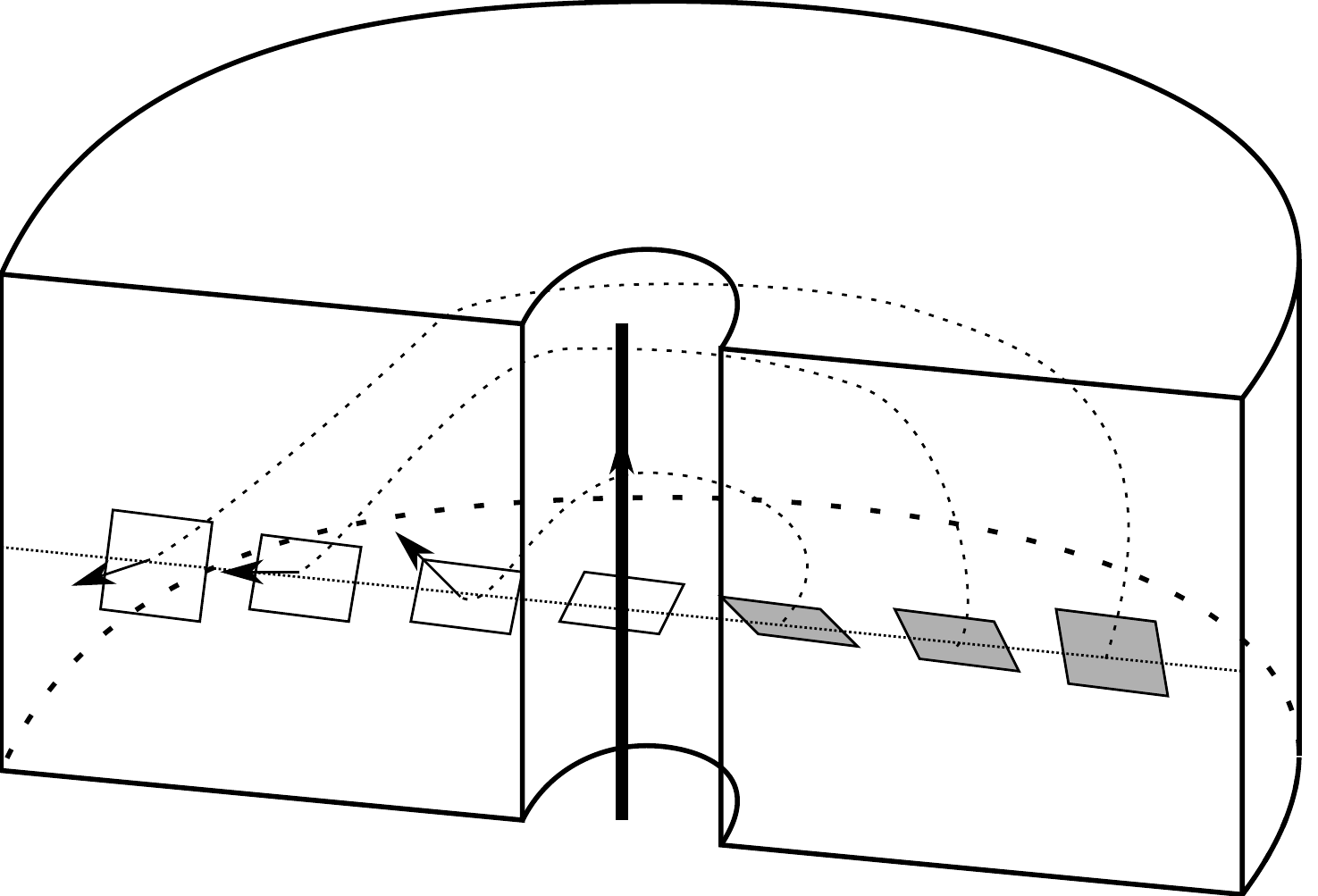%
}\endgroup
\caption{An open book carrying a contact structure: the Reeb flow of an adapted contact form is positively transverse to interior of the pages}
\label{fig:carrying}
\end{figure}

Giroux has shown that every cooriented contact manifold is carried by an open book, see his announcement and sketch in \cite{Giroux:ICM2002}.
\begin{theorem}[Giroux]
\label{thm:giroux_open_book}
Let $(Y,\xi)$ be a compact coorientable contact manifold.
Then there is an open book $(B,\Theta)$ carrying the contact structure whose pages carry the structure of a Weinstein domain.
\end{theorem}

There is also a construction providing a converse.
Given a Liouville domain $W$ (defined in Section~\ref{sec:blackbox_SH}) and a symplectomorphism $\psi$ that is the identity on the boundary of $W$, we can follow the construction of an abstract open book which was mentioned earlier.
In this case, the abstract open book can be given a contact structure, see for instance \cite{Giroux:ICM2002}.
We will call the resulting contact manifold with the underlying abstract open book a {\bf contact open book}, and write $\OB(W,\psi)$.

\subsection{Monodromies for Brieskorn manifolds}
\label{sec:monodromy_Brieskorn}
Using the covering trick from Section~\ref{sec:covering_tricks}, we can understand the monodromy of the contact open books from Lemma~\ref{lemma:Brieskorn_open_book} somewhat better.
If we denote the monodromy of such an open book by $\psi_{a_0,\ldots,a_{n-1},a_n}$, then we have the following corollary.
\begin{corollary}
The monodromy $\psi_{a_0,\ldots,a_{n-1},a_n}$ is symplectically isotopic relative to the boundary to $\psi_{a_0,\ldots,a_{n-1},1}^{a_n}$.
\end{corollary}
Indeed, this follows from the observation that a contact branched cover along the binding of an open book multiplies the monodromy.
Hence we can restrict ourselves to studying the monodromy $\psi_{a_0,\ldots,a_{n-1},1}^{a_n}$ of the Brieskorn manifold $\Sigma(a_0,\ldots,a_{n-1},1)$.
We remind the reader that $\Sigma(a_0,\ldots,a_{n-1},1)$ is contactomorphic to the standard sphere $(S^{2n-1},\xi_0)$.
Furthermore, note that $\tilde \Theta$ extends to a map
\[
\begin{split}
\tilde \Theta: V(a_0,\ldots,a_{n-1},1) & \longrightarrow \C \\
(z_0,\ldots,z_{n-1},z_n) & \longmapsto z_{n}=-\left( \sum_{j=0}^{n-1} z_j^{a_j} \right) .
\end{split}
\]
This map can deformed into a Lefschetz fibration by Morsifying the right hand side.
We also have a (concrete) open book for $\Sigma(a_0,\ldots,a_{n-1},1)$ with binding given by $z_n=0$, and the map to $S^1$, given by $\frac{z_n}{|z_n |}$.
This open book gives rise to an abstract open book $\OB(W,\psi)$.
The page $W$ of the open book is the Milnor fiber, which can be identified with the smooth affine variety
$$
V_\epsilon(a_1,\ldots,a_n)=\{ (z_1,\ldots,z_n)\in \C^n ~|~\sum_{i=1}^n z_i^{a_i}=\epsilon \}
.
$$
By a Morsification argument, one can see that $\psi$ can be written as the product of Dehn twists.
Since $\Sigma(a_0,\ldots,a_{n-1},a_n)$ is $a_n$-fold branched cover along the binding, we also obtain an open book for this Brieskorn manifold, namely
$\Sigma(a_0,\ldots,a_{n-1},a_n) \cong \OB(W,\psi^{a_n})$.

In the case of the Brieskorn manifolds $\Sigma(2,\ldots,2,N)$, the Morsification is much easier. Indeed, $\Sigma(2,\ldots,2,1)$ is the boundary of the Lefschetz fibration 
\[
\begin{split}
V_0(2,\ldots,2,1) & \longrightarrow \C \\
(z_0,z_1,\ldots,z_n) & \longmapsto z_n=-\sum_{i=0}^{n-1} z_i^2.
\end{split}
\]
This is the prototype of a Lefschetz fibration: the regular fiber of this fibration is the Brieskorn variety $V_\epsilon(2,\ldots,2)$ which is symplectomorphic to $(T^*S^{n-1},d\lambda_{can})$, and the monodromy is a right-handed Dehn twist $\tau$ along the zero section (one sometimes defines a right-handed Dehn twist as the monodromy of this particular fibration).
The observation that this monodromy is symplectic was first made by Arnold, \cite{Arnold:monodromy_A1}.
Detailed computations/hints to understand this monodromy are in \cite[Exercise 6.20]{McDuff_Salamon:introduction}.
 
Hence we get a contactomorphism $\Sigma(2,\ldots,2,1)\cong \OB(T^*S^{n-1},\tau)$.
As said earlier, a branched cover along the binding iterates the monodromy, so we conclude that 
\begin{proposition}
\label{prop:openbook_simple_brieskorn}
The Brieskorn manifold $\Sigma(2,\ldots,2,N)$ is supported by the contact open book $\OB(T^*S^{n-1},\tau^N)$.
\end{proposition}

\subsection{Lefschetz fibrations}
\label{sec:lefschetz}
Before we describe some Lefschetz fibrations on Brieskorn varieties, let us mention Seidel's monograph \cite{Seidel:Picard_Lefschetz}, which studies Lefschetz fibrations and their symplectic topology in detail.
We already saw the related notion of open books in Section~\ref{sec:decompositions}.
By virtue of being an affine variety, any Brieskorn variety admits the structure of a Lefschetz fibration over $\C$.
One way to understand this structure is to start with the variety $V_0(a_0,\ldots,a_{n-1},1)$, which is smooth since the last coordinate is a graph over the others.
In particular, $V_0(a_0,\ldots,a_{n-1},1)\cong \C^{n}$.
By projection, we get a singular fibration
\[
\begin{split}
pr: V_0(a_0,\ldots,a_{n-1},1) & \longrightarrow \C \\
(z_0,\ldots,z_{n-1},z_{n}) & \longmapsto z_{n}=-\sum_{j=0}^{n-1} z_j^{a_j}.
\end{split}
\]
The only singular value is $0$, and a regular fiber is diffeomorphic to $V_\epsilon(a_0,\ldots,a_{n-1})$.
This is not a Lefschetz fibration yet, but we can Morsify the singularity, and this will give us a Lefschetz fibration.

To do so, let $p_{i}$ be a generic polynomial of degree $a_i$ so that the $p'_{i}$ has $a_i-1$ distinct zeroes.
We can take $p_{i}(z)$ close to $z^{a_i}$.
An explicit Morsification of the Brieskorn polynomial $z_0^{a_0}+\ldots +z_{n-1}^{a_{n-1}}$ is then given by
$$
\tilde p=\sum_{i=0}^{n-1} p_{i}(z_i).
$$
This Morsification $\tilde p$ has only quadratic singularities and precisely $\prod_{i=0}^{n-1} (a_i-1)$ of them.
For the Morsified singularity, we obtain the Lefschetz fibration
\[
\begin{split}
\widetilde{pr}: \widetilde{V_0}(a_0,\ldots,a_{n-1},1) & \longrightarrow \C \\
(z_0,\ldots,z_{n-1},z_n) & \longmapsto z_n=-\sum_{j=0}^{n-1} p_{j}(z_j).
\end{split}
\]
Topologically, we can draw some conclusions from this fibration structure.
We find
$$
\widetilde{V_0}(a_0,\ldots,a_{n-1},1)\cong
\widetilde{V_0}(a_0,\ldots,a_{n-1}) \times D^2 \cup \coprod_{\prod_{i=0}^{n-1} (a_i-1)} n\text{-handles}.
$$
Indeed, from the Morsification argument, we see that the monodromy is the product of $\prod_{i=0}^{n-1} (a_i-1)$ Dehn twists since there are precisely that many quadratic singularities.
Each Dehn twist corresponds to critical handle attachment by \cite[Theorem 4.4]{vK:open_books_stab} (the reference formulates everything in terms of critical surgery on the boundary: this corresponds to critical handle attachment to the filling).
This gives the above claim.

We now also find
$$
\widetilde{V_0}(a_0,\ldots,a_{n-1},a_n)\cong
\widetilde{V_0}(a_0,\ldots,a_{n-1}) \times D^2\cup \coprod_{(a_0-1)\cdot\ldots \cdot(a_{n-1}-1)\cdot a_{n}} n\text{-handles}.
$$
We already know that $\widetilde{V_0}(a_0,\ldots,a_{n-1},1)\cong \C^{n}$, which is homotopy equivalent to a point. 
Attaching $(a_0-1)\cdot\ldots \cdot(a_{n-1}-1)\cdot(a_{n}-1)$ more $n$-handles gives us $\widetilde{V_0}(a_0,\ldots,a_{n-1},a_n)$, which is therefore homotopy equivalent to a bouquet of $(a_0-1)\cdot\ldots \cdot(a_{n-1}-1)\cdot(a_n-1)$ spheres.
We also see that these spheres can be made Lagrangian, so we can see part of the isotropic skeleton of Brieskorn varieties.

\section{Contact and symplectic invariants of Brieskorn manifolds}
In this section we compute some invariants of Brieskorn manifolds and their natural fillings.
We consider symplectic homology and equivariant symplectic homology, but we mainly focus on the equivariant theory since this will give us an invariant of contact manifolds via the mean Euler characteristic.

Whenever linearized contact homology for a contact manifold with a Liouville filling can be defined, the $+$-part of equivariant symplectic homology is isomorphic to linearized contact homology, see \cite{BO:SH_HC}.
The class of contact manifolds for which linearized contact homology works directly is very restricted, and hence we will mainly use symplectic homology.
For some Brieskorn manifolds contact homology can be defined with the results of Bourgeois and Oancea, \cite{BO:SH_HC}, but we do not go into this.

\subsection{Different versions of symplectic homology}
\label{sec:blackbox_SH}
We begin by describing symplectic homology by giving some of its properties: these are similar to some of the Eilenberg-Steenrod axioms, and this point of view has been worked out by Cieliebak-Oancea, \cite{Oancea:symplectic_algebraic_topology}. We won't need all of these axioms though, so we give a rather simple minded version.

($S^1$-equivariant) symplectic homology associates with a Liouville manifold $(W^{2n},\omega=d\lambda)$ graded $\Q$-vector spaces $SH_*^{(S^1)}(W,\omega)$ and $SH_*^{(S^1),+}(W,\omega)$ with the following properties.
\begin{enumerate}
\item There is a tautological exact sequence, the so-called Viterbo sequence,
$$
\longrightarrow
 SH^{(S^1),+}_{k+1}(W,\omega)
\longrightarrow
H_{n+k}^{(S^1)}(W,\partial W)
\longrightarrow
 SH^{(S^1)}_{k}(W,\omega)
\longrightarrow
 SH^{(S^1),+}_{k}(W,\omega)
 \longrightarrow
$$
\item the vector spaces $SH_*^{(S^1)}(W,\omega)$ and $SH_*^{(S^1),+}(W,\omega)$ are invariants of the symplectic deformation type under exact symplectomorphisms.
\end{enumerate}
In the above equivariant version of the Viterbo sequence, $S^1$ acts trivially on $(W,\partial W)$, so the homology groups $H_{*}^{S^1}(W,\partial W;\Z)$ can be computed using the K\"unneth formula and the homology of $BS^1\cong \C P^\infty$ as 
\begin{equation}
\label{eq:S1_equivariant_filling}
H_{*}^{S^1}(W,\partial W;\Z)
\cong
H_{*}(W,\partial W;\Z)\otimes \Z[u],
\end{equation}
where $u$ is a variable of degree $2$.

It is also known that symplectic homology does not change under subcritical surgery due to work of Cieliebak, \cite[Theorem 1.11]{Cieliebak}.
For the $+$-part of equivariant symplectic homology, the situation is more complicated.
We have the following Mayer-Vietoris style exact sequence from \cite[Theorem 4.4]{BO:SH_HC}.
\begin{theorem}[Surgery exact sequence for the connected sum]
\label{thm:les_connected_sum}
Let $W$ be a (disconnected) Liouville domain of dimension $2n$ satisfying $c_1(W)=0$, and suppose that  $\tilde W$ is obtained from $W$ by $1$-handle attachment (so $c_1(\tilde W)=0$).
Denote the ``symplectic homology'' of the $1$-handle by
\[
SH_*^{+,S^1}(tube):=SH^{+,S^1}_{*+1}(D^{2n},\omega_0)
=
\begin{cases}
\Q & *=n,n+2,n+4,\ldots \\
0 & otherwise.
\end{cases}
\]
Then the following surgery exact triangle holds,
\[
\entrymodifiers={+!!<0pt,\fontdimen22\textfont2>}
\xymatrix@R=20pt@C=15pt{
SH_*^{+,S^1}(tube) \ar[rr] & & SH_*^{+,S^1}(\tilde W) \ar[ld] \\
& SH_*^{+,S^1}(W) \ar[lu]^{[+1]} &
}
\]
\end{theorem}

\subsubsection{A little on defining symplectic homology}
Let $(W,\omega=d\lambda)$ be a compact Liouville manifold.
For grading purposes, suppose that $c_1(W)=0$ and that $W$ is simply-connected.
Then one can define Floer homology for the action functional depending on a time dependent, smooth Hamiltonian $H:W\times S^1 \to \R$. 
The action functional is defined on the loopspace of $W$, which we can describe with Sobolev spaces
$$
\Lambda W:=W^{1,p}(S^1,W)=\{ x:S^1=\R/\Z \to W~|~x~\text{ is of Sobolev class }W^{1,p}  \}
$$
In this note our convention for the Hamiltonian vector field $X_H$ is $i_{X_H}\omega=dH$, and the action functional is given by
\[
\begin{split}
\mathcal A_H: \Lambda W & \longrightarrow \R \\
\gamma & \longmapsto -\int_{S^1}\gamma^* \lambda-\int_{0}^1H(\gamma(t),t)dt
\end{split}
\]
The chain complex for Floer homology is a $\Z$-module generated by critical points of $\mathcal A_H$, which are given by $1$-periodic orbits of $H$.
We grade these $1$-periodic orbits by the Conley-Zehnder index using the following procedure.
By assumption, every $1$-periodic orbit $\gamma$ spans a disk $D$, along which we symplectically trivialize $(TW|_{D},\omega)$.
The linearized flow of $X_H$ with respect to this trivialization gives a path of symplectic matrices $\psi$.
The Conley-Zehnder index of a $1$-periodic orbit $\gamma$ of $X_H$ can then be defined with the crossing formula for the Maslov index, see \cite[Section 2, Section 5]{Robbin_Salamon:Maslov_path},
$$
\mu_{CZ}(\gamma):=-\mu_{CZ}(\psi),
$$
where the sign comes from the convention that the Hamiltonian vector field is a negative multiple of the Reeb vector field for an increasing Hamiltonian.
The differential for the Floer complex counts rigid solutions to the Floer equation, which can be thought of as the $L^2$-gradient of the action $\mathcal A_H$. In particular, the differential is action decreasing.
 
Symplectic homology can be defined as a Floer theory whose chain complex is generated by $1$-periodic orbits of a well-chosen Hamiltonian, or as the direct limit of the Floer homologies of a sequence of Hamiltonians.
For instance for cotangent bundles of compact smooth manifolds with their natural symplectic structure, the quadratic Hamiltonian $H=\frac{1}{2}\Vert p\Vert^2$ coming from a Riemannian metric can be used.
However, we use the latter approach, and we basically stick to the conventions of Bourgeois and Oancea, \cite{BO:SH_HC}.

Define the completion $\bar W:=W \cup_\partial \partial W\times [1,\infty[$.
Call a Hamiltonian $H:\bar W \times S^1 \to \R$ {\bf linear at infinity with slope $s$} if $H|_{\partial W\times [T,\infty[\times S^1}=s r+b$ for sufficiently large $T$. Here $r$ is the coordinate on the interval $[T,\infty[$.
The Hamiltonian vector field at infinity is given by $-sR$, where $R$ is the Reeb field of the contact form $\lambda|_{\partial W}$.
Assume that $s$ is not the period of any periodic Reeb orbit, so all $1$-periodic orbits of $X_H$ lie in a compact set.
Then Floer homology of $\bar W$ is defined.
It does, however, depend on the Hamiltonian.
To deal with this, we will consider a sequence of Hamiltonians.
If $H_1$ and $H_2$ are Hamiltonians with slope $s_1$ and $s_2$ such that $s_1\leq s_2$, then there is a {\bf continuation map}
\[
c_{12}: HF(\bar W,H_1,J_1) \longrightarrow  HF(\bar W,H_2,J_2)
\]
defined by counting suitable Floer trajectories.
One then defines symplectic homology of $\bar W$ as
$$
SH(\bar W):=\varinjlim_{N} HF(\bar W,H_N,J_N),
$$
where we take the direct limit over a sequence of Hamiltonians with increasing slope using the continuation maps.
\begin{remark}
We emphasize that, in general, the degrees in which symplectic homology is non-zero are not bounded from below or above.
\end{remark}
For a review of equivariant symplectic homology see Section~\ref{sec:review_equi_SH}.

\subsection{Morse-Bott spectral sequences}
In this section we will describe a Morse-Bott spectral sequence for (equivariant) symplectic homology.
A version of this spectral sequence with cohomology conventions appeared in \cite[Formula (3.2) and (8.9)]{Seidel:biased}.

Consider a simply-connected Liouville manifold $(W^{2n},d\lambda)$ with $c_1(W)=0$.
Suppose that $W$ has contact type boundary $\Sigma$ with a periodic Reeb flow.
We do not require all orbits to have the same period.
Define $\Sigma_T$ as the Morse-Bott submanifold in $\Sigma$ consisting of all periodic Reeb orbits of (not necessarily minimal) period $T$,
$$
\Sigma_T=\{ x\in \Sigma~|~Fl^R_T(x)=x \}
.
$$
\begin{remark}
\label{rem:MB_sub=Brieskorn_sub}
For every $T$ Morse-Bott submanifolds $\Sigma_T$ in a Brieskorn manifold $\Sigma(a)$ are Brieskorn manifolds themselves.
\end{remark} 
We can associate a Maslov index, often referred to as Robbin-Salamon index \cite{Robbin_Salamon:Maslov_path,Gutt:generalized_CZ}, with each connected component of a Morse-Bott submanifold $\Sigma_T$.
It turns out that for Brieskorn manifolds, even different connected components have the same index, so we will write $\mu(\Sigma_T)$ for this Maslov index.

\begin{theorem}
\label{thm:spectral_sequence_SH}
Let $(W,\omega=d\lambda)$ be a Liouville domain satisfying the assumptions.
\begin{enumerate}
\item The Reeb flow on $\partial W$ is periodic with minimal periods $T_1,\ldots, T_k$, where $T_k$ is the common period, i.e.~the period of a principal orbit.
We assume that all $T_i$ are integers.
\item The restriction of the tangent bundle to the symplectization to $\partial W$, $T(\R \times \partial W)|_{\partial W}$, is trivial as a symplectic vector bundle, $c_1(W)=0$ and $\pi_1(\partial W)=0$.
\item There is a compatible complex structure $J$ for $(\xi=\ker \lambda_{\partial W},d\lambda_{\partial W})$ such that for every periodic Reeb orbit $\gamma$ the linearized Reeb flow is complex linear with respect to some unitary trivialization of $(\xi,J,d\alpha)$ along $\gamma$.
\end{enumerate}
For each positive integer $p$ define $C(p)$ to be the set of Morse-Bott manifolds with (not necessarily minimal) return time $p$, and for each Morse-Bott manifold $\Sigma \in C(p)$ put
$$
shift(\Sigma)=\mu_{RS}(\Sigma)-\frac{1}{2}\dim \Sigma/S^1,
$$
where the Robbin-Salamon index is computed for a symplectic path defined on $[0,p]$.
Then there is a spectral sequence converging to $SH(W;R)$, whose $E^1$-page is given by
\begin{equation}
\label{eq:SS_SH}
E^1_{pq}(SH)=
\begin{cases}
\bigoplus_{\Sigma \in C(p) } H_{p+q-shift(\Sigma)}(\Sigma;R) & p>0
\\
H_{q+n}(W,\partial W;R) & p=0 \\
0 & p<0.
\end{cases}
\end{equation}
Furthermore, there is also a spectral sequence converging to $SH^{+,S^1}(W;R)$, whose $E^1$-page is given by
\begin{equation}
\label{eq:SS_SH_equivariant}
E^1_{pq}(SH^{+,S^1})=
\begin{cases}
\bigoplus_{\Sigma \in C(p) } H_{p+q-shift(\Sigma)}^{S^1}(\Sigma;R) & p>0 \\
0 & p\leq 0.
\end{cases}
\end{equation}
\end{theorem}
We give a proof in Appendix B: the idea is to perturb the autonomous Hamiltonians and filter the resulting complex by action. In a different setup, this idea was first used by Fukaya in \cite{F}.
For Brieskorn varieties we need to compute the indices to make more explicit statements and this will be done in the next section.

\begin{remark}
In the non-degenerate case (i.e.~all periodic Reeb orbits are non-degenerate) a spectral sequence was worked out by Gutt in his thesis~\cite{Gutt:thesis}, Theorem~2.2.2. He also used an action filtration and the $E^1$-page of his spectral sequence is generated by good periodic Reeb orbits (defined right after this remark).
One can also filter by index and this leads to yet another spectral sequence used by Bourgeois and Oancea in part II of their proof of Proposition 3.7 in \cite{BO:SH_HC}.
We briefly discuss this in Section~\ref{sec:index_filtration}) of Appendix~C.
\end{remark}

\begin{definition}
Call a non-degenerate periodic Reeb orbit $\gamma$ {\bf good} if $\gamma$ is not an even cover of $\tilde \gamma$ with $\mu_{CZ}(\gamma)-\mu_{CZ}(\tilde \gamma)$ odd.
{\bf Bad orbits} are periodic Reeb orbits that are not good.
\end{definition}

\subsection{Trivializations and Maslov indices}
\label{sec:trivializations}
Consider the Brieskorn manifold $\Sigma^{2n-1}(a)\subset S^{2n+1}\subset \C^{n+1}$.
Take the symplectic form $\omega_a:=i\sum_j a_j dz_j \w d\bar z_j$, and let $\xi=\ker \alpha_a$ denote the contact structure on $\Sigma(a)$.
We compute the ``transverse'' Robbin-Salamon index or Maslov index (i.e.~the ``Conley-Zehnder index'' for degenerate orbits) by extending the flow to $\C^{n+1}$ and then subtracting the normal part.

For the construction of the extension, we use the symplectic form $\omega_a$ to split the trivial bundle $\epsilon_\C^{n+1}|_{\Sigma(a)}$ as
$$
\epsilon_\C^{n+1}|_{\Sigma(a)}=\xi \oplus \xi^{\omega}.
$$
The Reeb flow of $\alpha_a$ is given by
$$
Fl^{R_a}_t(z_0,\ldots,z_n)=(e^{\frac{it}{a_0}}z_0,\ldots,e^{\frac{it}{a_n}}z_n),
$$
so this flow can clearly be extended to $\C^{n+1}$.
With the obvious unitary trivialization of $\C^{n+1}$, we obtain the linearized flow as the path of symplectic matrices given in complex coordinates by
$$
\Psi(t)=\diag(e^{\frac{it}{a_0}},\ldots,e^{\frac{it}{a_n}}).
$$
One can check, see \cite[Section 2]{vK:HC_Brieskorn}, that the symplectic complement of $\xi$ with respect to $\omega_a$ is trivial, and spanned by
\[
\begin{split}
\xi^{\omega_a}&=span
\left( 
X_1=\sum_j \bar z_j^{a_j-1}\partial_{z_j}+
z_j^{a_j-1}\partial_{\bar z_j}
, 
\quad
Y_1=i \sum_j \bar z_j^{a_j-1}\partial_{z_j}-
z_j^{a_j-1}\partial_{\bar z_j}
,
\right.
\\
&\phantom{=span
( ~\quad} 
\left.
X_2=
-i\sum_j \frac{z_j}{a_j} \partial_{z_j}
-
\frac{\bar z_j}{a_j} \partial_{\bar z_j}
,
\quad
Y_2=
\sum_j z_j \partial_{z_j}+
\bar z_j \partial_{\bar z_j} 
\right)
\end{split}
\]
This can be made into a \emph{symplectic trivialization} by the Gram-Schmidt process.
The linearized flow on the normal part $\xi^{\omega_a}$ ($4$-dimensional) with respect to this symplectic trivialization is given by
$$
\Psi_{\nu}(t)=\diag(e^{it},1 )
$$
We claim that the above trivializations extend over $V_\epsilon(a)$.
This is obvious for the trivialization that we used for the extended flow. For the normal part, we just use the above, explicit trivialization, whose vectors are linearly independent away from $0$. Since $0$ does not lie on the smoothed variety, the claim follows.
\begin{remark}
Although we compute Maslov indices with respect to capping disks in the filling, we want to point out that we can also compute indices in a consistent way in the contact manifold, even in dimension $3$.
Higher dimensional Brieskorn manifolds are simply-connected, and we can homotope disks in the filling to disks lying in the boundary: it makes therefore no difference to compute the indices in the boundary or in the filling.

In dimension $3$, Brieskorn manifolds are usually not simply-connected, but the Chern class of the contact structure of Brieskorn manifolds is trivial, so in dimension $3$ the contact structure, a complex line bundle, is itself trivial, and there is a global trivialization of the contact structure.
We can therefore compute the Maslov indices with respect to this global trivialization, even for non-contractible Reeb orbits.
\end{remark}

We also observe.
\begin{corollary}
There is a compatible complex structure $J$ and a unitary trivialization of $(\xi,J,d\alpha)$ for which the Reeb flow is complex linear.
Furthermore, the tangent bundle to the Brieskorn variety is symplectically trivial.
\end{corollary}
The first assertion is proved above.
We explain the second assertion in Section~\ref{sec:MB_SS}.
With this corollary, we can apply the Morse-Bott spectral sequence.

With the standard formula
\[
\mu(e^{it}|_{t\in[0,T]})=
\begin{cases}
\frac{T}{\pi} & \text{if }T\in 2\pi \Z \\
2 \lfloor \frac{T}{2\pi} \rfloor +1 & \text{otherwise,}
\end{cases}
\]
we find the following Maslov index for a periodic orbit $\gamma$ of period $T$ 
$$
\mu(\gamma|_{[0,T]})=
\sum_j \mu( e^{it/a_j}|_{t\in[0,T]} ) - \mu(\Psi_\nu(t)|_{t\in[0,T]} ).
$$

\subsubsection{Periods and orbit-types}
Although all Reeb orbits of a Brieskorn manifold $(\Sigma(a),\alpha_a)$ are periodic, the periods can vary, and this influences the Maslov indices.
To keep track of the periods, we introduce some notation in addition to that from Section~\ref{sec:homology_Brieskorn}.

First note that the period of a principal orbit is $2\pi \lcm_{j\in I} a_j$.
For all $T$ that divide $\lcm_{j\in I} a_j$, let $I_{T}$ denote the maximal subset of $I$ such that
$$
\lcm_{j\in I_{T}}a_j=T.
$$
If $\# I_{T}>1$, then $K(I_{T})\subset K(I)$ defines a Brieskorn submanifold $\Sigma(K(I_{T})\,)$, cf.~Remark~\ref{rem:MB_sub=Brieskorn_sub}.
This notation was defined in Formula~\eqref{eq:K_notation}.
The principal orbits of $\Sigma(K(I_{T})\,)$ have period $2\pi T$.

If the orbits with period $2\pi\cdot NT$ are not part of some larger orbit space $K(I_{T'})$, with $T|T'$, then the Maslov index of an $N$-fold cover of $K(I_{T})$ is given by
\begin{equation}
\label{eq_Maslov_index_Brieskorn}
\mu(N\cdot K(I_{T})\,)=2\sum_{j \in I_{T}} \frac{NT}{a_j}+
2\sum_{j \in I-I_{T}} \lfloor \frac{NT}{a_j} \rfloor+\#(I-I_{T})
-2NT.
\end{equation}
Note that the Brieskorn submanifolds $K(I_{T})$ consist of shorter orbits inside the bigger Brieskorn manifold $K(I)$.

\subsection{Index-positivity}
\label{sec:index_pos_neg}
Before we can state the next proposition, we need to introduce some notion measuring the growth rate of the Maslov-index.
Let $(\Sigma^{2n-1},\xi=\ker \alpha)$ be a cooriented contact manifold for which the Conley-Zehnder index or Maslov index of periodic Reeb orbits is well-defined.
This is for instance guaranteed if $c_1(\xi)=0$ and $\pi_1(\Sigma)=0$.

Suppose that $\gamma$ is a non-degenerate, periodic Reeb orbit. Then we have the following formula for iterates of $\gamma$ according to \cite[Lemma 13.4]{SZ},
$$
\mu_{CZ}(\gamma^N)=N\Delta(\gamma)+e(N).
$$
Here $\Delta(\gamma)$ is the {\bf mean index} and $e(N)$ an error term that is bounded in norm by $\dim_\C \xi=n-1$.
There is a similar formula for the Robbin-Salamon index of a degenerate Reeb orbit.

If $\Delta(\gamma)>0$ for all periodic Reeb orbits $\gamma$, then we call $(\Sigma^{2n-1},\alpha)$ {\bf index-positive}.
If $\Delta(\gamma)<0$ for all periodic Reeb orbits $\gamma$, then we call $(\Sigma^{2n-1},\alpha)$ {\bf index-negative}.
Contact manifolds that are neither index-positive nor index-negative will be referred to as {\bf index-indefinite}.
The same notions can also be defined for more general periodic orbits, such as orbits of Morse-Bott type.

Note that for a periodic flow, the mean index of a principal orbit equals the Maslov index of a principal orbit.
Now, look once more at Formula~\eqref{eq_Maslov_index_Brieskorn} to see that the following proposition holds.
\begin{proposition}
\label{prop:maslov_principal}
Let $\Sigma(a_0,\ldots,a_n)$ be a Brieskorn manifold with its natural contact form.
Then the Maslov index of a principal orbit is equal to
$$
\mu_P=2\lcm_{j\in I} a_j \, \left( \sum_{j=0}^n\frac{1}{a_j} \, -1 \right) .
$$
In particular, we see that $\Sigma(a_0,\ldots,a_n)$ is index-positive if $\sum_j \frac{1}{a_j}>1$, index-negative if $\sum_j \frac{1}{a_j}<1$ and index-indefinite if $\sum_j \frac{1}{a_j}=1$.
\end{proposition}

This proposition reflects an important geometric property of Brieskorn manifolds: if $\sum_j \frac{1}{a_j}>1$, then we can think of $\Sigma$ as a prequantization bundle over a Fano orbifold.
If $\sum_j \frac{1}{a_j}=1$, then $\Sigma$ is a prequantization bundle over a Calabi-Yau orbifold, and if $\sum_j \frac{1}{a_j}<1$, then $\Sigma$ is a prequantization bundle over an orbifold of general type.

\begin{remark}
As an entertaining corollary, we can immediately conclude that a Brieskorn manifold $\Sigma(a_0,\ldots,a_n)$ cannot be subcritically fillable if $\sum_j \frac{1}{a_j}\leq 1$.
Indeed, all generators for $SH^{+,S^1}_*$ have index bounded from above, and this is not true for subcritical Stein manifolds. 
\end{remark}

\subsection{Some explicit examples}
\label{sec:examples}
As an example, we work out the spectral sequence~\eqref{eq:SS_SH_equivariant} for the $+$-part of equivariant symplectic homology of the smoothed $A_{k-1}$-singularities, i.e.~the Brieskorn varieties $V_\epsilon(2,2,2,k)$.
We will work out the case $k>2$. The case $k=2$ is worked out in Section~\ref{sec:cotangent_sphere}, and the case $k=1$ is the standard contact sphere.

\begin{figure}
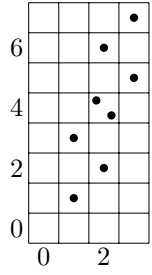

\begin{center}
\begin{sseq}{0...3}
{0...7}
\ssmoveto 1 1
\ssdropbull
\ssmove 0 2
\ssdropbull

\ssmoveto 2 2
\ssdropbull
\ssmove 0 2
\ssdropbull
\ssdropbull
\ssmove 0 2
\ssdropbull
\ssmoveto 3 5
\ssdropbull
\ssmove 0 2
\ssdropbull

\end{sseq}

\end{center}
\caption{$E^1$-page of Morse-Bott spectral sequence for $SH^{+,S^1}_*(V_\epsilon(4,2,2,2)\, )$: the short columns are given by $H^{S^1}_*(\Sigma(2,2,2);\Q)\cong H_*(S^2;\Q)$, and the long column is given by $H^{S^1}_*(\Sigma(4,2,2,2);\Q)\cong H_*(S^2\times S^2;\Q)$.}
\label{fig:SS_V4}
\end{figure}
There are two types of Morse-Bott submanifolds consisting of periodic Reeb orbits:
\begin{itemize}
\item Exceptional orbits. These form the Brieskorn submanifold $K( \{ 0,1,2 \})=\Sigma(2,2,2)\subset \Sigma(2,2,2,k)=K(\{ 0,1,2,3 \})$.
A simple cover of an exceptional orbit has period $2\pi \cdot 2$.
\item Principal orbits. These form the full Brieskorn manifold $K(\{ 0,1,2,3 \})=\Sigma(2,2,2,k)$.
A simple cover of a principal orbit has period $2\pi \cdot \lcm(2,k)$.
\end{itemize}

Now fill the $E^1$-page of the Morse-Bott spectral sequence \eqref{eq:SS_SH_equivariant} with the equivariant homology groups of these Brieskorn manifolds.

Using Formula~\eqref{eq_Maslov_index_Brieskorn}, we compute the degree shifts.
For an $N$-th cover of an exceptional orbit, denoted by $\gamma^{exc}_N$, such that $\lcm(2,k)$ does not divide $2N$, the degree shift is equal to
$$
shift(\gamma^{exc}_N)=\left( 2N+2 \lfloor \frac{N\cdot 2}{k} \rfloor +1 \right)-\frac{1}{2}(3-1)=2N+2 \lfloor \frac{N\cdot 2}{k} \rfloor.
$$
For an $N$-cover of principal orbit, denoted by $\gamma^{prin}_N$, the degree shift equals
$$
shift(\gamma^{prin}_N)=
\left( 
\lcm (2,k)+\frac{2\lcm(2,k)}{k}
\right)
N-2.
$$
Examples illustrating the case when $k$ is even and odd are given in Figures~\ref{fig:SS_V4} and \ref{fig:SS_V3}, respectively.
Since the total degree difference is always even, we conclude that the spectral sequence degenerates, and we can directly obtain the homology groups.
For $k=2m$ even we have 
$$
SH_\ell^{+,S^1}(V_\epsilon(2,2,2,2m)\,)\cong
\begin{cases}
\Q & \ell=2, \\ 
\Q^2 & \ell>2, \ell \text{ is even} \\
0 & \text{ otherwise.}
\end{cases}
$$
Note that these homology groups are independent of $m$.

For odd $k$, the $+$-part of equivariant symplectic homology is given by
\begin{equation}
\label{eq:SH_Ustilovsky}
SH_*^{+,S^1}(V_\epsilon(2,2,2,k)\,)\cong
\begin{cases}
0 & * \text{ is odd or } *<2, \\ 
\Q^2 & *=2 \lfloor \frac{2N}{k} \rfloor +2(N+1) \text{ with } N\in \Z_{N\geq 1} \text{ and }2N+1 \notin k\Z \\
\Q & \text{ otherwise.}
\end{cases}
\end{equation}
A typical page of the Morse-Bott spectral sequence looks like the one in Figure~\ref{fig:SS_V3}.
\begin{figure}[htp]
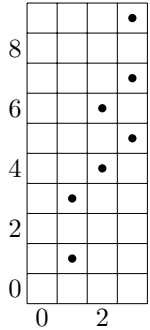

\begingroup\endlinechar=-1
{%
\begin{sseq}{0...3}
{0...9}
\ssmoveto 1 1
\ssdropbull
\ssmove 0 2
\ssdropbull

\ssmoveto 2 4
\ssdropbull
\ssmove 0 2
\ssdropbull

\ssmoveto 3 5
\ssdropbull
\ssmove 0 2
\ssdropbull
\ssmove 0 2
\ssdropbull

\end{sseq}

}\endgroup

\caption{$E^1$-page of Morse-Bott spectral sequence for $SH^{+,S^1}_*(V_\epsilon(3,2,2,2)\, )$: the short columns are given by $H^{S^1}_*(\Sigma(2,2,2);\Q)\cong H_*(S^2;\Q)$, and the long column is given by $H^{S^1}_*(\Sigma(3,2,2,2);\Q)\cong H_*(\C P^2;\Q)$.}
\label{fig:SS_V3}
\end{figure}

\begin{remark}
In the case of $A_{k-1}$-singularities, there are no orbits of contact homology degree $-1,0,1$, so cylindrical contact homology is (conjecturally) defined and gives an invariant. Alternatively, observe that all degree differences are even, so other exact fillings (with vanishing first Chern class) cannot give another $SH_*^{+,S^1}$; the differentials necessarily vanish.
This point of view has also been exploited in work of Gutt \cite{Gutt:thesis}, Theorem~2.2.2, where full details are given in the non-degenerate setup.
We also want to mention that Gutt obtained the same results on the Ustilovsky spheres independently with his spectral sequence, see \cite{Gutt:thesis,Gutt:invariant}, giving a rigorous proof of Ustilovsky's result.
Prior to Gutt's result, Fauck was able to distinguish Ustilovsky's spheres using Rabinowitz-Floer homology, \cite{Fauck:RFH}.
\end{remark}

\subsection{The $+$-part of equivariant symplectic homology of $ST^*S^n$}
\label{sec:cotangent_sphere}
As mentioned in Lemma~\ref{lemma:A1-singularity}, $(ST^*S^n,\lambda_{can})$ can be identified with the Brieskorn manifold $\Sigma(2,\ldots,2)$, so we can use the Morse-Bott spectral sequence and the formula in Theorem~\ref{thm:equivariant_homology_Brieskorn} for the homology of the quotient space $ST^*S^n/S^1$ to compute the $+$-part of equivariant symplectic homology.
We have the following explicit result for the Betti numbers of the $+$-part of equivariant symplectic homology of $(ST^*S^n,\lambda_{can})$.
\begin{proposition}
If $n$ is odd, then
\[
SH^{+,S^1}_k(T^*S^n,\lambda_{can})\cong
\begin{cases}
\Q^2 & \text{if }k=n-1+d(n-1) \text{ with }d\in \Z_{\geq 1}\\
0 & \text{if $k$ is odd or $k<n-1$}\\
\Q & \text{otherwise.}
\end{cases}
\]
If $n$ is even, then
\[
SH^{+,S^1}_k(T^*S^n,\lambda_{can})\cong
\begin{cases}
\Q^2 & \text{if }k=n-1+2d(n-1) \text{ with }d\in \Z_{\geq 1}\\
0 & \text{if $k$ is even or $k<n-1$} \\
\Q & \text{otherwise.}
\end{cases}
\]
\end{proposition}
\begin{proof}
We first do the case $n>2$.
From Formula~\eqref{eq_Maslov_index_Brieskorn} we find that the Maslov index of an $N$-fold cover of a principal orbit, which we denote by $\gamma_N$, is given by
$$
\mu(\gamma_N)=2(n-1)N.
$$
From Theorem~\ref{thm:equivariant_homology_Brieskorn} we find that
\[
H_k^{S^1}(ST^*S^n;\Q)\cong
\begin{cases}
\Q & \text{if }k\text{ is even, with } 0\leq k \leq 2n-2 \\
0 & \text{otherwise} 
\end{cases}
\oplus
\begin{cases}
\Q & \text{if }k=n-1 \text{ with }n-1 \text{ even} \\
0 & \text{otherwise}.
\end{cases}
\]
We put these homology groups in the Morse-Bott spectral sequence~\eqref{eq:SS_SH_equivariant}, which we display in Figure~\ref{fig:SS_STSn}.
Since the difference of the total degree of any two terms in this spectral sequence is even, nothing can kill or be killed, so we directly obtain the equivariant symplectic homology groups.
The result is
\[
\begin{split}
SH_k^{+,S^1}(T^*S^n,d\lambda_{can};\Q)&=
\bigoplus_{N=1}^\infty H^{S^1}_{k-\mu(\gamma_N)+\frac{1}{2}\dim(ST^*S^n/S^1) }
(ST^*S^n;\Q)\\
&=
\bigoplus_{N=1}^\infty H^{S^1}_{k-\mu(\gamma_N)+n-1 }
(ST^*S^n;\Q).
\end{split}
\]
\end{proof}

\begin{remark}
For an alternative argument, note that $ST^*S^n$ can be identified with a prequantization bundle over the complex quadric in projective space, or equivalently with a prequantization bundle over the Grassmannian of oriented $2$-planes, $Gr^+(2,n+1)$.

The homology and Chern class of the quadric can be computed with the method from \cite[Example 4.27]{McDuff_Salamon:introduction}.
The Chern class can then be used to compute the Maslov index and this gives the required information to work out the above spectral sequence.
\end{remark}

\begin{figure}
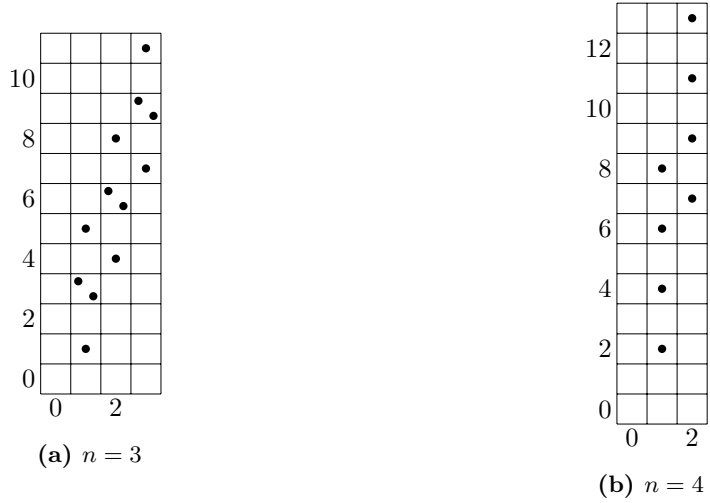

\centering

\begin{subfigure}{.5\textwidth}
  \centering
  \begin{sseq}{0...3}
{0...11}
\ssmoveto 1 1
\ssdropbull
\ssmove 0 2
\ssdropbull
\ssdropbull
\ssmove 0 2
\ssdropbull

\ssmoveto 2 4
\ssdropbull
\ssmove 0 2
\ssdropbull
\ssdropbull
\ssmove 0 2
\ssdropbull
\ssmoveto 3 7
\ssdropbull
\ssmove 0 2
\ssdropbull
\ssdropbull
\ssmove 0 2
\ssdropbull

\end{sseq}
  \caption{$n=3$}
  \label{fig:n=3}
\end{subfigure}%
\begin{subfigure}{.5\textwidth}
  \centering
  \begin{sseq}{0...2}
{0...13}
\ssmoveto 1 2
\ssdropbull
\ssmove 0 2
\ssdropbull
\ssmove 0 2
\ssdropbull
\ssmove 0 2
\ssdropbull

\ssmoveto 2 7
\ssdropbull
\ssmove 0 2
\ssdropbull
\ssmove 0 2
\ssdropbull
\ssmove 0 2
\ssdropbull

\end{sseq}
  \caption{$n=4$}
  \label{fig:n=4}
\end{subfigure}
\caption{$E^1$-pages of Morse-Bott spectral sequence for $ST^*S^n$}
\label{fig:SS_STSn}
\end{figure}

\subsection{Invariants of contact manifolds, and detecting exotic contact structures}
\label{sec:MEC}
So far we have discussed symplectic homology and equivariant symplectic homology.
These homology theories are invariants of symplectic manifolds with contact type boundary.
To obtain an invariant of the contact manifold we have following methods at our disposal.
\begin{itemize}
\item If the indices of all orbits are sufficiently high, one can use the methods from \cite{BO:SH_HC} to show that the $+$-part of (equivariant) symplectic homology is independent of the filling.

\item The mean Euler characteristic of equivariant symplectic homology will provide an invariant in many cases. 
This number does not depend on the choice of contact form because symplectic homology is a symplectic deformation invariant, see Gutt's thesis \cite{Gutt:thesis}.
On the other hand, this number does not depend on the choice of Liouville filling either.
The basic idea is that the Euler characteristic does not depend on the differential, and the only augmentation dependent part in the above story is the differential.
See Lemma~\ref{lemma:mec_invariant} and Proposition~\ref{prop:mean_euler_S^1-orbibundle} for precise statements.

\item If there is a chain complex (or spectral sequence) computing (equivariant) symplectic homology such that the differentials vanish for degree reasons, then $SH^+$ or $SH^{+,S^1}$ are invariants of the contact structure.
See Corollary~\ref{cor:invariance_SH+} for a precise sample statement of this kind.
\end{itemize}
To make the second point precise, we need some definitions.
Denote the Betti numbers of the $+$-part of equivariant symplectic homology of a symplectic manifold $(W,\omega)$ by $sb_i:=\rk SH^{+,S^1}_i(W)$.
We define the {\bf mean Euler characteristic} of $(W,\omega)$ as
$$
\chi_m(W) \,=\, 
\frac 12 
\left( 
\liminf_{N \to \infty}  \frac{1}{N} \sum_{i=-N}^N (-1)^i sb_i(W) \,+\,
\limsup_{N \to \infty}  \frac{1}{N} \sum_{i=-N}^N (-1)^i sb_i(W)
\right) 
$$
if this number exists.
If we have a uniform bound on the Betti numbers $sb_i$, then the limit inferior and the limit superior exist.
Furthermore, in all cases we will consider, the limit inferior and the limit superior exist and coincide, so the mean Euler characteristic reduces to
$$
\chi_m(W) \,=\, 
\lim_{N \to \infty}  \frac{1}{N} \sum_{i=-N}^N (-1)^i sb_i(W).
$$
We will use the following definition to get some nicely stated condition for the invariance of the mean Euler characteristic.
\begin{definition}
We say a contact structure $\xi$ on a simply-connected manifold $P^{2n-1}$ has {\bf convenient dynamics} if there exists a smooth family of contact forms $\{ \alpha_T=f_T \alpha\}_{T\in[T_0,\infty[}$ with $\xi=\ker \alpha_T$ and $T_0>0$, a positive constant $\Delta_m$ and a positive integer $k$ with the following properties:
\begin{enumerate}
\item all periodic orbits of $\alpha_T$ with period less than $T$ are non-degenerate, and satisfy the inequality for the mean index,
$$
{|\Delta(\gamma)|}>\Delta_m.
$$
\item there are simple periodic Reeb orbits of $\alpha_T$, denoted $\gamma_1^T,\ldots,\gamma_k^T$, with $\mathcal A_{\alpha_T}(\gamma_i^T)<T_0$ such that every periodic Reeb orbit of $\alpha_T$ with period at most $T$ is a cover of one of these $\gamma_i^T$.
\item the following inequality holds for all $p\in P$
$$
Tf_T(p)> (T-1) f_{T-1}(p).
$$
\end{enumerate}
\end{definition}
We will call the orbits $\gamma_1^T,\ldots,\gamma_k^T$ \emph{essential orbits}, and remark that these are a special case of the construction in \cite{BO:connected}.
In Remark~\ref{rem:why_convenient_dynamics} we motivate this definition.
Note that many contact manifolds do \emph{not} admit convenient dynamics.
However if the Reeb flow is periodic, the contact manifold often admits convenient dynamics, see Lemma~\ref{lemma:periodic_flow_convenient}.

\begin{lemma}[Mean Euler characteristic as an invariant]
\label{lemma:mec_invariant}
Suppose $(\Sigma,\alpha=\lambda|_{\partial W})$ is a compact, simply-connected contact manifold admitting a simply-connected Liouville filling $(W,d\lambda)$ with $c_1(W)=0$, so that grading in symplectic homology is well-defined.
Assume that the boundary $(\partial W,\ker \alpha )$ has convenient dynamics.
Then $\chi_m(W,d\lambda)$ is an invariant of the contact manifold $(\Sigma,\xi=\ker \alpha )$.
\end{lemma}
The idea for this statement was given before and we give a proof in Appendix C, Section~\ref{sec:AppendixC}.
\begin{remark}
For this reason, we will often write $\chi_m(\Sigma)$ instead of $\chi_m(W)$ to indicate that we actually have an invariant of the contact manifold rather than just of the symplectic filling.
\end{remark}

\begin{lemma}
\label{lemma:periodic_flow_convenient}
Suppose that $(P,\alpha)$ is a compact, simply-connected, cooriented contact manifold with the following properties.
\begin{itemize}
\item $c_1(\xi=\ker \alpha)=0$.
\item the Reeb flow of $\alpha$ is periodic and the mean index of a principal orbit is not equal to $0$.
\end{itemize}
Then $(P,\xi=\ker \alpha)$ has convenient dynamics. 
\end{lemma}

With following lemma we can construct more contact manifolds with convenient dynamics.
\begin{lemma}
\label{lemma:subcritical_preserves_convenient_dynamics}
Let $(W,d\lambda)$ be a Liouville domain with convenient dynamics on the boundary.
Suppose that the Liouville domain $(\tilde W,d\tilde \lambda )$ is obtained from $W$ by subcritical surgery along the isotropic sphere $S$ with framing $\epsilon$.
Then $(\tilde W,d\tilde \lambda )$ also admits convenient dynamics on the boundary.
\end{lemma}
See the appendix C, Section~\ref{sec:AppendixC}, for a proof.
The following sum formula for the mean Euler characteristic, originally found by Espina \cite[Corollary 5.7]{E}, will be very useful.
The current version is most easily found in \cite{BO:SH_HC}.
\begin{theorem}[Espina, Bourgeois-Oancea]
Let $(W_1^{2n},\omega_1)$ and $(W_2^{2n},\omega_2)$ be Liouville manifolds for which the mean Euler characteristic is defined, i.e.~the above limits exist.
Then the boundary connected sum of $W_1$ and $W_2$ satisfies
\begin{equation}
\label{eq:sum_formula}
\chi_m(\,(W_1,\omega_1) \natural (W_2,\omega_2) \,)
=
\chi_m(W_1,\omega_1)+\chi_m(W_2,\omega_2)
+(-1)^n\frac{1}{2}.
\end{equation}
\end{theorem}

\subsection{Mean Euler characteristic of Brieskorn manifolds}
\label{sec:MEC_Brieskorn}
We start with a general formula for the mean Euler characteristic of a Liouville fillable contact manifold $(\Sigma,\alpha)$ with a periodic Reeb flow.
Though explicit, the general formula is not so nice. 
Assume that the minimal periods of the Reeb flow of $\alpha$ are given by $T_1<\ldots<T_k$, where $T_k$ is the period of a principal orbit.
We follow \cite{FSvk:displaceability} and introduce the function
$$
\phi_{T_i;T_{i+1},\ldots,T_k} \,=\,
\# \{ a \in \N \mid aT_i < T_k \text{ and } a T_i \notin T_j \N \text{ for } j=i+1,\ldots, k \}
.
$$
We use the convention that $\phi_{T_k;\emptyset}=1$.
The Euler characteristic of the $S^1$-equivariant homology $H^{S^1}(\Sigma;\Q)$ will be denoted by $\chi^{S^1}(\Sigma)$.
This equivariant Euler characteristic equals the Euler characteristic of the quotient if the circle action has no fixed points.

\begin{proposition}[Not so nice formula for the mean Euler characteristic]
\label{prop:mean_euler_S^1-orbibundle}
Let $(\Sigma,\alpha)$ be a simply-connected contact manifold as above and assume that it admits a simply-connected Liouville filling $(W,d\lambda)$.
Suppose furthermore that the following conditions hold.
\begin{itemize}
\item The restriction of the tangent bundle to the symplectization to $\Sigma$, $T(\R \times \Sigma)|_{\Sigma}$, is trivial as a symplectic vector bundle, and $c_1(W)=0$.
\item There is a compatible complex structure $J$ for $(\xi=\ker \alpha,d\alpha)$ such that for every periodic Reeb orbit $\gamma$ the linearized Reeb flow is complex linear with respect to some unitary trivialization of $(\xi,J,d\alpha)$ along $\gamma$.
\end{itemize}
Let $\mu_P:=\mu(\Sigma)$ denote the Maslov index of a principal orbit of the Reeb action.
If $\mu_P\neq 0$ then the following hold.
\begin{itemize}
\item The contact manifold $(\Sigma,\alpha)$ is index-positive if $\mu_P>0$ and index-negative if $\mu_P<0$.
\item The mean Euler characteristic is an invariant of the contact structure and satisfies the following formula,
\begin{equation}
\label{eq:MEC_general}
\chi_m(W)=\frac{\sum_{i=1}^k (-1)^{\mu(\Sigma_{T_i})-\frac{1}{2}\dim (\Sigma_{T_i}/S^1) } \phi_{T_i;T_{i+1},\ldots T_k} \chi^{S^1}(\Sigma_{T_i})}{|\mu_P|}.
\end{equation}
\end{itemize}
\end{proposition}
The idea behind this proposition is to perform a signed count of the ranks of the entries, or ``the dots'', in the spectral sequence \eqref{eq:SS_SH_equivariant}.
The Euler characteristic of the equivariant homology of a Morse-Bott submanifold is equal to this signed number of dots in a single column.
Since this spectral sequence is periodic in the horizontal direction with shifts, we only need to count in one period, and divide by the absolute value of the mean index $\mu_P$.

We can apply this formula to Brieskorn manifolds.
The Maslov index $\mu_P$ of a principal orbit is computed in Proposition~\ref{prop:maslov_principal}, and Theorem~\ref{thm:equivariant_homology_Brieskorn} tells us how to compute $\chi^{S^1}(\Sigma_{T_i})$.
We illustrate the use of this proposition with an example that we will use later.

Let $p,q$ be odd integers that are relatively prime, and consider $\Sigma(2,2,p,q)$ with its natural contact form.
All assumptions of the proposition hold for this Brieskorn manifold, so we can start our computations.
Principal orbits in this Brieskorn manifold have then period $2pq$, and exceptional orbits can have period $2p$, $2q$, $2$ and $pq$ (and integer multiples).
We list all possible Morse-Bott submanifolds in the table below.
By frequency we mean the numbers $\phi_{T_i;T_{i+1},\ldots T_k}$.
\[
\begin{tabular}{llll}
Orbit space &  period & $\chi^{S^1}$ & frequency (in one period of $E^1$) \\
\hline
$\Sigma(2,2,p,q)$ & $2pq$ & 3 & 1 \\
$\Sigma(2,2,p)$ & $2p$ & 2 & $q-1$ \\
$\Sigma(2,2,q)$ & $2q$ & 2 & $p-1$ \\
$\Sigma(2,2)$ & $2$ & 2 & $pq-q-p+1$ \\
$\Sigma(p,q)$ & $pq$ & 1 & 1 \\
\end{tabular}
\]

We conclude that 
$$
\chi_m(V_\epsilon(2,2,p,q) \, )=\frac{1+2(pq-p-q+1)+2(q-1)+2(p-1)+3}{4(p+q)}
=\frac{1+pq}{2(p+q)}.
$$
We also want to mention another case where the numerics work out nicely.
If all exponents of the Brieskorn polynomial are pairwise relatively prime, then Formula~\eqref{eq:MEC_general} reduces to something more explicit, see \cite[Proposition 4.6]{FSvk:displaceability}.
\begin{proposition}
\label{prop:mean_euler_Brieskorn}
The Brieskorn manifold $\Sigma(a_0,\ldots,a_n)$ with its natural contact form~$\alpha$
is index-positive if $\sum_j \frac{1}{a_j} >1$, and index-negative if $\sum_j \frac{1}{a_j} <1$.
Furthermore, if the exponents $a_0,\ldots,a_n$ are pairwise relatively prime,
then the mean Euler characteristic of $\Sigma(a_0,\ldots,a_n)$ is given by
\begin{equation}
\label{eq:mean_euler}
\chi_m(\Sigma(a_0,\ldots,a_n),\alpha) \,=\,
(-1)^{n+1}\,
\frac{n+(n-1) \sum_{i_0}(a_{i_0}-1)+\ldots + 1 \cdot \sum_{i_0<\ldots< i_{n-2} }(a_{i_0}-1)\cdots (a_{i_{n-2}}-1)
}{
2 | (\sum_j a_0 \cdots \widehat{a_j}\cdots a_n)  -a_0\cdots a_n|
}
\end{equation}
\end{proposition}
Finally, for the $A_{k-1}$ singularities in dimension $5/6$, we can read off the mean Euler characteristic from the computations in Section~\ref{sec:examples},
\[
\chi_m( \Sigma(2,2,2,k)\,)=
\begin{cases}
1 & \text{if }k\text{ is even}\\
\frac{1}{2}\frac{2k+1}{k+2} & \text{if }k\text{ is odd}.
\end{cases}
\]
Note that Lemma~\ref{lemma:mec_invariant} applies here, so this gives a simple way to see that the Ustilovsky spheres, given by $\Sigma(2,2,2,k)$ with $k$ odd, are pairwise non-isomorphic.

\begin{theorem}
\label{thm:every_rational_number_is_MEC}
On the sphere $S^5$ every rational number can be realized as the mean Euler characteristic of some contact structure.
\end{theorem}
\begin{proof}
This is just a matter of finding the right generators, and using the sum formula \eqref{eq:sum_formula}.
We take the Brieskorn spheres $\Sigma(k,2,2,2)$ with $k$ odd (also known as Ustilovsky spheres), $\Sigma(2,2,p,q)$ and $\Sigma(p,q,r,s)$, where $p,q,r,s$ are pairwise relatively prime.
By Formula~\eqref{eq:MEC_general}, we find the following results (if $p,q,r,s$ are sufficiently large).
\[
\chi_m(\Sigma(k,2,2,2) \,)=\frac{1}{2}\frac{2k+1}{k+2},
\quad
\chi_m(\Sigma(p,q,r,s) \,)<\frac{1}{4},
\quad
\chi_m(\Sigma(2,2,p,q) \,)=\frac{1+pq}{2(p+q)}.
\]
We see that $\chi_m(\Sigma(2,2,3,5)\,)=1$, so by the sum formula, we find 
$$
\chi_m(\Sigma_1\# \Sigma_2 \#\Sigma(2,2,3,5)\, )=\chi_m(\Sigma_1)+\chi_m(\Sigma_2).
$$
Negative $\chi_m$ can be obtained by taking $\Sigma(p,q,r,s)\# \Sigma(p,q,r,s)$ if $p,q,r,s$ are sufficiently large to have $\chi_m(\Sigma(p,q,r,s) \,)<\frac{1}{4}$.
Hence it suffices to realize $1/p^\ell$ for any prime $p$. 
To get the prime powers $1/p^\ell$ for $p$ an odd prime, take $k=3\cdot p^\ell-2$.
Then
$$
\chi_m(\Sigma(k,2,2,2) 
\# \Sigma(k,2,2,2) 
\# \Sigma(2,2,3,5)
\,)=
\frac{2k+1}{k+2}=\frac{2\cdot 3\cdot p^\ell-4+1}{3\cdot p^\ell}=2-\frac{1}{p^\ell}.
$$
For the prime powers $1/2^\ell$, we observe that
$$
\chi_m(\Sigma(2,2,2^\ell-3,2^\ell+3)=\frac{2^{2\ell}-8}{2(2^{\ell+1}) }
=2^{\ell-2}-\frac{1}{2^{\ell-1}}
.
$$
\end{proof}

\begin{remark}
Since Brieskorn manifolds are prequantization orbibundles, the mean Euler characteristic is always a rational number.
As we have seen in Formula~\eqref{eq:sum_formula}, the connected sum preserves this property.
We do not know whether it is possible to get irrational numbers with more general constructions.
\end{remark}

\subsection{Exotic contact structures on a sphere form a monoid}
It is well-known that contact structures on $S^{2n-1}$ form a monoid under the connected sum operation.
The neutral element is the standard contact sphere $(S^{2n-1},\xi_0)$.
Let us call this monoid $\Xi(S^{2n-1})$.
Define the submonoid $\Xi_{nice}(S^{2n-1})$ consisting of contact structures on $S^{2n-1}$ that are convex fillable by simply-connected Liouville domains with vanishing first Chern class that admit convenient dynamics.

If we offset the mean Euler characteristic by a half, then we can reformulate Theorem~\ref{thm:every_rational_number_is_MEC} in more fancy language:
we have a surjective monoid homomorphism
\[
\begin{split}
\tilde \chi_m: (\Xi_{nice}(S^5),\#) & \longrightarrow (\Q,+) \\
\xi & \longmapsto \chi_m(S^5,\xi)-\frac{1}{2}.
\end{split}
\]
With the sum formula \eqref{eq:sum_formula} we see that this is indeed a homomorphism.
It follows that this monoid is infinitely generated.
The observation that the monoid $\Xi(S^{2n-1})$ is infinitely generated is not new, and was, in fact, already made by Ustilovsky in his thesis \cite{Ustilovsky:thesis} using cylindrical contact homology.
He looked at the chain complex of the connected sum, and although he didn't know the exact sequence from Theorem~\ref{thm:les_connected_sum}, he could deduce enough information to show that $\Xi_{nice}(S^5)$ is an infinitely generated monoid. See Section~\ref{sec:examples} for the necessary computations in the five-dimensional case: these can be easily extended to dimensions 9, 13, etc.

\subsection{The mean Euler characteristic and the $+$-part of equivariant symplectic homology over $\Q$ are not full invariants}
\label{sec:not_full_invariants}
Consider the Brieskorn manifolds $\Sigma(2k,2,2,2)$ from Section~\ref{sec:examples}.
We saw that the $+$-part of equivariant symplectic homology with $\Q$-coefficients does not depend on $k$.
Now consider the subcritically fillable manifold $(S^2\times S^3,\xi_{sub})=\partial( T^*S^2 \times \C,d\lambda_{can}+\omega_0)$.
To motivate this choice, observe that $\Sigma(2k,2,2,2)$ is the contact open book $\OB(T^*S^2,\tau^{2k})$, whose monodromy is a $2k$-fold Dehn twist.
On the other hand, $(S^2\times S^3,\xi_{sub})$ is given by the contact open book $\OB(T^*S^2,\id)$.

By examining the equivariant Viterbo long exact sequence we find that the $+$-part of equivariant symplectic homology of $(S^2\times S^3,\xi_{sub})$ is isomorphic to that of the $\Sigma(2k,2,2,2)$'s, so $SH^{+,S^1}$ cannot distinguish any of these manifolds.
We remind the reader that $(ST^*S^3,\xi_{can})\cong\Sigma(2,2,2,2)$.

Though we will not do this, one can show that $(S^2\times S^3,\xi_{sub})$ and $(ST^*S^3,\xi_{can})$ are not contactomorphic.
It is easier to show instead that $(S^2\times S^3,\xi_{sub})$ and $\Sigma(2k,2,2,2)$ are not contactomorphic for $k>1$.
A nice argument to see this can be found in \cite[Proposition 6.2(a)]{OV}. This proposition states that the link of an isolated singularity for which the intersection form of the smoothing is nonzero does not embed into a subcritical Stein manifold. 
Since this condition on the intersection form holds true for $\Sigma(2k,2,2,2)$ if $k>1$, the claim follows.

We reach the following conclusion.
\begin{proposition}
There are non-isomorphic, simply-connected Stein-fillable contact manifolds that are not distinguished by the $+$-part of equivariant symplectic homology with $\Q$-coefficients.
In other words, the $+$-part of equivariant symplectic homology (with $\Q$-coefficients) and their mean Euler characteristics are not full invariants.
\end{proposition}

\subsection{Classical invariants}
With modern technology around, classical invariants have become rather unpopular.
However, these invariants can still be very effective, especially in dimension 7, 11, and so on.
See \cite[Section~5.1]{Geiges:applications_contact_surgery} for a more detailed description of these invariants for Brieskorn manifolds and for some applications.
To explain the classical invariants, we start with some definitions.
\begin{definition}
An {\bf almost contact structure} on a $2n-1$-dimensional manifold is a reduction of the structure group from $\SO(2n-1)$ to $\U(n-1)\times \id$. (or equivalently to $\Sp(2(n-1)\,)\times \id$ ).
\end{definition}
Note that a cooriented contact structure $\xi=\ker \alpha$ on a manifold $\Sigma$ gives an almost contact structure.
Indeed, $(\xi,d\alpha)$ is a symplectic vector bundle and its complement is spanned by the Reeb field $R_\alpha$.
Since $T\Sigma\cong \xi \oplus \R R_\alpha$, we get an almost contact structure.
The homotopy class of almost contact structures provides an invariant and this is actually an invariant of the underlying contact structure.

Put differently, an almost contact structure is a lift of $f:\Sigma \to B\SO(2n-1)$, where $f$ is the classifying map for $T\Sigma$, to $\bar f:\Sigma \to B(\U(n-1)\times \id )$.
We note here that $B(\U(n-1) \times \id )$ is a fibration over $B\SO(2n-1)$ with fiber $\SO(2n-1)/\U(n-1)$.
The latter homogeneous space can be identified with $\SO(2n)/\U(n)$.
If $\Sigma$ is a sphere, then we see that homotopy classes of almost contact structures correspond to elements in $\pi_{2n-1}(\SO(2n)/\U(n)\,)$.
A classical result due to Massey, \cite{Massey:obstructions}, tells us that
\begin{equation}
\label{eq:Massey_homotopy_groups}
\pi_{2n-1}(\SO(2n)/\U(n)\, )
\cong
\begin{cases}
\Z\oplus \Z_2 & \text{for }n \mod 4=0 \\
\Z_{(n-1)!} & \text{for }n \mod 4=1 \\
\Z & \text{for }n \mod 4=2 \\
\Z_{\frac{(n-1)!}{2}} & \text{for }n \mod 4=3,
\end{cases}
\end{equation}
so then we have an idea how much information this classical invariant measures.

For Brieskorn manifolds that are diffeomorphic to standard spheres, Morita, \cite{Morita:Brieskorn} has worked out what value this invariant takes.
He worked in a different context, namely that of almost complex structures, but this can be translated to a contact setting.
If $ac$ denotes the map that sends an almost contact structure to the groups listed in \eqref{eq:Massey_homotopy_groups}, then Morita's computations can be stated as
\begin{equation}
\label{eq:almost_contact_Brieskorn_spheres}
ac( \Sigma(a),\ker \alpha_{a})
=
\begin{cases}
(\frac{n!}{4\cdot 2^n\cdot(2^{n-1}-1)B_{n/2}}\sig(a)-\frac{\prod_{i=0}^n(a_i-1) }{2},0) & \text{for }n \mod 4=0 \\
\frac{\prod_{i=0}^n(a_i-1) }{2} & \text{for }n \mod 4=1 \\
-\frac{n!}{4\cdot 2^n\cdot(2^{n-1}-1)B_{n/2}}\sig(a)-\frac{\prod_{i=0}^n(a_i-1) }{2} & \text{for }n \mod 4=2 \\
\frac{\prod_{i=0}^n(a_i-1) }{2} & \text{for }n \mod 4=3.
\end{cases}
\end{equation}
Here $B_n$ denotes the $n$-th Bernoulli number with the same conventions as before, and $\sig a$ is the signature of the manifold $V_\epsilon(a)\cup_\partial D^{2n}$, which is well-defined up to homeomorphism; this latter manifold can be defined by the assumption that $\Sigma(a)=\partial V_\epsilon(a)$ is diffeomorphic to a standard sphere.

The signature of $V_\epsilon(a)\cup_\partial D^{2n}$ can be computed with the help of the formula for the intersection form in Proposition~\ref{prop:intersection_form}.
This is rather complicated though, and fortunately there are more efficient ways, for example by counting lattice points satisfying certain conditions.
See \cite[Satz on page 98-99]{HM}.
\begin{remark}
Since this invariant is an element of a homotopy group, it behaves additively under connected sums involving only spheres.
\end{remark}

\section{Other applications: orderability and questions}
In \cite{eliashberg_polterovich:orderable} Eliashberg and Polterovich defined a relation $\succeq$ on $\widetilde{\Cont_0}(\Sigma,\xi=\ker \alpha)$, the universal cover of the identity component of the contactomorphism group of $(\Sigma,\xi)$ by the following rule.
For $\tilde f, \tilde g\in \widetilde{\Cont_0}(\Sigma,\xi)$, we say that $\tilde f \succeq \tilde g$ if $\tilde f \tilde g^{-1}$ is represented by a path that is generated by non-negative contact Hamiltonians.
The notion of orderability was defined in \cite{EKP}.
\begin{definition}[Eliashberg-Kim-Polterovich]
We call a contact manifold $(\Sigma,\xi=\ker \alpha)$ {\bf orderable} if $\succeq$ defines a partial order.
\end{definition}
Let us briefly point out that it has been shown that $(S^{2n-1},\alpha_0)$ is not orderable, whereas $(\R \P^{2n-1},\alpha_0)$ is orderable.
To understand non-orderable contact manifolds, the following proposition is helpful.
\begin{proposition}
A closed contact manifold $(\Sigma,\xi)$ is non-orderable if and only if there is a contractible loop $\phi:S^1 \to \Cont_0(\Sigma,\xi)$ with $\phi(0)=\id$ that is generated by a strictly positive contact Hamiltonian.
\end{proposition}

A recent result of Albers and Merry, \cite{AM:orderability}, allows us to show that many Brieskorn manifolds, including many exotic contact structures on spheres, are orderable.
For this, we need to compute the symplectic homology of a suitable Liouville filling.

\begin{theorem}
\label{thm:SH_isolated_sing}
Let $p\in \C[x_0,\ldots,x_n]$ be a polynomial with an isolated singularity at $0$ such that its Milnor number $\mu(p)$ is positive.
Denote the Liouville filling of the contact manifold $(L_{0,\delta}(p),\alpha_0)$ formed by the link of the singularity at $0$ by $V_{\tilde \delta}(p)=p^{-1}(\tilde \delta)\cap B_\delta(0)$, where $\delta$ and $\tilde \delta$ are sufficiently small with $0<\tilde \delta \ll \delta$.
Then $SH(V_{\tilde \delta}(p))\neq 0$.
Furthermore, if $p$ is a Brieskorn polynomial with positive Milnor number, then $(L_{0,\delta}(p),\alpha_0)$ is orderable.
\end{theorem}
\begin{proof}
Let $p$ be a polynomial with an isolated singularity at $0$, and let $L_{0,\delta(p)}$ denote its link, formed by $p^{-1}(0)\cap S^{2n-1}_\delta$.
Choose a Morsification $\tilde p$ satisfying the following:
\begin{itemize}
\item the critical points of $\tilde p$ are $z_1,\ldots,z_\mu$ with $z_1=0$.
\item on the ball with radius $\delta$, $B_\delta$, the polynomial $\tilde p$ is close to $p$.
\end{itemize}
Note that the function $z\mapsto \Vert z\Vert^2$ is a strictly plurisubharmonic function, and then observe that the link $L_{0,\epsilon} (\tilde p)$ is filled by a copy of disk cotangent bundle $D^*S^{n-1}$ if $\epsilon<\delta$ is sufficiently small by Lemma~\ref{lemma:A1-singularity}.
Now define $P:=\tilde p-\epsilon'$ such that
\begin{itemize}
\item $0$ is a regular value of $P$
\item $\epsilon'$ is so small that $L_{0,\epsilon}(P)$ is still filled by $D^*S^{n-1}$ 
\end{itemize}
The larger link $L_{0,\delta}(P)$ encloses the critical points $z_1,\ldots,z_\mu$ of $P$, and since $P$ is close to $p$, the contact structure on the links $L_{0,\epsilon}(p)$ and $L_{0,\epsilon}(P)$ are contactomorphic.
The Liouville domain $V(P):=P^{-1}(0)\cap B_\delta$ is hence a Liouville (even Stein) filling for the contact manifold $(L_{0,\delta}(p),\alpha_0)$.

Furthermore, the Liouville domain $V(P)$ contains a Lagrangian sphere $L$, so the symplectic homology of $V(P)$ is non-vanishing by a theorem of Viterbo, \cite{V2}.
We briefly sketch a proof of this theorem.
First note that a neighborhood of the exact Lagrangian submanifold $L$ forms a Liouville subdomain $D^*L\subset V(P)$.
This subdomain gives rise to the so-called transfer map, which is a ring homomorphism
$$
SH(V(P)\, ) \longrightarrow SH(D^*L)\neq 0.
$$
In particular, this transfer map sends the unit in $SH(V(P)\,)$ to the unit in $SH(D^*L)$.
Since the latter element does not vanish, the unit in $SH(V(P)\,)$ cannot be zero either.
Symplectic homology is a symplectic deformation invariant, \cite{Gutt:thesis}, so $SH(V_{\tilde \delta}(p)\,)$ is non-vanishing as well.

To obtain the statement about orderability we need a theorem from \cite{AM:orderability}, namely Corollary~1.3.
The formulation here is somewhat weaker, but suffices for our purposes.
\begin{theorem}
\label{thm:orderable_albers_merry}
Suppose that $(W,d\lambda)$ is a Liouville filling with vanishing first Chern class for a compact, connected, cooriented contact manifold $(Y,\alpha)$.
Let $\{ H_S \}_{S=1}^\infty$ denote a sequence of Hamiltonians that are linear at infinity with slope $S$. Denote the homomorphism $HF(W,J_S,H_S)\to SH(W)$ that is induced by the continuation maps by $i_S$.
Assume that there is $n_0>0$ and a class $c\in  SH(W)$ satisfying
\begin{itemize}
\item $c\neq 0$
\item if $i_S([x])=c$ for some $x\in CF(W,J_S,H_S)$, then $|\mathcal A_{H_S}(x)|<n_0$.
\end{itemize}
Then $(Y,\ker \alpha)$ is orderable.
\end{theorem}
We fix a sequence of Hamiltonians with increasing slope as indicated in Figure~\ref{fig:Hamiltonians} from Appendix B.
It then suffices to find a non-zero class $c\in SH(V_{\tilde \delta}(p)\,)$ whose action is bounded.

We distinguish the case when the Maslov index of a principal orbit is non-zero, $\mu_P\neq 0$, and the case when $\mu_P=0$.
In the former case the (perturbed) sequence of Hamiltonians has the property that there are only finitely $1$-periodic orbits (i.e.~Reeb orbits) with a given Conley-Zehnder index: the action of these $1$-periodic orbits is hence bounded. 
Furthermore, we find a non-zero homology class, so the above theorem applies.
Alternatively, $\mu_P$ implies index-positivity or negativity, so by Lemma~3.3 from \cite{AM:orderability} orderability follows (observe that the proof of Albers and Merry also works in the index-negative case if one works with the absolute value of the mean index).

In the case when $\mu_P=0$ there are infinitely many periodic Reeb orbits with the same Conley-Zehnder index, so the action bound does not follow directly.
Instead we show that there is a non-zero class that is represented by only finitely many $1$-periodic orbits (in fact by a unique orbit).
\begin{lemma}
Let $W^{2n}:=V_1(a)$ denote the natural Liouville filling of a Brieskorn manifold $\Sigma(a)$ of dimension at least $5$ satisfying $\sum_j \frac{1}{a_j}=1$.
Then $SH_{-n+1}(W)$ is infinite-dimensional, and each class $c\in SH_{-n+1}(W)$ has a unique representant $x\in CF_{-n+1}(W;H_S,J_S)$ for $S$ sufficiently large.
\end{lemma}
\begin{proof}
To prove the claim, we consider the Morse-Bott spectral sequence \eqref{eq:SS_SH_prf} for $SH(W)$.

We claim that any element with minimal total degree, which equals $-n+1$, cannot be killed, since the degree difference with the element with next smallest index is even and positive.
This will prove both assertions of the theorem.

The Maslov index of a principal orbit is zero by Proposition~\ref{prop:maslov_principal}, so the minimal total degree of a generator on the Morse-Bott manifold consisting of principal orbits, $\Sigma_P$, equals
$$
\mu_P-\frac{1}{2}\dim\Sigma_P/S^1+\ind_{p_M} f=-\frac{1}{2}\dim\Sigma_P/S^1 +0=1-n.
$$
Now consider the Maslov index of an exceptional orbit, which can be computed with formula~\eqref{eq_Maslov_index_Brieskorn}.
We consider an exceptional orbit with exponents $\{ a_i \}_{i\in I_T}$, where $\# I_T=1+n-k$.
The orbit space has dimension $2(n-k)-1$, so the degree shift equals
\[
\begin{split}
\mu(N\cdot K(I_{T})\,)-\frac{1}{2}\dim \Sigma_{exp}/S^1 &=2\sum_{j \in I_{T}} \frac{NT}{a_j}+
2\sum_{j \in I-I_{T}} \lfloor \frac{NT}{a_j} \rfloor+\#(I-I_{T})
-2NT -(n-k-1) \\
&=2\sum_{j} \frac{NT}{a_j}-2NT
+2 \sum_{j \in I-I_{T}}\left(
 \lfloor \frac{NT}{a_j} \rfloor
-\frac{NT}{a_j}
\right)
+k-(n-k-1)
\\
&
=2 \sum_{j \in I-I_{T}}\left(
 \lfloor \frac{NT}{a_j} \rfloor
-\frac{NT}{a_j}
\right) +2k+1-n
\end{split}
\]
By assumption, $a_i$ divides $N$ for $i\in I_T$.
On the other hand, $\sum_j \frac{1}{a_j}=1$, so $(\sum_{i\in I_T}\frac{NT}{a_i})+(\sum_{j\in I-I_T}\frac{NT}{a_j})=NT$.
Hence $\sum_{j\in I-I_T}\frac{NT}{a_j}$ is an integer, and the above term $2 \sum_{j \in I-I_{T}}\left(
 \lfloor \frac{NT}{a_j} \rfloor
-\frac{NT}{a_j}
\right)$ is even, so we see that the degree difference between generators with minimal degree on a principal orbit and an exceptional orbit is even.
Finally, to see that the degree difference is positive we note that the sum $\sum_{j \in I-I_{T}}\left(
 \lfloor \frac{NT}{a_j} \rfloor
-\frac{NT}{a_j}\right)$ contains $k$ terms and that each term is at least $-1$. 
In fact, since $a_j$ does \emph{not} divide $NT$ for $j\in I-I_{T}$, we see that each term is larger than $-1$.
Hence any term in the Morse-Bott spectral sequence corresponding to an exceptional cannot kill the minimum on a principal orbit.

Furthermore, since $\Sigma$ is simply-connected, there are no generators on covers of principal orbits that can kill the minimum either.
Therefore each minimum on any cover of a principal orbit survives to the $E^\infty$-page.
\end{proof}
This leaves $3$-dimensional Brieskorn manifolds with $\mu_P=0$: these are $\Sigma(2,3,6)$, $\Sigma(2,4,4)$ and $\Sigma(3,3,3$: with the Morse-Bott spectral sequence we can also show in these cases that symplectic homology of its natural filling does not vanish, and that certain classes come with an action bound.

Now take a class $c \in SH_{-n+1}(V_1(a)\,)$.
By the above lemma we find a unique representant and hence we obtain an action bound.
Now apply Theorem~\ref{thm:orderable_albers_merry}.
It follows that all Brieskorn manifolds whose exponents are greater than $1$ are orderable. 
\end{proof}

\begin{remark}
Alternatively, we can apply the Morse-Bott spectral sequence \eqref{eq:SS_SH} to say more about symplectic homology in the case of Brieskorn varieties or more generally weighted homogeneous polynomials.
The case when $\sum_j \frac{1}{a_j}\leq 1$ is particularly nice, since the indices grow in the ``wrong'' direction. This implies that many differentials must vanish.
In particular, if the exponents are sufficiently large, then the degree shifts are very large in negative direction as well, and the Morse-Bott spectral sequence will degenerate already at the $E^1$-page.

If  $\sum_j \frac{1}{a_j}>1$ non-vanishing is not obvious without further assumptions, so we preferred the above argument.
\end{remark}
The above remark also implies that the $+$-part of symplectic homology is in many cases an invariant of the contact structure.
Here is a simple explicit example.
\begin{corollary}
\label{cor:invariance_SH+}
Let $\Sigma(k,\ldots,k)$ be a Brieskorn manifold of dimension $2n-1$ with $k>3n$, and assume that $W$ is a simply-connected Liouville filling with $c_1(W)=0$.
Then $SH^+(W)$ does not depend on the choice of Liouville filling $W$ satisfying these properties.
In particular, $SH^+(W )$ is then an invariant of the contact structure. 
\end{corollary}
\begin{proof}
The proof of invariance is done by computing the invariant.
Consider the Morse-Bott spectral sequence~\eqref{eq:SS_SH_+} for $SH^+(W)$.
With Proposition~\ref{prop:maslov_principal} we see that the mean index is given by $k\cdot( \frac{n+1}{k} -1)=n+1-k$.
We compare two adjacent columns on the $E^1$-page.
The maximum on the $p+1$-st column (after re-indexing) has total degree $(p+1)\cdot(n+1-k)+2n-1$ and the minimum on the $p$-th column has total degree $p\cdot (n+1-k)+0$.
As $k>3n$, there is no differential, so the spectral sequence abuts at the $E^1$.
The same as true for a different filling, say $W'$. Since the $E^1$-pages are isomorphic, so are the symplectic homology groups $SH^+$. 
\end{proof}

We also want to point out another way to detect non-vanishing of symplectic homology.
\begin{proposition}
Let $(W^{2n},\omega)$ be a Liouville manifold such that $c_1(W)=0$, and such that the mean Euler characteristic $\chi_m(W,\omega)$ is defined.
Suppose that $\chi_m(W,\omega)\neq (-1)^{n-1} \frac{\chi(W)}{2}$.
Then $SH(W,\omega)\neq 0$.
\end{proposition}
\begin{proof}
If $SH(W,\omega)=0$, then by \cite[Theorem~4.1]{BO:SH_HC} $SH^{S^1}(W,\omega)=0$.
By the Viterbo sequence for equivariant symplectic homology, we find $\chi_m(W,\omega)=(-1)^{n-1}\chi_m(H^{S^1}_*(W,\partial W)\,)$.
The homology groups $H^{S^1}_*(W,\partial W)$ are isomorphic to $H_*(W,\partial W)\otimes H_*(BS^1)$, see \eqref{eq:S1_equivariant_filling}, so the claim follows.
\end{proof}

\subsection{Open problems}
We conclude this survey with a list of problems related to Brieskorn manifolds.
We have seen that Brieskorn manifolds and connected sums of them always have a rational mean Euler characteristic. This leads us to.
\begin{question}
Are there contact manifolds for which the mean Euler characteristic is an irrational number?
\end{question}

\begin{question}[AIM workshop]
Are $\Sigma(2k,2,2,2)$ and $\Sigma(2\ell,2,2,2)$ non-contactomorphic for $k\neq \ell$?
Does the additional algebraic structure on contact homology distinguish the $\Sigma(2k,2,2,2)$'s?
\end{question}

A related question is the following.
\begin{question}[AIM workshop]
Is there an exact symplectic cobordism from $\Sigma(2k,2,2,2)$ at the concave end to $\Sigma(2\ell,2,2,2)$ at the convex end for $k> \ell$?
\end{question}

Here are some questions involving classical invariants.
\begin{question}
Are there infinitely many non-isomorphic contact structures on $S^{2n-1}$ with the same classical invariants as the standard structure?
\end{question}
For $S^5$, $S^9$ and so on, this question was answered affirmatively by Ustilovsky, \cite{Ustilovsky:contact}.
The Brieskorn manifolds $\Sigma(k,2,\ldots,2)$ with $k \mod 8=\pm 1$ provide such contact structures. 
The condition on $k$ is necessary to get the standard smooth structure on a sphere, see Theorem~\ref{thm:std_and_kervaire}.
Since there are only finitely many homotopy classes of almost contact structures, one can imagine that things will work out.
This is indeed the case according to Formula~\eqref{eq:almost_contact_Brieskorn_spheres}.
However, there are infinitely many homotopy classes of almost contact structures in dimension $7,11$, and it seems to be difficult to reproduce the trivial homotopy class.

Another question related to Brieskorn manifolds concerns non-fillable contact manifolds.
We have seen that $\Sigma(2,\ldots,2,k)\cong \OB(T^*S^{n-1},\tau^k)$, where $\tau$ is a right-handed Dehn twists along the zero section.
\begin{question}
What are the classical invariants of the contact manifolds $\OB(T^*S^{n-1},\tau^{-k})$.
\end{question}
It might be interesting to note here that in dimension $3$ the contact open book $\OB(T^*S^1,\tau)$ is contactomorphic to the standard sphere, whose formal homotopy class of almost contact structures is trivial.
On the other hand, the contact open book $\OB(T^*S^1,\tau^{-1})$ is an overtwisted sphere with a non-trivial homotopy class of almost contact structures.

\section{Appendix A: the $+$-part of equivariant symplectic homology of $A_k$-singularities}
The argument from Section~\ref{sec:examples} applies to all simple singularities.
In fact, the $A_k$, $E_6$ and $E_8$-singularities are of Brieskorn type.
For future reference, we include the formulas for the $A_k$-singularities.

The $A_k$-singularity in dimension $2n$ is the Brieskorn singularity $V_0(k+1,2,\ldots,2)$, where the number of $2$'s is equal to $n$.
Let us give some detailed arguments for dimension $3/4$ using a perturbation. The Morse-Bott methods also apply, of course.

\begin{lemma}
The Brieskorn manifold $\Sigma(k,2,2)$ with its natural contact structure is contactomorphic to the lens space $(L(k,k-1),\alpha_0)$.
\end{lemma}
Here $L(k,k-1)=S^3/(z_1,z_2)\sim(e^{2\pi i/k}z_1,e^{-2\pi i/k}z_2)$, and $\alpha_0$ is the standard contact form on $S^3$, which is invariant under this action.
Since $SH^{+,S^1}$ is a symplectic deformation invariant, we can use other Liouville forms, for example corresponding to the contact form $\alpha_{a_1,a_2}$ given by
\begin{equation}
\label{eq:contact_ellipsoid}
\alpha_{a_1,a_2}=\frac{i}{2}\sum_{j=1}^2a_j\left(
z_j d\bar z_j-\bar z_j dz_j
\right)
,
\end{equation}
which is also invariant under the action defining the lens space.
If $a_1$ and $a_2$ are rationally independent, then $(S^3,\alpha_{a_1,a_2})$ has precisely two simple periodic Reeb orbits, namely $\gamma_1(t)=(e^{it/a_1},0)$ and $\gamma_2(t)=(0,e^{it/a_2})$.
These orbits are both non-degenerate and descend to periodic Reeb orbits $[\gamma_1(t)]_k$ and $[\gamma_2(t)]_k$ in $L(k,k-1)$, where $[\cdot]_k$ denotes the equivalence class in the lens space.

To compute the Conley-Zehnder indices, we will use a global trivialization of the tangent bundle, inspired by the quaternions on $S^3$:
\[
\begin{split}
R&=\frac{1}{a_1}(-y_1\partial_{x_1}+x_1\partial_{y_1})
+
\frac{1}{a_2}(-y_2\partial_{x_2}+x_2\partial_{y_2}) \\
U&=a_2(-x_2\partial_{x_1}+y_2\partial_{y_1})
+
a_1(x_1\partial_{x_2}-y_1\partial_{y_2}) \\
V&=a_2(-y_2\partial_{x_1}-x_2\partial_{y_1})
+
a_1(y_1\partial_{x_2}+x_1\partial_{y_2})
\end{split}
\]
Here $R$ is the Reeb field and $U,V$ form a basis for the contact structure.
This trivialization extends to the quotient and also to the filling $V_\epsilon(k+1,2,2)$, where all orbits are contractible.

We find the Conley-Zehnder indices for the $N$-fold covers
$$
\mu_{CZ}([\gamma_{1,N}]_k)=\mu(e^{it(1/a_1+1/a_2)}|_{t\in[0,N\cdot 2a_1\pi/k]})=
2\lfloor
\frac{N}{k}(1+\frac{a_1}{a_2})
\rfloor
+1,
$$
and a similar formula for $[\gamma_{2,N}]_k$.
By choosing $a_2$ to be much larger than $a_1$, we see that all generators in the Morse-Bott spectral sequence with small index correspond to covers of $[\gamma_1]_k$.
Furthermore, the index of covers of $[\gamma_1]_k$ forms a sequence of the form of $k-1$ copies of $1$, followed by a degree jump, then $k$ copies of $3$'s, a degree jump, and so on.
The symplectic invariance implies independence of $a_1$ and $a_2$, so we find
$$
SH^{+,S^1}_*(V_\epsilon(k,2,2)\,)
=
\begin{cases}
\Q^{k-1} & *=1 \\
0 & *\text{ is even or }*<1\\
\Q^k & \text{otherwise.} 
\end{cases}
$$
The result in dimension $5/6$ was already described in Section~\ref{sec:examples}. 
We now list the results for $V_\epsilon(k,2,\ldots,2)$ in dimension $2n$ with $n\geq 3$.
We will write $N$ for a natural number, and $k\in \Z_{\geq 2}$.
For $n$ odd we have the equivariant symplectic homology groups 
$$
SH_*^{+,S^1}(V_\epsilon(2k,2,\ldots,2))=\begin{cases} \Q^2 & \ast=2\lfloor \frac{N}{k} \rfloor + (2n-4)N+n-1 ;\text{ or}
\\  \; &
\; \ast \in 2\{(n-2)k+1\}\N \\
						0 & \text{$\ast$ is odd or $\ast < n-1$} \\
						\Q & \text{otherwise}   \end{cases} 
$$
and
$$
SH_*^{+,S^1}(V_\epsilon(2k+1,2,\ldots,2))=
\begin{cases} 
\Q^2 & \ast=2\lfloor \frac{2N}{2k+1} \rfloor + (2n-4)N+n-1 \text{ for $2N+1 \not\in (2k+1)\N$} \\
0 & \text{$\ast$ is odd or $\ast < n-1$} \\
\Q & \text{otherwise.}   
\end{cases} 
$$
And for $n$ even we have
$$
SH_*^{+,S^1}( V_\epsilon(2k,2,\ldots,2))=\begin{cases} 
\Q^2 & \ast=2\lfloor \frac{N}{k} \rfloor + (2n-4)N +n-1 ;\text{ or}
\\  \; & 
*=\; 2\lfloor \frac{N}{k} \rfloor + (2n-4)N+1 \text{ for $N \not\in k\N$}\\
0 & \text{$\ast$ is even or $\ast < n-1$}  \\
\Q & \text{otherwise}   
\end{cases} 
$$
and
$$
SH_*^{+,S^1}(V_\epsilon(2k+1,2,\ldots,2))=
\begin{cases} 
\Q^2 & \ast=2\lfloor \frac{2N}{2k+1} \rfloor + (2n-4)N +n-1 \text{ for $2N+1 \not\in (2k+1)\N$};\text{ or}
 \\  \; &  *=2\lfloor \frac{2N}{2k+1} \rfloor + (2n-4)N+1 \text{ for $N \not\in (2k+1)\N$}\\
0 & \text{$\ast$ is even or $\ast < n-1$}  \\
\Q & \text{otherwise}   
\end{cases} 
$$
Denote the quadric of complex dimension $n$ by $Q_n$.
Then the mean Euler characteristics of the ``$A_k$''-singularities are given by the following formulas
$$
\chi_m(V_{\epsilon}(2k,2,\ldots,2)\, )
=
(-1)^{n+1}
\frac{(k-1)\chi(Q_{n-2})+\chi(Q_{n-1}) }{2\cdot k(n-2+\frac{1}{k})},
$$
and
$$
\chi_m(V_{\epsilon}(2k+1,2,\ldots,2)\, )
=
(-1)^{n+1}
\frac{2k\chi(Q_{n-2})+\chi(\C \P^{n-1}) }{2(2k+1)(n-2+\frac{2}{2k+1})},
$$
which reveals the relationship between the mean Euler characteristic of these prequantization orbibundles, the above Brieskorn manifolds $\Sigma(a)$, and the orbifold Euler characteristic and orbifold Chern class of the quotient $\Sigma(a)/S^1$.
This relation is also described in \cite{CDvK:right-handed}.

\section{Appendix B: Morse-Bott spectral sequences for symplectic homology of Liouville domains with periodic Reeb flows on the boundary}
\label{sec:MB_SS}

In the following $(W,\omega=d\lambda )$ will be a Liouville domain, and $H:W\to \R$ will be an autonomous Hamiltonian function that is linear at infinity such that its slope $s$ is not the period of a periodic Reeb orbit of $\lambda_{\partial W}$.
In addition, we will assume that the $1$-periodic orbits of $X_H$ are of Morse-Bott type, which intuitively means that the critical set $C$ of the action functional $\mathcal A$ is a smooth manifold for which the restriction of the Hessian to the normal bundle is non-degenerate.
Since we only consider the autonomous case, we can map this critical set into $W$ by taking the starting point of each orbit, so we take the following definition.
\begin{definition}
We say that the $1$-periodic orbits of an autonomous Hamiltonian $H$ are of {\bf Morse-Bott type} if
\begin{itemize}
\item the {\bf critical set} or {\bf critical manifold} $C=\{ x\in W~|~Fl^{X_H}_1(x)=x \}$ forms a (possibly disconnected) compact submanifold of $W$ without boundary.
\item the restriction of the linearized return map to the normal bundle of each connected component $\Sigma$ of $C$ is non-degenerate, i.e.~the linear map
$$
T_x Fl^{X_H}_1|_{\nu(\Sigma)}-\id|_{\nu(\Sigma)}
$$
is invertible for all $x\in \Sigma$.
\end{itemize}
\end{definition}
\begin{remark}
Since the Hamiltonian is assumed to be linear at infinity and its slope does not lie in the spectrum, all $1$-periodic orbits lie in a compact set.
Furthermore, the Morse-Bott condition tells us that the critical manifolds are non-degenerate in the normal direction, so they are isolated, i.e.~there is a neighborhood of each connected component of the critical manifold that does not contain any other components.
By compactness the number of components is finite.
\end{remark}

Such a Hamiltonian cannot be used for Floer homology unless the critical set consists of non-degenerate critical points of the Hamiltonian.
However, the methods of \cite{CFHW} apply.
There it was proved that one can still define (local) Floer homology if the critical set $C$ consists of circles that are transversely non-degenerate. It was also pointed out that the methods should generalize to other critical manifolds.
We will work out a class of such critical manifolds that is particularly relevant to symplectic homology, but to simplify later discussions we will make the following assumption:
\begin{enumerate}
\item[(ST)] For each connected component $\Sigma$ of the critical manifold there is a symplectic trivialization $\epsilon_\Sigma:\Sigma\times(\R^{2n},\omega_0)\to TW|_{\Sigma}$.
Furthermore, if $\gamma$ is a critical point of the action functional $\mathcal A$, then $\gamma$ has a capping disk in $W$ to which the trivialization $\epsilon_{\Sigma}|_{\gamma}$ extends. 
\end{enumerate}

Note that this condition implies that all $1$-periodic orbits are contractible in $W$.
Relevant examples include $\Sigma=S^1$, a transversely non-degenerate circle, and Brieskorn manifolds.
To see the latter, we observe that any Brieskorn variety $V_1(a)$ is homotopy equivalent to a wedge of spheres and that $c_n(V_1(a)\,)$=0, so $TV_1(a)$ has a global symplectic trivialization.
In particular $TV_1(a)|_{\Sigma}$ is symplectically trivial, and the trivialization extends over disks.
To give one non-example, we note that $ST^*\R P^2$ in $T^*\R P^2$ does not satisfy the (ST)-assumption.

Choose a Morse function $h$ on the (disconnected) critical manifold $C$, extend this function to a neighborhood $U$ of $C$ (we still denote this extension by $h$), and cut it off by a bump function $\rho_b$ to get a function supported in the neighborhood $U$.
If we introduce a metric, then we define such a function $\rho_b$ depending only on the distance to $\Sigma$.
We will denote the resulting function by $h \rho_b$, and by choosing $\rho_b$ sufficiently small, we can assume that $0\leq h \rho_b <1$. 
Let $\Sigma$ be a connected component of the critical manifold $C$ and denote the slope of $H$ at $\Sigma$ by $s$.
We use $h|_{\nu_W(\Sigma)}$ to define a time dependent function by putting
\[
\begin{split}
\bar h: \nu(\Sigma) \times S^1 & \longrightarrow \R \\
(p,n;t) & \longmapsto h(Fl^{sR}_t(p)) \rho_b(n),
\end{split}
\] 
where $p$ is a point in $\Sigma$ and $n\in \nu(\Sigma)_p$ a normal vector to $\Sigma$.
For $\delta>0$, define
$$
H_\delta=H+\delta \bar h.
$$

\begin{lemma}
\label{lemma:orbits_curves}
Suppose that $\Sigma$ is a connected component of the critical manifold of an autonomous Hamiltonian $H$ of Morse-Bott type.
Let $U$ denote a neighborhood of $\Sigma$ that does not contain $1$-periodic orbits of $X_H$ other than those in $\Sigma$.
Then for any open neighborhood $V\subset U$ of $\Sigma$, there exists $\delta_0>0$ with the property that for any $\delta \in]0,\delta_0[$ the following hold.
\begin{enumerate}
\item all $1$-periodic orbits of the perturbed Hamiltonian $H_\delta$ that are contained in $U$ are already contained in $V$.
\item all Floer trajectories completely contained in $U$ are already contained in $V$.
\end{enumerate} 
\end{lemma}
\begin{proof}
The proof follows that of Lemma~2.1 in \cite{CFHW} almost verbatim.
\end{proof}

We can be more specific about $1$-periodic orbits of the perturbation $H_\delta$ given above: if $\delta$ is sufficiently small then all $1$-periodic orbits of the perturbation $H_\delta$ correspond to critical points of the Morse function $h$.
We denote the $1$-periodic orbits of $H_\delta$ that are contained in $U$ by ${\mathcal P}_{H_\delta}^U$.
We follow \cite{CFHW} to define the local Floer homology of $\Sigma$. Fix a ring $R$, and define the local Floer homology as the homology of the following complex.
\begin{itemize}
\item define $CF^{loc}_*(\Sigma,H_\delta,J)$ as the $R$-module freely generated by $1$-periodic orbits of $H_\delta$.
We grade each $1$-periodic orbit with minus the Conley-Zehnder index measured with respect to a trivialization on a capping disk.

\item define the local ``moduli space'' as the set
\[
\begin{split}
\mathcal M^{U}(x_+,x_-,H_\delta,J):=
\{
u: \R \times S^1 \to U \subset W ~|&~u_s+J_t(u) 
\left(u_t
-X_{H_\delta}
\right)
=0 \\
&~
\lim_{s \to \mp \infty } u(s,\cdot)=x_{\pm}(\cdot) \}
\end{split}
\]
where we mean uniform convergence. We will assume that $J$ is chosen generically, such that the moduli spaces are smooth manifolds of dimension determined by the Fredholm index.

\item the differential is induced from the full Floer complex in the following way.
We define coherent orientations on the full perturbed complex $CF_*(H_\delta,J)$ with the method of \cite{FH}.
In general, there can be several choices of coherent orientations as pointed out in Remark~8.1.15 of \cite{FOOO:Anomaly2}.
We will fix one particular choice by requiring that Floer flow lines for a $C^2$-small Hamiltonian (these correspond to Morse flow lines as we shall see below) are counted in the same way as the version of Morse homology that is isomorphic to singular homology.

Now define the differential on generators of $CF^{loc}_*(\Sigma,H_\delta,J)$  by
\[
\begin{split}
\partial_k^{loc}: CF^{loc}_k & \longrightarrow CF^{loc}_{k-1} \\
\bar x & \longmapsto \sum_{\underset{-\mu_{CZ}(\ubar x)=k-1}{\ubar x  \in {\mathcal P}_{H_\delta}^U } } \sum_{[u]\in \mathcal M^{U}(\bar x,\ubar x,H_\delta,J)/\R} \epsilon([u])  \ubar x,
\end{split}
\]
where $\epsilon([u])$ is the sign assigned to a Floer trajectory by comparing the flow orientation (induced by the $\R$-action) of $u$ to that of the coherent orientation on $\mathcal M^{U}(\bar x,\ubar x,H_\delta,J)/\R$.

\item the local Floer homology is $HF^{loc}_*(\Sigma,H_\delta,J)=H_*(CF^{loc}_*(\Sigma,H_\delta,J),\partial^{loc})$.
\end{itemize}
It follows from Lemma~\ref{lemma:orbits_curves} that the local Floer homology of $\Sigma$ is well-defined and independent of the choice of complex structure $J$, and perturbation $h$.
The following is an analogue of Proposition~2.2 in \cite{CFHW}.
\begin{proposition}
\label{prop:local_Floer}
Suppose that $\Sigma$ is a component of the critical manifold of an autonomous Hamiltonian of Morse-Bott type $H:W\to \R$ satisfying 
\begin{enumerate}
\item the conditions from Lemma~\ref{lemma:orbits_curves}.
\item the symplectic triviality assumption (ST) holds and $c_1(W)=0$.
\item assume that $c$, where $H(\Sigma)=c$, is a regular value of $H$ such that the Liouville vector field $X$ is transverse to $\Sigma_c:=H^{-1}(c)$.
\item the restriction of the Hessian to the Liouville direction is positive definite.
\item one of the following conditions hold: $H^1(\Sigma;\Z_2)=0$, $\Sigma=S^1$ is a good Reeb orbit, or the linearized Reeb flow is complex linear with respect to some unitary trivialization of the contact structure along every periodic Reeb orbit in $\Sigma$.
\end{enumerate}
Then the local Floer homology is isomorphic to
$HF^{loc}_{*+shift}(\Sigma,H,J;R)\cong H^{Morse}_*(\Sigma;R)$, where $shift=\mu_{RS}(\Sigma)-\frac{1}{2}\dim \Sigma/S^1$.
\end{proposition}
In this statement we compute the Robbin-Salamon index of an orbit $\gamma$ in $\Sigma$ as a periodic Reeb orbit. Hence there is no minus sign.

\begin{proof}
The argument follows the ideas of the proof of Proposition~2.2 from \cite{CFHW} closely.
See also the survey of Oancea, \cite{Oancea:survey}, section~3.3    who calls the following ``spinning''.

\noindent
{\bf Simplifying the flow by unwrapping or spinning}
This is essentially step 1 in the proof of Proposition~2.2 in \cite{CFHW}.
By assumption (iii), a neighborhood of $\Sigma$ is diffeomorphic to $]1-\epsilon,1+\epsilon[ \times \nu_{\Sigma_c}(\Sigma)$ with symplectic form $d(r \lambda|_{\Sigma_c})$, where $r$ is the coordinate on $]1-\epsilon,1+\epsilon[$.
Assume that the slope of $H$ along $\Sigma$ is equal to $s$, i.e.~$\partial_r H|_{\Sigma}=s$.
Define the Hamiltonian $K=-s\cdot r$, and let $\Delta_t=Fl^{X_K}_t$ denote its time-$t$ flow.
Put $\tilde K(t,x):=H_\delta(t,Fl^{X_K}_{-t}(x))+ K(x)$.
We construct a map between the Floer complexes $CF^{loc}(H_\delta,J)$ and $CF^{loc}(\tilde K,J_{\tilde K})$ as follows
\begin{itemize}
\item send a $1$-periodic orbit $x$ of $H_\delta$ to $\hat x=\Delta_t\circ x$. By the formula for the Hamiltonian of a product, we see that $\hat x$ is a $1$-periodic orbit of $\tilde K$.
\item define the complex structure by conjugation, $J_{\tilde K}:=T \Delta_t \circ J \circ T\Delta_t^{-1}$.
With this choice of complex structure, every Floer trajectory for $CF^{loc}(H_\delta,J)$ maps to a Floer trajectory for $CF^{loc}(\tilde K,J_{\tilde K})$.
Furthermore, regularity of Floer trajectories holds for $CF^{loc}(\tilde K,J_{\tilde K})$ if and only if it holds for $CF^{loc}(H_\delta,J)$.
\item We pull back the coherent orientations on $CF^{loc}(H_\delta,J)$ with $\Delta_t^{-1}$ to get a local system on $CF^{loc}(\tilde K,J_{\tilde K})$.
Below we explain these orientations and the local coefficient system in more detail.

\end{itemize}
Note that the resulting map between chain complexes need not be grading preserving, since the map $\Delta_t$ affects the capping disks for the $1$-periodic orbits.

\noindent
{\bf Local Floer homology via Morse homology}
This follows from an analogue of steps 2 and 4 in the proof of Proposition~2.2 in \cite{CFHW}.
We compute the local Floer homology up to a degree shift.
The restriction of $\tilde K$ to $\Sigma$ is $C^2$-small, and all $1$-periodic orbits of $\tilde K$ in the neighborhood $\nu(\Sigma)$ are critical points of $\tilde K$ by Lemma~\ref{lemma:orbits_curves}.
Consider the local Floer homology $HF^{loc}(\Sigma,\tilde K,J;O)$, where we have chosen $J$ to be time-independent, and such that the gradient flow of the metric $g=\omega(\cdot,J\cdot)$ is Morse-Smale.
The notation $O$ indicates that we take the coherent orientations from the complex $CF^{loc}(H_\delta,J)$.
The Floer equation reads $\partial_s u+J\partial_t u=-grad_g \tilde K$, so we see that rigid gradient flow lines of $\tilde K$ with respect to $g$ are also Floer trajectories.
This is actually a $1-1$-correspondence, which was already noticed by Floer, \cite{Floer:Lagrangian_Floer}, as well in step 4 of the proof of Proposition~2.2 in \cite{CFHW}.
A detailed version of this correspondence was worked out by Po\'zniak in \cite[Proposition 3.4.6]{Pozniak}, who identifies the kernels of the linearized Floer equation and the kernel of linearized gradient flow, to obtain an isomorphism between Floer homology and Morse homology in a Lagrangian setup.
For convenience we will refer to this relation in the Hamiltonian setting also as Po\'zniak's correspondence.

On the Floer side we use coherent orientations following the scheme of \cite{FH}, and with Po\'zniak's correspondence we get a corresponding system of orientations on the Morse homology side, yielding a local coefficient system on $\Sigma$, which we denote by $\mathcal L_{\Sigma}$.
We will now describe this local coefficient system in more detail.

\noindent
{\bf From orientations to local coefficient systems}
We follow some of the ideas in \cite{BO:SH_MB} Section~4.4 in order to define coherent orientations for the local Floer homology of a Morse-Bott critical manifold $\Sigma$.
Since we are in a non-degenerate setup after perturbation, we will only need coherent orientations for Floer trajectories corresponding to Morse flow lines, so we will not go into the general setup which involves cascades.

Suppose that $\Sigma$ is a critical manifold of the Hamiltonian vector field $X_H$ belonging to the autonomous Hamiltonian $H$, so each point $\gamma$ of $\Sigma$ is a $1$-periodic orbit of $X_H$.
With respect to a unitary trivialization $\epsilon$ of $\gamma^*TW$ the linearized flow $TFl^{X_H}_t$ is a path of symplectic matrices $\psi(t)$ satisfying the ODE
\begin{equation}
\label{eq:path_sympl_loop_Ham}
\begin{split}
\dot \psi &= J_0 S(t) \psi \\
\psi(0) & = \id.
\end{split}
\end{equation}
Here $S(t)$ is a \emph{loop} of symmetric matrices that is non-degenerate in the direction normal to $\Sigma$ by the Morse-Bott assumption.
We will consider loops of symmetric matrices satisfying this non-degeneracy assumption,
$$
\mathcal S
:=
\{
S: S^1 \to Sym_{2n}(\R)~|~\det( \left( \id-\psi_S(1) \right)|_{\nu(\Sigma)} )\neq 0 
\}
.
$$
Because we allow for Morse-Bott degeneracies we will need asymptotic weights for the Sobolev spaces in order to get the Fredholm property for the operators considered below, just as in \cite{BO:SH_MB} Section~4.4.
For a cylindrical end $([R,\infty[ \times S^1,J_0,ds\w dt)$ one uses the measure $e^{\delta s}ds \w dt$, where $\delta>0$ is chosen to be smaller than the spectral gap of the asymptotic operator.
If we use polar coordinates for $\C$, then we have $(\C,J_0,rdr\w d\theta)$, with $J_0 \partial_r =\frac{1}{r}\partial_\theta$, and the appropriate measure with asymptotic weights becomes $\rho_\delta(r) dr\w d\theta$, where $\rho_\delta(r)$ equals $r$ near $r=0$ (polar measure) and $r^{\delta-1}$ for large $r$ (pulled back measure).
Given a vector space $E$ (to be replaced by a trivialization of the symplectic vector bundle over $\C$), we define the Banach spaces
\[
\begin{split}
W^{1,p}_\delta(\C,E)&=\{ 
 f:\C \to E ~|~f(r,\theta) \in W^{1,p}(\C,E;\rho_\delta (r) dr \w d\theta)
\} 
,\text{ and} \\
L^{p}_\delta(\C,E)&=\{ 
 f:\C \to E ~|~f(r,\theta) \in L^{p}(\C,E;\rho_\delta (r) dr \w d\theta)
\}
.
\end{split}
\]
For the Morse-Bott setup we need to include elements in the kernel of the asymptotic operator.
Find a basis of solutions of the ODE given by the asymptotic operator
\begin{equation}
\label{eq:asymptotic_ODE}
J_0 \partial_t e_j+S(t) e_j=0
\end{equation}
Choose a cutoff function that is supported at infinity, say $\rho:\C \to \R$ depending only on $|z|$ such that $\rho(z)=1$ if $|z|>2R$ and $\rho(z)=0$ for $|z|<R$, where $R$ is a constant.
Then the vector-valued functions $\tilde e_j(z):=\rho(z)e_j(\frac{z}{|z|})$ extend to vector valued functions on $\C$.
We define a vector space of dimension $\dim \Sigma$, denoted by ${\mathcal T}_\gamma \Sigma=span_\R (\tilde e_1,\ldots,\tilde e_{\dim \Sigma})$.
The notation indicates that ${\mathcal T}_\gamma \Sigma$ is morally the tangent space to the Morse-Bott manifold.
For $\tilde S \in \mathcal S$ define a space ${\mathcal O}(\C,E;\tilde S)$ consisting of operators 
\begin{equation}
\label{eq:operator}
D: W^{1,p}_\delta(\C,E)\oplus {\mathcal T}_\gamma \Sigma \longrightarrow L^{p}_\delta(\C,\Omega^{0,1}(E)\, )
\end{equation}
satisfying the following conditions.
\begin{enumerate}
\item in a local unitary trivialization of $E$, the operator has the form
\begin{equation}
\label{eq:operator_unitary_triv}
D=\left( \partial_r+J_0\partial_\theta+S(z) \right)d\bar z,
\end{equation}
where we interpret an element in $W^{1,p}_\delta(\C,E)\oplus {\mathcal T}_\gamma \Sigma$ as a linear combination of the form
$$
X=\tilde X+\sum_{j=1}^{\dim \Sigma} a_j \tilde e_j,
$$  
with $\tilde X\in W^{1,p}_\delta(\C,E)$.
Together with the next condition, we see that this makes sense.
\item $\lim_{s\to \infty}S(s,t)=\tilde S(t)$.
So by construction, the $\tilde e_j$'s map under $D$ into $L^{p}_\delta(\C,\Omega^{0,1}(E)\, )$.
\end{enumerate}
Because of the asymptotic weights, the space ${\mathcal O}(\C,E;\tilde S)$ consists of Fredholm operators.
Also, each space ${\mathcal O}(\C,E;\tilde S)$ has fixed asymptotics, so the determinant bundle over ${\mathcal O}(\C,E;\tilde S)$ is trivial by \cite{FH}[Proposition 7].
Below we will construct a bundle version of the map~\eqref{eq:operator}.

\subsubsection{Banach bundles over $\Sigma$}
In the following it will be useful to think of $\Sigma$ as a submanifold in $\Lambda \bar W$.
Throughout the following construction we fix $R>0$ and a cutoff function $\rho:\C \to \R$ such that $\rho(z)=1$ for $|z|>2R$ and $\rho(z)=0$ for $|z|<R$.
We will also need a symmetric connection $\nabla$ for $TW$.

Cover $\Sigma$ with finitely many good charts $\{ U_a\subset \R^{\dim \Sigma} \}_a$, meaning that each intersection $U_{a_1}\cap \ldots \cap U_{a_k}$ is either diffeomorphic to an open ball or empty.
For each $U_a$ choose a collection of pairs of maps $(u_\gamma^a,\epsilon_\gamma^a)$ such that
$$
u_\gamma^a: \C \to W
$$
is a capping plane for $\gamma$, i.e. $u_\gamma^a(z)=\gamma( \frac{z}{|z|})$ for $|z|>R$ and 
$$
\epsilon_{\gamma}^a: \C \times \R^{2n} \to {u_\gamma^a}^* TW
$$
is a trivialization of ${u_\gamma^a}^* TW$ that only depends on $\frac{z}{|z|}$ if $|z|>R$.
We will now use the (ST) assumption to make the following construction work.
Denote the symplectic trivialization of $TW|_{\Sigma}$ by $\epsilon_\Sigma:\Sigma \times \R^{2n} \to TW|_{\Sigma}$.
Then we choose the above trivializations $\epsilon_\gamma^a$ by putting
\[
\begin{split}
\epsilon_\gamma^a:\C-B_R(0) \times \R^{2n} \longrightarrow {u_\gamma^a}^* \\
(z,v) & \longmapsto \epsilon_\Sigma(u_\gamma^a(z),v).
\end{split}
\]
We extend this trivialization to the rest of $\C$. For fixed $a$, this can be done in a way that depends smoothly on $\gamma\in U_a$.
We now construct a bundle version of the operator~\ref{eq:operator}.
We first define the analogue of $\mathcal T_\gamma \Sigma$.
For each $\gamma\in U_a$ find a basis of solutions to the following ODE for the vector field $X$ along $\gamma$, seen as a loop in $W$,
$$
J(\nabla_t X-\nabla_X X_H)=0.
$$
This is a coordinate-free version of the Equation~\eqref{eq:asymptotic_ODE}.
We can assume that this basis depends smoothly on the point $\gamma$ in $U_a$, and we will denote this basis of solutions by $\{ {e_\gamma^a}_{,i} \}_i$.
Using the cutoff function $\rho$ we extend these functions to a map
\[
\begin{split}
{\tilde e_\gamma^a}_{,i}: \C & \longrightarrow {u_\gamma^a}^*TW \\
z & \longmapsto \rho(z) {e_\gamma^a}_{,i}(\frac{z}{|z|}).
\end{split}
\]
Define the finite-dimensional vector space
$$
{\mathcal T}_\gamma^a\Sigma:=span_\R ({\tilde e_\gamma^a}_{,1},\ldots,{\tilde e_\gamma^a}_{,\dim \Sigma} ).
$$
We construct a Banach bundle $W_\Sigma$ over $\Sigma$ whose fiber over $\gamma\in U_a$ is
$$
W^{1,p}_\delta(\C, \R^{2n})
\oplus
\R^{\dim \Sigma}
\cong
W^{1,p}_\delta(\C, {u_\gamma^a}^* TW)
\oplus
{\mathcal T}_\gamma^a\Sigma
.
$$
Define the transition function $g_{ab}:\C \times \R^{2n}$ by putting $g_{ab}(z,\cdot):=\id$.
This makes sense since ${\epsilon_\gamma^b}^{-1}\circ {\epsilon_\gamma^a}(z,\cdot)=\id$ for $|z|>R$ by the above choice of trivializations.
The above functions $\tilde {e^a_\gamma}_i$ are related to one another by the formula
$$
\sum_i {\tilde e^a_\gamma}_i {\tilde g_{ab}}^i{}_j={\tilde e^b_\gamma}_j.
$$
These functions satisfy the cocycle condition, so we say that elements are equivalent,
$$
(\gamma;X,v)\in U_a \times \left( W^{1,p}_\delta(\C, \R^{2n})
\oplus
\R^{\dim \Sigma} \right)
\sim
(\delta;Y,w)\in 
U_b \times \left( W^{1,p}_\delta(\C, \R^{2n})
\oplus
\R^{\dim \Sigma} \right)
$$
if $\gamma=\delta$, $Y(z)=g_{ab}(z,X(z)\,)$ and $v^i=\sum_j {\tilde g_{ab}}^i{}_j w^j$.
This gives rise to the Banach bundle
$$
W_\Sigma := \coprod U_a \times \left( W^{1,p}_\delta(\C, \R^{2n})
\oplus
\R^{\dim \Sigma} \right) / \sim.
$$
Similarly, we define a Banach bundle $L_\Sigma$ whose fiber over $\gamma\in U_a$ is given by
$$
L^{1,p}_\delta(\C, \R^{2n})
\cong
L^{1,p}_\delta(\C,\Omega^{0,1}( {u_\gamma^a}^* TW)\, )
.
$$
We define a vector bundle homomorphism $D_\Sigma:W_\Sigma \to L_\Sigma$ covering the identity. On the level of charts, put
\begin{equation}
\label{eq:D_chart}
\begin{split}
{D_{\Sigma}}^{U_b}_{U_a}: W_\Sigma|_{U_a} & \longrightarrow L_\Sigma|_{U_b} \\
(\gamma;X,v) & \longmapsto 
\left( \gamma;(\epsilon^b_\gamma)^{-1}\circ
\rho(z)\cdot \left( \nabla_s \cdot +J (\nabla_t \cdot -\nabla_\cdot X_H) 
\circ (
\epsilon^a_\gamma(X)+\sum_i v^i {\tilde e^a_\gamma}_i
)
\right)
\right),
\end{split}
\end{equation}
so we simply substitute the vector $\epsilon^a_\gamma(X)+\sum_i v^i {\tilde e^a_\gamma}_i$ into the differential operator $X \mapsto \nabla_s X +J (\nabla_t X -\nabla_X X_H)$.

\begin{lemma}
The map $D_{\Sigma}$ defines a bundle homomorphism that restricts to a Fredholm operator on each fiber.
Furthermore ${D_{\Sigma}}^{U_a}_{U_a}|_\gamma$ is a compact perturbation of an operator in $\mathcal O(\C,{u_\gamma^a}^*TW,S)$ for some $S\in \mathcal S$.
\end{lemma}
\begin{proof}
We first check that the map $D_{\Sigma}$ is compatible with the transition functions.
For $|z|\leq R$ the cutoff function $\rho(z)$ vanishes, so the right-hand side is then $0$.
For $|z|>R$, the trivializations $\epsilon^a_\gamma$ do not depend on $a$, so the $X$-terms in the expression~\eqref{eq:D_chart} are independent of $a$ and $b$.
For the terms involving $v$, we go to the $U_b$ chart, so $v\mapsto \tilde g_{ab}^{-1} v$, resulting in the term
$$
\sum_{i,j} v^i {({\tilde g_{ab}}^{-1})}^j{}_i {\tilde e^b_\gamma}_j=
\sum_{i,j,k} v^i {({\tilde g_{ab}}^{-1})}^j{}_i 
{\tilde g_{ab}}^k{}_j {\tilde e^a_\gamma}_k
=\sum_k v^k {\tilde e^a_\gamma}_k
$$
in the above expression, so the map $D_\Sigma$ is well-defined.
To check the Fredholm property, we define
\[
\begin{split}
D_{\gamma,a}: W_\Sigma|_{\gamma} & \longrightarrow  L_\Sigma|_{\gamma} \\
X & \longmapsto 
\nabla_s X+J(\nabla_t X-\nabla_X X_H )
.
\end{split}
\]
We claim that $D_{\gamma,a}\in \mathcal O(\C,{u_\gamma^a}^*TW;\tilde S)$ for some $\tilde S\in \mathcal S$.
To see this, we use the trivialization $\epsilon_\gamma^a$.
This gives a unitary frame $\{f_j\}$ of ${u_\gamma^a}^* TW$ along $\C$.
By assumption, this frame only depends on $\frac{z}{|z|}$ if $|z|>R$.
With respect to such a trivialization the expression $\nabla_s X+J(\nabla_t X-\nabla_X X_H )$ has terms of the form
$$
\nabla_s \sum_j(f_jX^j)=\sum_j(\nabla_s f_j)X^j+ \sum_j f_j\partial_s X^j.
$$
The latter corresponds to $\partial_s X$ and the former is a constituent of the term $S(z)$ in Equation~\eqref{eq:operator_unitary_triv}.
The other terms can be identified similarly.

Now observe that $\nabla_s f_j$ vanishes for $|z|>R$ due to our choice of trivialization.
Furthermore, for large $|z|$, the second term in the above expression for $D_{\gamma,a}$ equals $J(\nabla_t X-\nabla_X X_H )$.
This is a coordinate-free expression for the asymptotic operator with a minus sign, so $D_{\gamma,a}\in \mathcal O(\C,{u_\gamma^a}^*TW;\tilde S)$.
In particular, we conclude that $D_{\gamma,a}$ is a Fredholm operator.
To see that $D_{\Sigma}|_\gamma$ is a Fredholm operator we observe that the map $L^p \to L^p$, which sends $X$ to $\rho(z) X$ is a compact operator, which preserves the Fredholm condition.
\end{proof}

\subsubsection{Determinant bundles}
We take the determinant bundle of the vector bundle homomorphism $D_{\Sigma}$ to obtain a real line bundle $Det(D_{\Sigma})\to \Sigma$.
We can choose the transition functions of this line bundle to lie in $O(1)$, so we obtain a local coefficient system $\mathcal L_{\Sigma}$ on $\Sigma$.
The signed count of Morse flow lines from Po\'zniak's correspondence is hence twisted with the local coefficient system $\mathcal L_{\Sigma}$.
To make this precise, it is convenient to use Abouzaid's description of orientations in Floer homology, see \cite{Abouzaid:Viterbo_thm}, Sections~1.4 and 1.5.
We briefly review his description.
Given a $1$-periodic Hamiltonian orbit $\hat \gamma$ of $H_\delta$, choose a capping plane $u_{\hat \gamma}$ and a trivialization $\epsilon_{\hat \gamma}$ of $TW$ along $u_{\hat \gamma}$.
As before we obtain an expression for the asymptotic operator of the form $-(J\partial_t+\tilde S_{\hat \gamma})$.
We extend $\tilde S_{\hat \gamma}$ to a map $S_{\hat \gamma}$ defined on $\C$, for example by using the cutoff function $\rho$.
We obtain a Fredholm operator 
$$
D_{\hat \gamma}=\partial_s+J\partial_t+S_{\hat \gamma},
$$
and define the determinant line $Det(D_{\hat \gamma})$.
This line comes with a $\Z$-grading given by minus the Conley-Zehnder index of $\psi$, where $\psi$ solves the ODE~\eqref{eq:path_sympl_loop_Ham} with $S=\tilde S_{\hat \gamma}$.

\begin{remark}
\label{rem:det_bundle=line_bundle}
For later purposes, recall that $\hat \gamma$ corresponds to a critical point $\gamma$ on $\Sigma$, so if we choose the above capping plane $u_{\hat \gamma}$ and trivialization $\epsilon_{\hat \gamma}$ with the above construction for $D_\Sigma$ we get a canonical isomorphism $Det(D_{\hat \gamma}) \cong Det(D_{\Sigma})|_{\gamma}$. 
\end{remark}

We return to Abouzaid's construction of orientations.
Choose orientations $\delta_{\hat \gamma}^+$ and $\delta_{\hat \gamma}^-$ for the determinant line $Det(D_{\hat \gamma})$, and define the orientation line $o_{\hat \gamma}$ of $\hat \gamma$ as the abelian group generated by  $\delta_{\hat \gamma}^+$ and $\delta_{\hat \gamma}^-$ subject to the relation that their sum vanishes, i.e.
$$
o_{\hat \gamma}:=\langle \delta_{\hat \gamma}^+,\delta_{\hat \gamma}^-~|~\delta_{\hat \gamma}^+ +\delta_{\hat \gamma}^-=0 \rangle
.
$$
Now define the Floer chain complex as the graded abelian groups
$$
CF_k^{loc}(\nu(\Sigma),H_\delta,J):=\bigoplus_{ \underset{-\mu_{CZ}(\hat x)=k}{\hat x\in {\mathcal P}_{H_\delta}^{\nu(\Sigma)} }} o_{\hat x}.
$$
The differential is defined on orientation lines by
$$
\partial|_{o_x}: a \longmapsto
\sum_y \sum_{[u]\in \mathcal M(x,y, H_\delta,J)}
\partial_u a,
$$
where $\partial_u: o_x \to o_y$ is a map between orientation lines constructed in Lemma~1.5.4 of \cite{Abouzaid:Viterbo_thm}. This is similar to the earlier description where we fixed orientations.
As always, this map $\partial$ is extended linearly.
Now define $L_x$ to be the orientation line of the line $Det(D_\Sigma)|_{x}$.
The corresponding bundle $L$ with fiber $L_x$ is an $O(1)$-bundle over $\Sigma$.
By Remark~\ref{rem:det_bundle=line_bundle}, we have a canonical isomorphism between the previous Floer complex and the twisted Morse complex
$$
C_k^{Morse}(\tilde K;\mathcal L_\Sigma):=\bigoplus_{ \underset{\ind_x\tilde K=k}{x\in Crit(\tilde K) }} L_x.
$$
This is an isomorphism of chain complexes, since we can identify Floer flow lines and their determinant lines in the complex $CF_*^{loc}(\nu(\Sigma),H_\delta,J)$ with Morse flow lines with local coefficients by Po\'zniak's correspondence.
We can implement Morse homology with local coefficients using flat vector bundles as in \cite{Ono:flux}, Section 6: we are in such a setup as we have chosen the transition functions of the determinant bundles to lie in $O(1)$.
Hence we obtain the following isomorphism, up to a degree shift, which we will work out below.
\begin{equation}
\label{eq:local_Floer}
HF^{loc}_{*+shift(\Sigma)}(\nu(\Sigma),H_\delta,J;R)
\cong 
H^{Morse}_{*}(\Sigma,\tilde K;R\otimes_\Z \mathcal L_{\Sigma}).
\end{equation}

The next lemma provides a simple criterion to see when $\mathcal L_{\Sigma}$ is trivial.
\begin{lemma}
\label{lemma:local_coefficients_trivial}
Suppose that one of the following conditions hold: $H^1(\Sigma;\Z_2)=0$, $\Sigma=S^1$ is a good Reeb orbit, or the linearized Reeb flow is complex linear with respect to some unitary trivialization of the contact structure along each $\gamma \in \Sigma$.
Then $\mathcal L_{\Sigma}$ is trivial.
\end{lemma}
\begin{proof}
If $H^1(\Sigma;\Z_2)$ is trivial, then all local systems on $\Sigma$ are trivial.

In case $\Sigma=S^1$ is a transversely non-degenerate orbit, then $\mathcal L_{\Sigma}$ is a trivial system precisely when $\Sigma$ is good.
This was verified in \cite{JY_Zhao}, Proposition 6.2. See also \cite{BO:SH_MB}, Lemma 4.29.

For the third case, we already know that $Det(D_\Sigma)|_{U_a}$ is trivial since $U_a$ is contractible.
We will construct a nowhere-vanishing section of $Det(D_\Sigma)$.
To keep notation simple we assume that $H$ equals $f(r)$ near $\Sigma$, where $r$ is the coordinate on the interval component of $]1-\epsilon,1+\epsilon[\times \nu_{\Sigma_c}(\Sigma)$.
Define the metric
$$
g=dr\otimes dr+\alpha\otimes \alpha+d\alpha(\cdot,J|_{\xi}\cdot),
$$
where $\alpha$ is the contact form on $\Sigma_c$, and let $\nabla$ denote the Levi-Civita connection for $g$.
We need to compute the asymptotic operator, given by
$-J(\partial_t-\nabla X_H )$.
Put $S_o:=J\nabla X_H$. We have $X_H=-f'(r)R$, so if we use that $\nabla_R R=0$ for this metric, we see that $S_o$ decomposes as
\[
S_o=
\left(
\begin{array}{ccc}
rf''(r) & 0 & 0 \\
0 & 0 & 0 \\
0 & 0 & S_\xi
\end{array}
\right)
\]
if we choose a frame $r\partial_r, R$ and any frame for $\xi$.
This clearly splits.
Now fix $\gamma\in U_a$.
We extend $S_o$ to the capping plane $u_\gamma^a$ using the cutoff function $\rho$. 
Choose the trivialization $r\partial_r, R$ for the symplectic complement of $\xi$ and any frame for $\xi$.
Then the operator 
$$
D_{\gamma}:=\rho(z)\cdot(\partial_s +J_0\partial_t+S_o)
$$ 
splits. 
Furthermore, this operator is conjugated to the operator $D_\Sigma|_{\gamma}$ by a change of trivialization.
We write $D_{\gamma}=D_\epsilon \oplus D_\xi$.
By assumption, there is a unitary trivialization of the contact structure $\xi$ along $\gamma$ such that the path of symplectic matrices $\psi$ corresponding to the linearized Reeb flow with respect to this trivialization is complex linear.
So we have
$$
\dot \psi=J_0 S_\xi \psi,
$$
where the loop $S_\xi$ commutes with $J_0$ by complex linearity.
This implies that the operator $D_\xi$ is complex linear, i.e.~$D_\xi(J_0\Psi)=J_0 D_\xi\Psi$.
Hence both the kernel and cokernel of $D_\xi$ carry a complex structure induced by $J_0$ and therefore a canonical orientation. Thus the determinant line of $D_\xi$ also carries a canonical orientation.
The operator $D_\epsilon$ is not complex linear, but it still carries a natural orientation (independent of $\gamma$).

It follows that the determinant line of $D_\Sigma|_{\gamma}$ carries a canonical orientation. In particular, it is independent of the chart $U_a$ and hence extends to a nowhere-vanishing section of the orientation line bundle $L$.
We conclude that the local coefficient system $\mathcal L_{\Sigma}$ is trivial.
\end{proof}

\noindent
{\bf Grading shift}
We have seen that the local Floer homology $HF^{loc}(H_\delta;R)$ is isomorphic to $H_*(\Sigma;R)$ up to a degree shift, which we will compute now.
This is also carried out in step 3 of the proof of Proposition~2.2 in \cite{CFHW}.
Instead of their method we find the shift by comparing the degrees of the elements corresponding to the minima of $h$.
Such a computation was done for the transverse Conley-Zehnder index in \cite[Chapter 2]{Bourgeois:thesis}.
See also \cite[Lemma 2.4]{vK:BW} for details.
To get the full index, we use the assumption that the Hessian of $H$ restricted to $\nu_W(\Sigma)$ is positive-definite.
See also \cite{Oancea:survey}, section 3.3, for this computation.
\end{proof}

\begin{remark}
If the boundary of $W$ is a prequantization bundle $(P,\alpha)$ over a symplectic manifold $(Q,\omega)$ with $c_1(Q)=c[\omega]$ for some $c\in \R$, then the assumption of complex linearity of the linearized Reeb flow is satisfied.
Namely, every periodic Reeb orbit of $(P,\alpha)$ is a fiber over a point $q\in Q$. Choose a unitary trivialization of $TQ$ near $q$.
Since the contact structure on $P$ is a horizontal lift of the tangent bundle to $Q$, we obtain a unitary trivialization of $\xi$ along the fiber over $q$. The linearized Reeb flow is the identity for this trivialization, so it is complex linear and $S\equiv 0$.

On the other hand, a hyperbolic periodic Reeb orbit $\Sigma$ in a contact $3$-manifold never satisfies this property, because the return map does not commute with any choice of complex structure $J$.
Still the determinant bundle $Det(\Sigma)$ will be orientable if $\Sigma$ is good by \cite{BO:SH_MB}, Lemma 4.29.
\end{remark}

For general Morse-Bott manifolds, many different local coefficient systems are possible. Let us describe one simple case, namely that $\Sigma$ is a bad Reeb orbit for the flow Hamiltonian of $H$.
Then by \cite{BO:SH_MB}, Lemma 4.29, the local coefficient system is non-trivial. There is only one such system on $S^1$. 
So if we use the isomorphism~\eqref{eq:local_Floer} we end up with a twisted Morse complex computing the local Floer homology,
$$
\Z \stackrel{\cdot 2}{\longrightarrow} \Z,
$$
and we obtain the homology
\[
HF^{loc}_{*+shift}(\Sigma,H_\delta,J;\Z)\cong
\begin{cases}
\Z_2 &  *=0 \\
0 & \text{otherwise.}
\end{cases}
\]

\subsection{Filtering by action}
Since the Floer differential is action decreasing, one can filter the Floer complex by action.
As in Lemma~\ref{lemma:orbits_curves}, choose a Morse function $h$ for the critical manifold $C$ of the Morse-Bott Hamiltonian $H$.
We will assume that all critical values of $H$ are negative and close to $0$, and that other $1$-periodic orbits have action larger than $0$.

First observe that a Floer trajectory of the autonomous Hamiltonian $H$ escaping a neighborhood $\nu(\Sigma)$ of a component $\Sigma$ of the Morse-Bott manifold of $1$-periodic orbits has energy bounded from below.
Since there are only finitely many Morse-Bott submanifolds, we find some minimal $\tilde \delta>0$ such that a Floer trajectory escaping a neighborhood of the critical manifold $C$ has energy larger than $\tilde \delta$.
Now perturb the Hamiltonian $H$ into a time-dependent Hamiltonian $H_\delta$ that is admissible for symplectic homology.
We will take $\delta$ smaller than $\tilde \delta/2$,
so any Floer trajectory escaping $\nu(\Sigma)$ has energy larger than $\delta$, the action difference between $1$-periodic orbits corresponding to the minimum and the maximum of $h$.

Choose a strictly increasing function $a_H:\Z \to \R$ such that the following holds.
\begin{itemize}
\item We set $a_H(0)=0$, and impose for a critical point $x$ of $H$ that $\mathcal A(x)\in ]a_H(-1),a_H(0)]$.

\item if $\Sigma$ is a connected component of the critical set $C$ with action $\mathcal A(\Sigma)$ in the interval $]a_H(p-1),a_H(p)]$, then the $1$-periodic orbits corresponding to the perturbed Hamiltonian $H_\delta$ also have action in this interval. 

We define 
$$
C(p):=\{ \text{connected component }\Sigma \text{ of }C~|~\mathcal A(\Sigma)\in ]a_H(p-1),a_H(p)]\}
.
$$

\item if $p>0$, and $\Sigma_1$ and $\Sigma_2$ are distinct components in $C(p)$, then there are no Floer trajectories between them.
\item $\lim_{p\to \infty}a_H(p)=\infty$.
\end{itemize}
To see that such a function exists, choose both the perturbation parameter $\delta$ and the differences $a_H(p)-a_H(p-1)$ sufficiently small.
As a Floer trajectory connecting different components must have energy larger than $\tilde \delta$, this ensures the second and third condition.

Now introduce a filtration on the whole Floer complex $CF_*(W,H_\delta)$,
$$
F_p CF_{q}(W,H_\delta)=\{ x\in CF_{q}(W,H_\delta)~|~\mathcal A(x)\leq a_H(p) \}
.
$$
This filtration exhausts the complex in finitely many steps by our assumptions on the Hamiltonian $H$ and the condition that $a_H$ goes to $\infty$. 
Let $p_{H}$ denote the minimal value of $p$ such that $F_p CF_q(W,H_\delta)=CF_q(W,H_\delta)$ for all $p>p_H$.
We consider the spectral sequence associated with this filtration where we use the conventions from~\cite[Section 5.4]{Weibel}.
The $E^0$-page of this spectral sequence is given by 
$$
E^0_{pq}=F_p CF_{p+q}(W,
H_\delta)/F_{p-1} CF_{p+q}(W,H_\delta)
$$
The differential on the $E^0$-page only counts Floer trajectories which decrease the action level less than $a_H(p)-a_H(p-1)$, which is small by our assumptions on the function $a_H$.

In particular, for $p>0$, all these Floer trajectories are counted by the local Floer homologies of the critical manifolds.
For $p=0$, the arguments from the proof of Proposition~\ref{prop:local_Floer} still apply and we have
\begin{equation}
\label{eq:hom_filling}
E^1_{0q}\cong H_{q+n}(W,\partial W;R).
\end{equation}
Hence the $E^1$-page is given by the homologies of the critical manifolds with appropriate degree shifts.
The higher differentials follow the familiar recipe from the spectral sequence of the filtration, and are hard to determine explicitly.

Finally, we observe that the filtration is bounded, so the spectral sequence converges by the classical convergence theorem \cite[Theorem 5.5.1, 1.]{Weibel}.
In order to keep track of the grading, we define the shift of a Morse-Bott submanifold $\Sigma$ by
$$
shift(\Sigma)=\mu_{RS}(\Sigma)-\frac{1}{2} \dim \Sigma /S^1,
$$
which is always an integer. We summarize.
\begin{lemma}
\label{lemma:SS_Floer_hom}
Suppose that $(W,\omega=d\lambda)$ is a Liouville domain with a Hamiltonian function $H$ satisfying the following.
\begin{itemize}
\item the symplectic triviality assumption (ST) holds and $c_1(W)=0$.
\item the Hamiltonian $H$ is autonomous, linear at infinity, and its slope is not the period of any periodic Reeb orbit of $\lambda_{\partial W}$.
\item the $1$-periodic orbits of $H$ are of Morse-Bott type.
\item the restriction of the Hessian of $H$ along each critical manifold $\Sigma$ to the Liouville direction is positive definite.
\end{itemize}
Choose a perturbation $H_\delta$ for which all $1$-periodic orbits are non-degenerate and such that its slope is the same as that of $H$.
Then there is a spectral sequence converging to $HF(W,H_\delta;R)$, whose $E^1$-page is given by
\[
E^1_{pq}(HF(W,H_\delta;R) \,)=
\begin{cases}
\bigoplus_{\Sigma \in C(p) } H_{p+q-shift(\Sigma)}(\Sigma;R\otimes_{\Z}\mathcal L_{\Sigma}) & 0<p<p_H
\\
H_{q+n}(W,\partial W;R) & p=0 \\
0 & p<0 \text{ or }p\geq p_H.
\end{cases}
\]
\end{lemma}

\begin{remark}
Note that the value of the function $a_H(p)$ is irrelevant for the above if $a_H(p)$ is larger than the maximal action of a $1$-periodic orbit of $H$.
We will use this observation when defining a spectral sequence for symplectic homology.
\end{remark}

For symplectic homology, we need to understand the continuation maps, which we shall do by choosing a suitable sequence of Hamiltonians.
\begin{theorem}
[Morse-Bott spectral sequence for symplectic homology of periodic flows]
\label{thm:SS_SH}
Let $(W,\omega=d\lambda)$ be a Liouville domain satisfying the following.
\begin{itemize}
\item the symplectic triviality assumption (ST) holds and $c_1(W)=0$.
\item The Reeb flow of $\partial W$ is periodic with periods $T_1,\ldots, T_k$, where $T_k$ is the common period, i.e.~the period of a principal orbit.
\item the linearized Reeb flow is complex linear.
\end{itemize}
Then there is a spectral sequence converging to $SH(W;R)$, whose $E^1$-page is given by
\begin{equation}
\label{eq:SS_SH_prf}
E^1_{pq}(SH)=
\begin{cases}
\bigoplus_{\Sigma \in C(p) } H_{p+q-shift(\Sigma)}(\Sigma;R) & p>0
\\
H_{q+n}(W,\partial W;R) & p=0 \\
0 & p<0.
\end{cases}
\end{equation}
Furthermore, there is a spectral sequence converging to $SH^+(W;R)$ with $E^1$-page
\begin{equation}
\label{eq:SS_SH_+}
E^1_{pq}(SH^+)=
\begin{cases}
\bigoplus_{\Sigma \in C(p) } H_{p+q-shift(\Sigma)}(\Sigma;R) & p>0
\\
0 & p\leq 0.
\end{cases}
\end{equation}
\end{theorem}

\begin{proof}
We construct a complete Liouville manifold $\bar W$ by attaching the positive end of a symplectization, so we have the decomposition $\bar W=W \cup_\partial P\times \R_{\geq 1}$, where $P\times \{1 \}$ denotes the boundary of $W$.
We take a sequence of Hamiltonians $\{ H_T: \bar W \to \R \}$ satisfying the following.
\begin{itemize}
\item $H_T|_{W_{}}$ is $C^2$-small
\item $H_T|_{P\times \R_{\geq 1} }=f_T(t)$, where $f_T$ is an increasing function that is linear at infinity with slope $(T\cdot T_k+\epsilon)$, and that has small slope near $P\times \{ 1\}$.
\item $H_T$ coincides with $H_{T-1}$ on $W_{}$ and on $P\times [1,t_{T-1}]$, where $t_{T-1}$ is such that $H_{T-1}|_{P\times [t_{T-1},\infty[ }$ is a linear function. 
\end{itemize}
Some functions in the sequence of Hamiltonians are sketched in Figure~\ref{fig:Hamiltonians}.
\begin{figure}[htp]
\def\svgwidth{0.85\textwidth}%
\begingroup\endlinechar=-1
\resizebox{0.85\textwidth}{!}{%
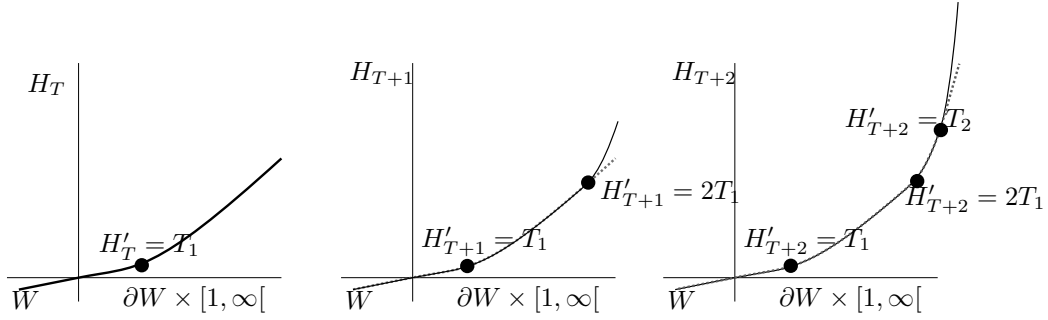%
}\endgroup
\caption{Sequence of Hamiltonians and the location of the Morse-Bott submanifolds}
\label{fig:Hamiltonians}
\end{figure}
For this sequence of Hamiltonians we consider the sequence of spectral sequences from Lemma~\ref{lemma:SS_Floer_hom}, which we denote by $E^r_{pq}(H_T,J_T)$.
Each of these spectral sequences converges to the Floer homology $HF(H_T,J_T)$.

Due to the special choice of Hamiltonians, we can choose a function $a_{H_\infty}$ by the following procedure.
Define $a_{H_1}$ following the procedure before Lemma~\ref{lemma:SS_Floer_hom}, 
We define $a_{H_{T}}$ inductively by the requirements.
\begin{enumerate}
\item Put $a_{H_{T}}(p):=a_{H_{T-1}}(p)$ for $p\leq p_{H_{T-1}}$, where $p_{H_{T-1}}$ is defined as before as the minimal value of $p$ such $F_pCF_*(H_{T-1})=CF_*(H_{T-1})$ for all $p\geq p_{H_{T-1}}$.
\item For $p> p_{H_{T-1}}$ we let $a_{H_T}$ increase so slowly that $F_pCF_*(H_{T})=CF_*(H_{T-1})$ for $p=p_{H_{T-1}},\ldots,2 p_{H_{T-1}}$.
In other words, the first new non-zero column, if any, of $E^0_{p*} (H_T,J_T)$ appears for filtration degree $p$ at least $2 p_{H_{T-1}}+1$.
\end{enumerate}
The function can be extended to $\Z$ using the procedure described before Lemma~\ref{lemma:SS_Floer_hom}.
As a result, all terms in the spectral sequence for $HF(H_T)$ also appear in the same place in the spectral sequences for $HF(H_{T'})$ with $T'>T$. 
Via the above construction we obtain a function $a_{H_\infty}$ by taking the stable value from first requirement.

\begin{figure}
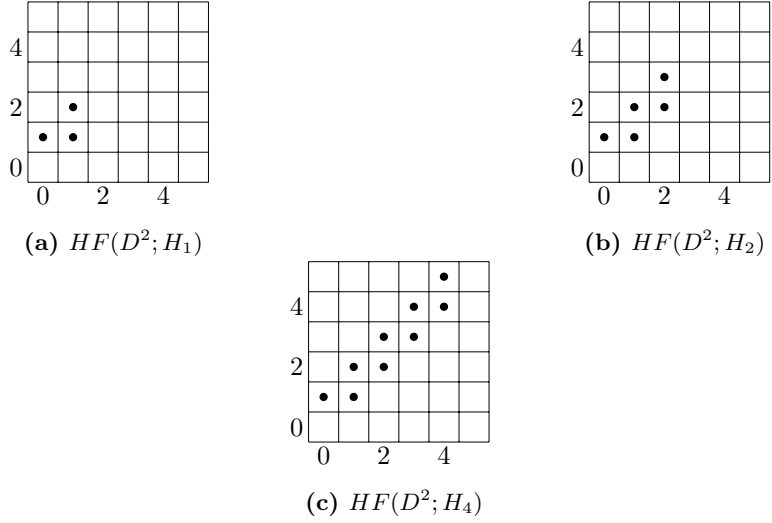

\centering

\begin{subfigure}{.5\textwidth}
  \centering
  \begin{sseq}{0...5}
{0...5}
\ssmoveto 0 1
\ssdropbull

\ssmoveto 1 1
\ssdropbull
\ssmove 0 1
\ssdropbull

\end{sseq}
  \caption{$HF(D^2;H_1)$}
  \label{fig:H1}
\end{subfigure}%
\begin{subfigure}{.5\textwidth}
  \centering
 \begin{sseq}{0...5}
{0...5}
\ssmoveto 0 1
\ssdropbull

\ssmoveto 1 1
\ssdropbull
\ssmove 0 1
\ssdropbull

\ssmoveto 2 2
\ssdropbull
\ssmove 0 1
\ssdropbull

\end{sseq}
  \caption{$HF(D^2;H_2)$}
  \label{fig:H2}
\end{subfigure}
\begin{subfigure}{.5\textwidth}
  \centering
 \begin{sseq}{0...5}
{0...5}
\ssmoveto 0 1
\ssdropbull

\ssmoveto 1 1
\ssdropbull
\ssmove 0 1
\ssdropbull

\ssmoveto 2 2
\ssdropbull
\ssmove 0 1
\ssdropbull

\ssmoveto 3 3
\ssdropbull
\ssmove 0 1
\ssdropbull

\ssmoveto 4 4
\ssdropbull
\ssmove 0 1
\ssdropbull
\end{sseq}
  \caption{$HF(D^2;H_4)$}
  \label{fig:H4}
\end{subfigure}
\caption{$E^1$-pages of Morse-Bott spectral sequences Hamiltonians $H_1$, $H_2$ and $H_4$ leading to a spectral sequence for $SH(D^2)$ (after reindexing to make smaller gaps between the columns)}
\label{fig:SS_SH}
\end{figure}

To construct the Morse-Bott spectral sequence for symplectic homology, we consider for each Hamiltonian $H_T$ the spectral sequence from Lemma~\ref{lemma:SS_Floer_hom}.
The local coefficients are trivial by Lemma~\ref{lemma:local_coefficients_trivial}.
We claim that for $T'>T$ the continuation map $c_{TT'}$ induces a morphism between spectral sequences
$$
c^r_{TT'}:E^r_{pq}(H_T)\longrightarrow E^r_{pq}(H_{T'}).
$$
To see this, first observe that the spectral sequence $E^r_{pq}(H_T)$ is bounded with maximal filtration degree $p_{H_T}$, so it converges not any later than on the $p_{H_T}$-th page.
By our choice of the function $a_{H_{\infty}}$, namely requirement (ii), the gap between the last non-zero column of $E^0_{pq}(H_T)$ and the first new column of $E^0_{pq}(H_{T'})$ is larger than $p_{H_T}$.
Then $c_{TT'}^r$ commutes with the differential for all $r$:
\begin{itemize}
\item for $p\leq p_{H_T}$ this is by construction: the Hamiltonians and complex structures coincide on the relevant sets, and the linear slope between the different regions imply that there are no Floer trajectories escaping the region.
\item for $p>p_{H_T}$, the continuation map $c_{TT'}^r$ vanishes, and the differential $d^{r,T}$ vanishes by requirement (ii) of the above choice for $a_{H_\infty}$.
\end{itemize}
Hence we obtain a directed system of spectral sequences, and we take the direct limit over the slopes $T$.
This gives the spectral sequence~\eqref{eq:SS_SH_prf}.
To see that this spectral sequence converges, we observe that the action filtration is bounded from below and exhausting, so by the classical convergence theorem \cite[Theorem 5.5.1, 2.]{Weibel}, this spectral sequence converges.
\end{proof}

\subsection{Spectral sequence for equivariant symplectic homology}
\label{sec:review_equi_SH}
In principle, the ideas here are very similar to the non-equivariant case, but we need a Morse-Bott setup that is ``Botter'' than before.

We first give a quick review of equivariant symplectic homology defined by Bourgeois-Oancea, \cite{BO:SH_HC}. We will use their notation, and refer to \cite{BO:SH_HC} for details.
Fix a Liouville domain $(W,\omega=d\lambda)$ which we complete by attaching the positive part of a symplectization to $\bar W$.
The idea is to mimic the Borel construction to obtain equivariant symplectic homology. We approximate $ES^1=\varinjlim_{N}S^{2N+1}$ by the non-contractible space $S^{2N+1}$.
The circle $S^1$ acts freely on $S^{2N+1}$, and for larger $N$ more and more homotopy groups vanish.

Consider the approximation of $\Lambda \bar W \times ES^1$ by $\Lambda \bar W \times S^{2N+1}$, which carries a diagonal circle action, namely $g\cdot(\gamma(\cdot),z)=(\gamma(\cdot+g),g\cdot z)$, where $g\cdot z$ is the usual circle action on the Hopf fibration.

In order to get an $S^1$-invariant functional we consider Hamiltonians defined on $H:S^1\times W \times S^{2N+1}\to \R$ that are invariant under the diagonal action on $S^1\times S^{2N+1}$, so $H(t+g,x,g\cdot z)=H(t,x,z)$.
Define the parametrized action functional
\[
\begin{split}
\mathcal A^N: \Lambda \bar W \times S^{2N+1} & \longrightarrow \R \\
(\gamma,z) & \longmapsto -\int_{S^1} \gamma^* \lambda
-\int_{0}^1 H(t,\gamma(t),z)dt.
\end{split}
\]
This is designed to be invariant under the above circle action on $\Lambda \bar W \times S^{2N+1}$.

The critical points of $\mathcal A^N$ are pairs $(\gamma,z_0)$, where
\begin{itemize}
\item $\gamma$ is a $1$-periodic orbit of $H(\cdot,\cdot,z_0)$
\item $\int_{S^1}\frac{\partial H}{\partial z}(t,\gamma(t),z_0)dt=0$.
\end{itemize}
Denote critical points of $\mathcal A^N$ by $\mathcal P(H)$.
Given such a critical point $p=(\gamma,z)$ we get its orbit under the circle action $S^1\cdot p$, which we denote by $S_p$.

To get a version of symplectic homology, we assume that the $S^1$-invariant Hamiltonian satisfies in addition that it is linear at infinity with slope independent of $z\in S^{2N+1}$.
Choose an almost complex structure $J$ with domain $S^1\times \bar W \times S^{2N+1}$ that is $S^1$-invariant, so $J(t+g,x,g\cdot z)=J(t,x,z)$.
This gives a family of $L^2$-metrics on $\Lambda \bar W$ parametrized by $z\in S^{2N+1}$, namely
$$
\langle X,Y \rangle_z:=\int_{S^1} \omega\left( X(t),J(t,\gamma(t),z)Y(t) \right) dt
$$
where $X,Y\in T_\gamma \Lambda \bar W=\Gamma(S^1,\gamma^* T\bar W)$.
Together with an $S^1$-invariant metric on $S^{2N+1}$ we get an $S^1$-invariant metric on $\Lambda \bar W\times S^{2N+1}$.
Take the $L^2$-gradient flow of $\mathcal A^N$ to obtain the \emph{parametrized Floer equations}
\begin{equation}
\label{eq:param_Floer}
\begin{split}
\partial_s u+J(t,u(s,t),z(s)\,)\left( \partial_t u-X^t_{H_{z(s)}}(u(s,t) \right)&=0, \\
\dot z(s)-\int_{S^1}{\vec \nabla}_z H(t,u(s,t),z(s)\, )dt&=0.
\end{split}
\end{equation}
where one requires $(u,z)$ to converge asymptotically to a circle of critical points of $\mathcal A^N$,
\[
\lim_{s\to -\infty}(u(s,t),z(s)\, )\in S^1 \cdot (\bar \gamma(t),\bar \lambda)
\text{ and }
\lim_{s\to \infty}(u(s,t),z(s)\, )\in S^1 \cdot (\ubar \gamma(t),\ubar \lambda).
\]
We write ${\vec \nabla}_z$ to indicate that we are taking the gradient with respect to the $z$-coordinates, and $X^t_{H_{z}}(x)$ denotes the Hamiltonian vector field of the Hamiltonian $H(\cdot,\cdot,z)$ at time $t$ and position $x$.
A $\bar{\phantom{S}}$ indicates an asymptote at the negative puncture, a $\ubar{\phantom{S}}$ indicates an asymptote at the positive puncture.
It is shown in \cite{BO:param}, Theorem A that this equation leads to a Fredholm problem for a generic choice of $S^1$-invariant Hamiltonian.
Furthermore for a suitable choice of Floer data $(H,J,g)$ the linearized operator is surjective leading to moduli spaces $\mathcal M(S_{\bar p},S_{\ubar p};H,J,g)$ that are smooth manifolds.
These moduli spaces carry a free circle action, and the quotient is a smooth manifold of dimension
$$
\dim \mathcal M_{S^1}(S_{\bar p},S_{\ubar p};H,J,g)=-\mu(\bar p)+\mu(\ubar p)-1.
$$

\subsubsection{Coherent orientations}
\label{sec:coherent_S1-equi}
As in the unparametrized Floer theory, these moduli spaces can be given orientations.
Since the asymptotes are a priori not fixed, the methods from \cite{FH} do not directly apply.
However, given an $S^1$-family of critical points $S_p$, where $p=(\gamma,z)$ we can choose trivializations of $\gamma^* T\bar W\oplus T_z S^{2N+1}$ that are invariant under the $S^1$-action on $S_p$.
In this way, the expression for the linearized operator, \cite{BO:param} Formula~2.11, only depends on the orbit $S_p$.
In particular, in this description the asymptotic operator is fixed, so \cite{FH} tells us then that the determinant bundle is trivial.
One then defines a system of coherent orientations on the moduli spaces $\mathcal M_{S^1}(S_{\bar p},S_{\ubar p};H,J,g)$ using the usual scheme.

Now the define the equivariant Floer complex
$$
SC_*^{S^1,N}(H,J,g)=\bigoplus_{S_p} \Z \langle S_p \rangle
$$
with differential
\begin{equation}
\label{eq:equi_Floer_diff}
\partial^{S^1} {\bar S}_p=
\sum_{  \underset{-\mu({\bar S}_p)+\mu({\ubar S}_p)=1}{ {\ubar S}\in Crit {\mathcal A}^N}    }  \sum_{u\in {\mathcal M}_{S^1}(S_{\bar p},S_{\ubar p};H,J,g) } \epsilon([u]) {\ubar S}_p.
\end{equation}
The sign $\epsilon([u])$ is obtained comparing the coherent orientation on $\mathcal M_{S^1}(S_{\bar p},S_{\ubar p};H,J,g)$ with orientation induced by the circle action.
It is proved in the work of Bourgeois-Oancea that $\partial^{S^1}$ defines a differential for generic Floer data $(H,J,g)$ satisfying the earlier assumptions.
Furthermore they also show that there are continuation maps for Hamiltonians with increasing slope, and a direct system of maps induced by the embeddings $S^{2N+1}\to S^{2N+3}$. 

The associated equivariant symplectic homology is defined as the direct limit
\begin{equation}
\label{eq:def_SHS1}
SH^{S^1}(W):=\varinjlim_N \varinjlim_T HF^{S^1,N}(H_{T,N},J_{T,N},g_{T,N} ).
\end{equation}
Here we first take the direct limit over all Hamiltonians that are linear at infinity with increasing slope $T$, and after that we take the direct limit over the the embeddings $S^{2N+1}\to S^{2N+3}$.
We will use the notation $(H_{T,N},J_{T,N},g_{T,N} )$ to indicate the Hamiltonian $H_{T,N}$ and the almost complex structure $J_{T,N}$ on $W$ as well as the metric $g_{T,N}$ on $S^{2N+1}$. We use $T$ to denote the slope of the Hamiltonian at infinity.

As in symplectic homology, one can consider Hamiltonians that are $C^2$-small on $W$ and this leads to a subcomplex $SC^{-,S^1,N}$ and a quotient complex $SC^{+,S^1,N}=SC^{S^1,N}/SC^{-,S^1,N}$.
The corresponding direct limits on homology level are $SH^{-,S^1}(W)$ which was shown to be isomorphic to $H^{S^1}(W,\partial W;R)$ and $SH^{+,S^1}(W)$.

\subsection{Local equivariant Floer homology}
\label{sec:local_equiv_Floer}
The above version of the equivariant Floer differential is action decreasing because it counts $L^2$-gradient flow trajectories.
We can hence filter the equivariant Floer complex by action.

In the following we will take an autonomous Hamiltonian whose $1$-periodic orbits form Morse-Bott manifolds in the old sense, which we denote by $\Sigma$.
We also get corresponding Morse-Bott manifolds of critical points in the equivariant setup, namely $\Sigma\times S^{2N+1}$.

We can again take suitable perturbations of the Hamiltonian to go to a non-degenerate setup.
Hence we need to define and study the local equivariant Floer homology of $\Sigma$. Fix a ring $R$.
\begin{itemize}
\item choose $N\in \N$: we obtain the above free circle action on $\Lambda \bar W\times S^{2N+1}$.
\item find perturbations $H_\delta$, $J_\delta$ and $g_\delta$.
\item define $CF^{loc,S^1,N}_*(\nu(\Sigma),H_\delta$, $J_\delta,g_\delta)$ as the $R$-module freely generated by $S^1$-orbits of critical points $S^1\cdot (\gamma_\delta,z_\delta)$ of the parametrized action functional.
By choosing the perturbations sufficiently small and applying an Arzela-Ascoli argument as in Lemma~\ref{lemma:orbits_curves} one can check that $(\gamma_\delta,z_\delta)$ converges to a point $(\gamma,z)$ in the Morse-Bott manifold $\Sigma\times S^{2N+1}$ as $\delta$ converges to $0$.

\item define the $N$-th approximation of the equivariant local ``moduli space'' as the set
\[
\begin{split}
\mathcal M_{S^1}^{\nu(\Sigma)}(S_{\bar p},S_{\ubar p};H_\delta,J_\delta,g_\delta):=
\{ (u,z) \in \mathcal M_{S^1}(S_{\bar p},S_{\ubar p};H_\delta,J_\delta,g_\delta)~|~\im(u) \subset \nu(\Sigma)
\}
.
\end{split}
\]
\item for the differential, count those parametrized Floer trajectories that are completely contained in $\nu(\Sigma)\times S^{2N+1}$ with sign $\epsilon([u])$ defined in Section~\ref{sec:coherent_S1-equi},
$$
\partial^{loc,S^1}{\bar S}_p=
\sum_{  \underset{-\mu({\bar S}_p)+\mu({\ubar S}_p)=1}{ {\ubar S}_p\in Crit {\mathcal A}^N}    }  \sum_{u\in {\mathcal M}_{S^1}^{\nu(\Sigma)}(S_{\bar p},S_{\ubar p};H,J,g) } \epsilon([u]) {\ubar S}_p.
$$

\item define the $N$-th approximation of the local equivariant Floer homology of $\nu(\Sigma)$ as
$$
HF^{loc,S^1,N}_*(\nu(\Sigma);H_\delta,J_\delta,g_\delta):=
H_*(CF^{loc,S^1,N}_*(\nu(\Sigma),H_\delta,J_\delta,g_\delta)\, ).
$$
As in Proposition~\ref{prop:local_Floer} there is an isomorphism.
\begin{equation}
\label{eq:equivariant_identification}
\begin{split}
HF^{loc,S^1,N}_{*+shift(\Sigma) }(\nu_{\Sigma_c}(\Sigma),H_\delta,J_\delta,g_\delta;R) 
\cong H^{Morse}_*(\Sigma\times_{S^1}S^{2N+1},\tilde K;\tilde {\mathcal L}_{\Sigma,N}\otimes_{\Z}R).
\end{split}
\end{equation}
This isomorphism comes from the unwrapping procedure in the first step of the proof of Proposition~\ref{prop:local_Floer} and Po\'zniak's correspondence applied to \eqref{eq:param_Floer}.

Here the local coefficients are obtained as follows.
As in the proof of Proposition~\ref{prop:local_Floer} we define a real line bundle $L=Det(\Sigma\times S^{2N+1})$ over the critical manifold $\Sigma\times S^{2N+1}$.

Fix a critical point $(\gamma,z)$, so $\gamma$ is a $1$-periodic orbit in $\Sigma$ and $z\in S^{2N+1}$.
As in Section~\ref{sec:coherent_S1-equi}, choose trivializations of $\gamma^* T\bar W\oplus T_z S^{2N+1}$ that are invariant under the circle action.
The expression for the linearized operator for the parametrized Floer equation, \cite{BO:param} Formula~2.11, only depends on the orbit of the $(\gamma,z)$ with such an $S^1$-family of trivializations, so the restriction of $L$ to $S_{(\gamma,z)}$ is trivial.

We hence get a well-defined line bundle $\tilde L$ on the quotient space $\Sigma \times_{S^1} S^{2N+1}$.
We denote the corresponding local coefficient system on $\Sigma \times_{S^1} S^{2N+1}$ by $\tilde{\mathcal L}_{\Sigma,N} $.

\end{itemize}

\begin{lemma}
\label{lemma:local_equi_Floer}
Suppose that one of the following conditions hold: $H^1(\Sigma;\Z_2)=0$, $\Sigma=S^1$ is a good Reeb orbit, or the linearized Reeb flow is complex linear.
Then
\[
HF^{loc,S^1,N}_{*+shift(\Sigma) }(\nu_{\Sigma_c}(\Sigma),H_\delta,J_\delta,g_\delta;R) 
\cong
H_*(\Sigma \times_{S^1} S^{2N+1};R).
\]
\end{lemma}
We only need to show that the local coefficient system $\tilde{\mathcal L}_{\Sigma,N} $ in Equation~\eqref{eq:equivariant_identification} is trivial.
For this, we have a simple lemma.
\begin{lemma}
\label{lemma:Det_bundle_relation}
The real line bundle $\tilde L \to \Sigma \times_{S^1} S^{2N+1}$ is trivial if and only if $L\to \Sigma\times S^{2N+1}$ is trivial.
\end{lemma}
\begin{proof}
To see why this is true, suppose that $\sigma$ is a nowhere vanishing section of $L$.
Because the expression for the linearized operator for the parametrized Floer equation only depends on the orbit $S_{(\gamma,z)}$ by the above, we can assume that $\sigma$ is $S^1$-invariant, and thus we obtain a nowhere vanishing section $\tilde \sigma$ of $\tilde L$ by putting $\tilde \sigma([v,z])=\sigma(v,z)$ which is then well-defined.
The converse is obtained by the same observation.
\end{proof}

To prove Lemma~\ref{lemma:local_equi_Floer} there are three cases to consider.
The first case is obvious.
In the last case the linearized Reeb flow is complex linear with respect to some unitary trivialization, and we can argue as in the proof of Lemma~\ref{lemma:local_coefficients_trivial}, to show that $L=Det(\Sigma \times S^{2N+1})$ is trivial. By Lemma~\ref{lemma:Det_bundle_relation}, $\tilde{\mathcal L}_{\Sigma,N}$ is then trivial.
We explain the second case below with a computation.

First we make a couple of observations to clarify the situation when the determinant bundles are non-trivial.
We will need this for the Appendix C.
First of all, a circle fiber of the circle bundle $\Sigma \times S^{2N+1}\to \Sigma \times_{S^1} S^{2N+1}$ is not necessarily homotopic to a \emph{simple} circle fiber of $\Sigma$ (seen as a periodic orbit).
We explain this in the case when $\Sigma$ is a non-degenerate Reeb orbit that is a $k$-fold cover of a simple Reeb orbit $\Sigma'$.
In that case the $S^1$-action on $\Sigma\cong S^1$ is given by $(g,\phi)\mapsto g^k \phi$, so the space $\Sigma\times_{S^1} S^{2N+1}$ is by defined the quotient space
$$
\Sigma \times S^{2N+1}/(\phi,z)\sim (g^k\phi,z\cdot g).
$$
We identify this manifold with the lens space $L^{2N+1}(k)$ by the associated bundle construction for the principal circle bundle $S^{2N+1}$.

If $\Sigma$ is good, we claim that the local coefficient system is trivial.
Indeed, along a loop of the form $S^1\times \{ z \}$ the bundle $L$ is trivial by \cite{BO:SH_MB}, Lemma 4.29.
Hence $L$ is trivial over $\Sigma\times S^{2N+1}$, so by Lemma~\ref{lemma:Det_bundle_relation} the local coefficient system $\tilde{\mathcal L}_{\Sigma,N}$ is trivial, which completes the proof of Lemma~\ref{lemma:local_equi_Floer}.
We can compute the Morse homology with the cell complex
$$
\underset{*=2N+1}{\Z}  \stackrel{\cdot 0}{\longrightarrow} \underset{*=2N}{\Z} \stackrel{\cdot k}{\longrightarrow}
\ldots
\underset{*=2}{\Z}  \stackrel{\cdot k}{\longrightarrow} \underset{*=1}{\Z} \stackrel{\cdot 0}{\longrightarrow} \underset{*=0}{\Z}.
$$
whose homology is given by
\begin{equation}
\label{eq:homology_lens_space}
H_*(S^1\times_{S^1} S^{2N+1};\Z)\cong H_*(L^{2N+1}(k);\Z)
\cong
\begin{cases}
\Z & *=0,2N+1\\
\Z_k & * \text{ odd}\\
0 & \text{otherwise.}
\end{cases}
\end{equation}

If $\Sigma$ is bad, then $k$ is even and we claim that the local coefficient system is non-trivial.
Again, consider a loop of the form $S^1\times \{ z \}$.
The bundle $L$ along this loop is non-trivial by \cite{BO:SH_MB}, Lemma 4.29.
So the above lemma tells us that $\tilde{\mathcal L}_{\Sigma,N}$ is non-trivial. 
There is only one such non-trivial local coefficient system as $H^1(L^{2N+1}(k);\Z_2)\cong \Z_2$.
We use cellular homology with local coefficients to compute the homology.
This yields the complex.
$$
\underset{*=2N+1}{\Z}  \stackrel{\cdot k}{\longrightarrow} \underset{*=2N}{\Z} \stackrel{\cdot 0}{\longrightarrow}
\ldots
\underset{*=2}{\Z}  \stackrel{\cdot 0}{\longrightarrow} \underset{*=1}{\Z} \stackrel{\cdot k}{\longrightarrow} \underset{*=0}{\Z}
$$
and find the homology
\[
H_*(S^1\times_{S^1} S^{2N+1};\tilde{\mathcal L}_{\Sigma,N} )\cong H_*(L^{2N+1}(k);\tilde{\mathcal L}_{\Sigma,N} )
\cong
\begin{cases}
\Z_k & *=0,2,\ldots,2N\\
0 & \text{otherwise}.
\end{cases}
\]
We note that bad orbits do not give a contribution to local equivariant Floer homology if we use $\Q$-coefficients. See \cite{Gutt:thesis}, Corollary~2.2.5 and the discussion before that for another argument.
To make the connection with equivariant homology, observe that the Gysin sequence for circle bundles implies that for $k<N$
\begin{equation}
\label{eq:equivariant_hom_approx}
H_k(\Sigma\times_{S^1} S^{2N+1};R) \cong H_k^{S^1}(\Sigma;R).
\end{equation}

With the identifications from~\eqref{eq:equivariant_identification}, we obtain the following analogue of Lemma~\ref{lemma:SS_Floer_hom}.
\begin{lemma}
\label{lemma:SS_equi_Floer_hom}
Assume that the conditions of Lemma~\ref{lemma:SS_Floer_hom} hold and replace the pair $(H,J)$ by $(H_{T,N},J_{T,N},g_{T,N})$ defined on $S^1\times \bar W \times S^{2N+1}$.
Then there is a spectral sequence converging to $HF^{S^1,N}(W,H_{T,N}^\delta,J_{T,N}^\delta,g_{T,N}^\delta;R)$ whose $E^1$-page is given by
\[
E^1_{pq}(HF^{S^1,N} \,)=
\begin{cases}
\bigoplus_{\Sigma \in C(p) } H_{p+q-shift(\Sigma)}(\Sigma\times_{S^1} S^{2N+1};R\otimes_{\Z} \tilde{\mathcal L}_{\Sigma,N} ) & 0<p<p_H
\\
H_{q+n}( (W,\partial W)\times_{S^1} S^{2N+1};R) & p=0 \\
0 & p<0 \text{ or }p\geq p_H.
\end{cases}
\]
Similarly, there is a spectral sequence converging to the $+$-part of equivariant Floer homology, whose $E^1$-page is given by
\begin{equation}
\label{eq:SS_+_equi_Floer}
E^1_{pq}(HF^{+,S^1,N} \,)=
\begin{cases}
\bigoplus_{\Sigma \in C(p) } H_{p+q-shift(\Sigma)}^{S^1}(\Sigma\times S^{2N+1};R\otimes_{\Z} \tilde{\mathcal L}_{\Sigma,N} ) & 0<p<p_H
\\
0 & p\geq 0.
\end{cases}
\end{equation}
\end{lemma}
\begin{proof}
The proof follows the one of Lemma~\ref{lemma:SS_Floer_hom} verbatim if one use the identifications from Equation~\eqref{eq:equivariant_identification}.
\end{proof}

\begin{theorem}[Morse-Bott spectral sequence for periodic flows]
\label{thm:SS_equivariant_homology}
Let $(W,\omega=d\lambda)$ be a Liouville domain satisfying the following.
\begin{itemize}
\item the symplectic triviality assumption (ST) holds and $c_1(W)=0$.
\item The Reeb flow of $W$ is periodic with minimal periods $T_1,\ldots, T_k$, where $T_k$ is the common period, i.e.~the period of a principal orbit.
\item The linearized Reeb flow is complex linear with respect to some unitary trivialization of the contact structure.
\end{itemize}
Then there is a spectral sequence converging to $SH^{S^1}(W;R)$, whose $E^1$-page is given by
\[
E^1_{pq}(SH^{S^1})=
\begin{cases}
\bigoplus_{\Sigma \in C(p) } H_{p+q-shift(\Sigma)}^{S^1}(\Sigma;R) & p>0 \\
H_{q+n}^{S^1}(W,\partial W;R) & p=0 \\
0 & p<0.
\end{cases}
\]
Furthermore, there is also a spectral sequence converging to $SH^{+,S^1}(W;R)$. Its $E^1$-page is given by
\begin{equation}
\label{eq:MB_SS_SH+S1}
E^1_{pq}(SH^{+,S^1})=
\begin{cases}
\bigoplus_{\Sigma \in C(p) } H_{p+q-shift(\Sigma)}^{S^1}(\Sigma;R) & p>0 \\
0 & p\leq 0.
\end{cases}
\end{equation}
If in addition, $\pi_1(W)=0$ and $\pi_1(\partial W)=0$, then the mean Euler characteristic of $SH^{+,S^1}$ is an invariant of the contact manifold $(\partial W,\lambda|_{\partial W})$ and it can be computed with formula~\eqref{eq:MEC_general}.
\end{theorem}
We point out that a similar spectral sequence for a Liouville domain with non-degenerate periodic Reeb orbits on the boundary was found by Gutt, \cite{Gutt:thesis,Gutt:invariant}.

\begin{proof}
The arguments are similar to those of the proof of Theorem~\ref{thm:SS_SH}.
Equivariant symplectic homology is defined by taking two direct limits, first over the slopes $T$ and then over the dimension of the spheres $S^{2N+1}$, i.e.~
$$
SH^{S^1}(W)=\varinjlim_N \varinjlim_T HF^{S^1,N}(H_{T,N},J_{T,N},g_{T,N} ).
$$
For the equivariant Floer homology groups $HF^{S^1,N}(H_{T,N},J_{T,N},g_{T,N} )$ we have the spectral sequence from Lemma~\ref{lemma:SS_equi_Floer_hom}.
The local coefficients are trivial by Lemma~\ref{lemma:local_equi_Floer}.
We can argue as in the proof of Theorem~\ref{thm:SS_SH} with a special choice of Hamiltonians to obtain a directed system of spectral sequences.
Taking the direct limits over $N$ and over the slopes $T$ gives the above spectral sequence by noting that equivariant homology of $\Sigma$ can be defined as $H^{S^1}(\Sigma;R):=\varinjlim_N H(\Sigma\times_{S^1}S^{2N+1};R)$.
Since the filtration is bounded from below and exhausting, convergence from the classical convergence theorem.
\end{proof}

\subsection{Morse-Bott spectral sequences obtained by index filtrations}
\label{sec:index_filtration}
In some cases we can also filter by index to produce a spectral sequence.
Such an argument needs more analysis for the Morse-Bott setup; in particular one needs transversality results.
This has not yet been published for general Morse-Bott manifolds, but a complete argument has been given by Bourgeois and Oancea, \cite{BO:SH_MB}, for critical manifolds that are circles.
On a perturbed Floer complex one can hence consider the filtration
$$
F_p CF_*(W,\tilde H)=\{ \hat x\in CF_*(W,\tilde H)~|~\mu(x)-\frac{1}{2}\dim \Sigma_x /S^1\leq p \}
$$
where $\hat x$ comes from a $1$-periodic orbit $x$ of the unperturbed problem. The Floer differential respects this filtration if one has sufficient transversality.
Such an argument was used in an equivariant setup to obtain a spectral sequence for $SH^{+,S^1}$ in the proof of Proposition~3.7, part (II) of \cite{BO:SH_HC}.
This index filtration is useful, since it excludes other differentials than the action filtration.

\section{Appendix C: Computing the mean Euler characteristic and its invariance}
\label{sec:AppendixC}
\begin{proof}[ of Lemma~\ref{lemma:mec_invariant}: the mean Euler characteristic as an invariant]
Denote the boundary of $W$ by $P$.
We have convenient dynamics, so there is a sequence of contact forms $\alpha_T=f_T \alpha$.

Define the completion $\bar W=W \cup_{\partial} P\times [1,\infty [$, and construct a sequence of Hamiltonians $H_T:\bar W \to \R$ with the following properties.
\begin{enumerate}
\item $H_T$ is $C^2$-small on $W$.
\item the restriction of $H_T$ to the set $(P\times [1,\infty[,d( t \alpha_T )\, )$ only depends on the interval coordinate $t$.
\item $H_T$ is linear at infinity for the contact form $\alpha_T$ with slope $T$.
With respect to the standard coordinates on the symplectization $(P\times [1,\infty[,d( t \alpha )\,)$ this means that $H_T=T\cdot f_T \cdot t+c_T$ for some constant $c_T$. 
\item $H_T\geq H_{T-1}$.
\end{enumerate}
The third condition of convenient dynamics guarantees that the last condition can be realized.

Under the assumptions of the lemma, the Conley-Zehnder index is independent of the choice of capping disk, and we will choose the capping disks to lie in the region $P\times [1,\infty[$, which can be done by the assumption that $\pi_1(P)=0$.

We fix a degree window $[-M_0,M_0]$ in which we want to compute equivariant symplectic homology with $\Q$-coefficients.
Fix a slope $T$, and consider the spectral sequence~\eqref{eq:SS_+_equi_Floer} for the Floer data $(H_{T,N},J_{T,N},g_{T,N})$.
We will write $E^r_{pq}(H_{T,N})$ for this spectral sequence.
All elements in $E^r_{pq}(H_{T,N})$ with total degree in the degree window $[-M_0,M_0]$ lie in a skew-diagonal band having width $2M_0+1$, depicted in Figure~\ref{fig:deg_window}.
Note that the $E^1$-page, $E^1_{pq}(H_{T,N})$, can be computed with Equation~\eqref{eq:homology_lens_space}: we only need to take those periodic orbits with action less than $T$ that are good.

\begin{figure}
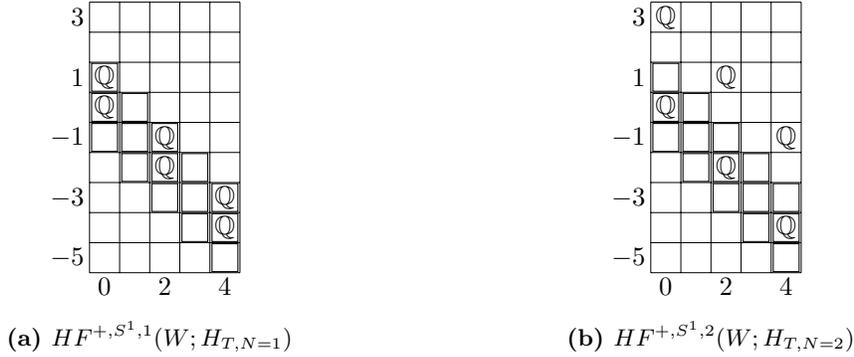

\centering

\begin{subfigure}{.5\textwidth}
  \centering
  \begin{sseq}{0...4}
{-5...3}
\ssmoveto 0 {-1}
\ssdropboxed{\phantom{\Q}}
\ssmove 0 1
\ssdropboxed{\Q}
\ssmove 0 1
\ssdropboxed{\Q}

\ssmoveto 1 {-2}
\ssdropboxed{\phantom{\Q}}
\ssmove 0 1
\ssdropboxed{\phantom{\Q}}
\ssmove 0 1
\ssdropboxed{\phantom{\Q}}

\ssmoveto 2 {-3}
\ssdropboxed{\phantom{\Q}}
\ssmove 0 1
\ssdropboxed{\Q}
\ssmove 0 1
\ssdropboxed{\Q}

\ssmoveto 3 {-4}
\ssdropboxed{\phantom{\Q}}
\ssmove 0 1
\ssdropboxed{\phantom{\Q}}
\ssmove 0 1
\ssdropboxed{\phantom{\Q}}

\ssmoveto 4 {-5}
\ssdropboxed{\phantom{\Q}}
\ssmove 0 1
\ssdropboxed{\Q}
\ssmove 0 1
\ssdropboxed{\Q}

\end{sseq}
  \caption{$HF^{+,S^1,1}(W;H_{T,N=1})$}
  \label{fig:unwanted1}
\end{subfigure}%
\begin{subfigure}{.5\textwidth}
  \centering
  \begin{sseq}{0...4}
{-5...3}
\ssmoveto 0 {-1}
\ssdropboxed{\phantom{\Q}}
\ssmove 0 1
\ssdropboxed{\Q}
\ssmove 0 1
\ssdropboxed{\phantom{\Q}}
\ssmove 0 2
\ssdrop{\Q}

\ssmoveto 1 {-2}
\ssdropboxed{\phantom{\Q}}
\ssmove 0 1
\ssdropboxed{\phantom{\Q}}
\ssmove 0 1
\ssdropboxed{\phantom{\Q}}

\ssmoveto 2 {-3}
\ssdropboxed{\phantom{\Q}}
\ssmove 0 1
\ssdropboxed{\Q}
\ssmove 0 1
\ssdropboxed{\phantom{\Q}}
\ssmove 0 2
\ssdrop{\Q}

\ssmoveto 3 {-4}
\ssdropboxed{\phantom{\Q}}
\ssmove 0 1
\ssdropboxed{\phantom{\Q}}
\ssmove 0 1
\ssdropboxed{\phantom{\Q}}

\ssmoveto 4 {-5}
\ssdropboxed{\phantom{\Q}}
\ssmove 0 1
\ssdropboxed{\Q}
\ssmove 0 1
\ssdropboxed{\phantom{\Q}}
\ssmove 0 2
\ssdrop{\Q}

\end{sseq}
  \caption{$HF^{+,S^1,2}(W;H_{T,N=2})$}
  \label{fig:unwanted2}
\end{subfigure}%
\caption{Choosing a degree window to avoid unwanted generators}
\label{fig:deg_window}
\end{figure}

As before, let $p_{H_T}$ denote the $p$-index of the last non-zero column of $E^0_{pq}(H_{T,N})$.
Take $p\leq p_{H_T}$, and let $q_p$ be the smallest integer such that $E^1_{pq_p}\neq 0$ if the column $E^1_{p*}$ does not vanish (put $q_p:=0$ otherwise). Such a $q_p$ is bounded from below and does not depend on $N$.
Now choose $N_0(T)$ such that $p+q_p+N_0(T)>M_0$ for $p=0,\ldots,p_{H_T}$.
With equation~\eqref{eq:equivariant_hom_approx} we see that the entries in the spectral sequence with total degree in the degree window $[-M_0,M_0]$ do not depend on $N$ provided $N\geq N_0(T)$.
This can be done for any slope, leading to an increasing sequence $\{ N_0(T) \}_T$.

We claim that the Euler characteristic in the degree window $[-M_0,M_0]$, defined by
$$
\chi_{[-M_0,M_0]}(SH^{+,S^1,N}(H_{T,N})\, ):=
\sum_{k=-M_0}^{M_0} (-1)^k \rk SH^{+,S^1,N}_k(H_{T,N})
$$
is independent of $T$ and $N$ provided $T$ is sufficiently large and $N>N_0(T)$.
To see this, choose $T>\frac{M+2n}{\Delta_m}T_0$.
By the second condition of convenient dynamics, those generators in the spectral sequence with total degree in the degree window $[-M_0-2n,M_0+2n]$ are given by covers of $\gamma_1^T,\ldots,\gamma_k^T$ up to a covering number bounded by $\frac{M+2n}{\Delta_m}$.
By condition (i) of convenient dynamics, these orbits are non-degenerate.
Furthermore, they have that property for all $T>\frac{M+2n}{\Delta_m}T_0$ (varying $T$ continuously), so their transverse Conley-Zehnder index does not depend on $T$ by the homotopy property of the Conley-Zehnder index.
This implies that those terms of the $E^1$-page with total degree in $[-M_0-2n,M_0+2n]$ form a $\Q$-vector space that independent of $T$ and $N$, provided these are chosen sufficiently large.
The Euler characteristic in the degree window $[-M_0,M_0]$ does not depend on differential, and is hence independent of $T$ and $N$ as well if $T$ and $N$ are sufficiently large.

Since equivariant symplectic homology is invariant under exact symplectomorphisms, \cite{Gutt:thesis} Section~3.2, it is independent of the choice of sequence of Hamiltonians (or contact forms) defining it, so in particular $\chi_{[-M_0,M_0]}(SH^{+,S^1}(\bar W)\, )$ is an invariant of the symplectic manifold $\bar W$.

To see that the mean Euler characteristic is an invariant of the contact structure rather than the symplectic filling, we claim that given another filling $W'$ we find an isomorphism between the bigraded vector spaces corresponding to the $E^1$-page \eqref{eq:SS_+_equi_Floer} with total degree in$[-M_0,M_0]$,
$$
\bigoplus_{p+q\in [-M_0,M_0]}E^1_{pq}(SH^{+,S^1,N}(\bar W,H_T)\, ) \stackrel{\cong}{\longrightarrow} \bigoplus_{p+q\in [-M_0,M_0]}E^1_{pq}(SH^{+,S^1,N}(\bar W',H_T')\, ).
$$
To see why, note that we can choose a sequence of Hamiltonians $H_T'$ on $\bar W'$ that coincide with the Hamiltonians $H_T$ on the set $P\times [1,\infty [$.
This isomorphism does not respect the differential in general, but the Euler characteristics of these bigraded vectors spaces coincide up to an error term that is bounded by a constant which is independent of $M_0$, $N$ and $T$ provided $N$ and $T$ are sufficiently large.
It follows that $\chi_m$ does not depend on the filling $W$, provided $W$ satisfies the assumptions of the lemma.
\end{proof}

\begin{remark}
\label{rem:why_convenient_dynamics}
Even though the conditions of convenient dynamics may be relaxed somewhat, we briefly explain why convenient dynamics seem to be necessary:
\begin{enumerate}
\item dropping the first condition altogether allows for prequantization bundles over Calabi-Yau manifolds. As a simple explicit example, we consider $\Sigma(4,4,4,4)$ which is a circle bundle over the quartic $K3$-surface.
With the Morse-Bott spectral sequence~\eqref{eq:MB_SS_SH+S1}  we can verify that $SH^{+,S^1}_0(W)$ is infinite-dimensional, so $\chi_m(W)$ is not defined.

\item a fixed number of essential orbits implies linear growth of the Floer chain complexes as a function of $T$, and that is needed to ensure that the mean Euler characteristic can be computed already on chain level. Although there may be cancellations on homology level such that the mean Euler characteristic is still defined, i.e.~the relevant limits exist, it is not obvious why it is still an invariant.
One may put a different weight in the definition to have a more general notion.

\item The third condition is a technical condition, used to get a sequence of increasing Hamiltonians.
\end{enumerate}
\end{remark}

\begin{proof}[of Lemma~\ref{lemma:periodic_flow_convenient}]
Let $(P,\alpha)$ be a contact manifold as in the lemma.
Wadsley's theorem, \cite{Wadsley}, tells us that a geodesible flow for which all orbits are periodic induces a circle action.
The Reeb flow is geodesible for any metric of the form $\alpha\otimes\alpha+d\alpha(\cdot, J\cdot)$,
where $J$ is a compatible complex structure for $\xi=\ker  \alpha$.
Hence we see that the Reeb flow induces a circle action.
By compactness we find minimal periods $T_1<\ldots<T_k$, where $T_k$ is the period of a principal orbit.
Define 
$$
\Sigma_{T_i}:=\{ x\in P ~|~Fl^R_{T_i}(x)=x \}
.
$$
Following the procedure from Bourgeois' thesis, \cite{Bourgeois:thesis} Chapter 2, we define Morse functions on the orbit spaces, which are in general symplectic orbifolds. 
We use the following notation. 
An orbit space $\Sigma_{T_i}/S^1$ is denoted by $Q_{T_i}$.
The steps of this procedure are as follows:
\begin{enumerate}
\item Consider $Q_{T_1}$, which is a symplectic \emph{manifold}. Choose a positive Morse function $g_{T_1}$.
\item Go to $T_{i+1}$. If $\Sigma_{T_j}\subset \Sigma_{T_{i+1}}$ for some $T_j<T_{i+1}$, then extend the Morse function $g_{T_j}$ to a positive Morse function on $Q_{T_{i+1}}$ such that $Hess g_{T_{i+1}} |_{Q_{T_j}}$ is positive definite. If not choose an independent positive Morse function $g_{T_{i+1}}$.
\item Repeat this procedure for all $T_i$.
\end{enumerate}
For us the point of this procedure is to have a nice formula for the Conley-Zehnder index of a perturbed Reeb flow we will define later.
We end up with a Morse function $g:=g_{T_k}$ on $Q_{\Sigma_k}$, which we extend to an $S^1$-invariant function $\bar g$.

Given $T$ define the contact form
$$
\alpha_T=(1+\frac{\epsilon_T \bar g}{T})\alpha.
$$
As pointed out in Bourgeois' thesis \cite{Bourgeois:thesis}, chapter 2, the $S^1$-fiber over a critical point of $g$ is a periodic Reeb orbit.
Moreover, such a Reeb orbit is non-degenerate provided we choose $\epsilon_T$ sufficiently small, and by an Arzela-Ascoli argument it follows that all Reeb orbits with period less than $T$ correspond to such an $S^1$-fiber or a multiple cover thereof.
By compactness of $P$ we can assume that $\epsilon_T \bar g<1/2$.

We verify now that these $\alpha_T$ satisfy the required condition of convenient dynamics.
\begin{enumerate}
\item The above argument shows that all periodic orbits of $\alpha_T$ with period less than $T$ are non-degenerate. 
Using Bourgeois' thesis \cite{Bourgeois:thesis}, chapter 2, we compute the Conley-Zehnder index of all periodic Reeb orbits with period less than $T$.
A detailed computation for prequantization bundles can also be found in \cite[Lemma 2.4]{vK:BW}.
Together with the assumption that the mean index of a principal orbit is not equal to $0$ and the fact that there are only finitely many orbit spaces we see that there is $\Delta_m>0$, independent of $T$, so for all periodic Reeb orbits $\gamma$ of $\alpha_T$ with period less than $T$ the mean index satisfies 
$$
{|\Delta(\gamma)|}>\Delta_m.
$$
\item the fixed number of essential orbits needed for condition (ii) of convenient dynamics is equal to the number of critical points of the Morse function $f$ on the orbit space of the principal orbits.
By construction every periodic Reeb orbit with period less than $T$ is a cover of an $S^1$-fiber over a critical point of $f$.
\item We directly verify the inequality for $f_T=(1+\frac{\epsilon_T \bar g}{T})$.
We have $T\cdot (1+\frac{\epsilon_T \bar g}{T})>T$, and
$(T-1)\cdot (1+\frac{\epsilon_{T-1} \bar g}{{T-1}})<(T-1)+{1/2}<T-1/2$,
so the third condition for convenient dynamics holds as well.
\end{enumerate}
\end{proof}

The idea of the proof of next lemma can be summarized as ``thinning the handle''. This idea was used by Cieliebak in \cite{Cieliebak} to show that symplectic homology does not change under subcritical surgery.
\begin{proof}[of Lemma~\ref{lemma:subcritical_preserves_convenient_dynamics}]
Choose a family of contact forms $\alpha_T$ on $\Sigma=\partial W$ that satisfies the conditions of convenient dynamics.
By a transversality argument for a submanifold of high codimension, namely the embedding of the isotropic sphere $S:=i(S^k)$, we can assume that there are no Reeb chords from $S$ to itself.

Now fix $T>0$.
We will perform surgery on a neighborhood $\nu^T(S,\epsilon)\cong S\times \R^{2n-1-k}$ (recall that the contactomorphism type  of the surgered manifold depends in general on the framing $\epsilon$). We will specify the size of the neighborhood $\nu^T(S,\epsilon)$ later. 
The goal is to construct a family $\tilde \alpha_T$ of contact forms on the surgered manifold $\tilde \Sigma$ that satisfies the conditions of convenient dynamics.

We claim that there is the neighborhood $\nu^T(S,\epsilon)$ such that all periodic Reeb orbits with action $\mathcal A_{\tilde \alpha_T}<T$ are either
\begin{itemize}
\item essential periodic Reeb orbits of $\alpha_T$
\item so-called Lyapunov orbits in the middle of the handle.
\end{itemize}
To see that this can be done we first observe that since we perform the handle attachment away from periodic orbits of $\alpha_T$, the essential periodic Reeb orbits of $\alpha_T$ are unaffected.
Furthermore, the middle of the handle contains Lyapunov orbits coming from the Lyapunov center theorem, see Ustilovsky's thesis, \cite{Ustilovsky:thesis} Chapter 5, for a description of them in the Weinstein model for the connected sum. 
To see that there are no other periodic orbits with action less than $T$ we use the fact that there are no Reeb chords from $S$ to itself.
We can therefore choose $\nu^T(S,\epsilon)$ so small that the return time of the $\alpha_T$-Reeb flow of $\nu^T(S,\epsilon)$ is larger than $T$.
This verifies condition (ii) of convenient dynamics.

By the convenient dynamics assumption on $\Sigma$, the mean indices of the essential periodic Reeb orbits of $\alpha_T$ remain the same after surgery.
Furthermore, for subcritical surgery, the mean index of any of the Lyapunov orbits can be chosen to be larger than $1$, see \cite{Ustilovsky:thesis} Chapter 5.
This verifies condition (ii) of convenient dynamics.

For condition (iii) we use the handle attachment model of Weinstein, \cite{Weinstein:surgery}.
We will verify the stronger claim that on the handle $\tilde f_T\geq \tilde f_{T-1}$.
Let $H_T$ denote the family of Hamiltonians as constructed in the proof of Lemma~\ref{lemma:mec_invariant}.
Fix a neighborhood $\nu_W^{T_0}(S)$ of $S$ in $W$ such that $\nu_W^{T_0}(S)\cap \partial W=\nu^{T_0}(S)$.
Define a new Liouville domain $\tilde W$ by attaching a subcritical $k$-handle along $S$ with framing $\epsilon$, so $\tilde W:=W\cup_{S,\epsilon} k-handle$.
Attach $(\tilde \Sigma\times [1,\infty[,d(t\tilde \alpha_{T_0}) \, )$ along a collar neighborhood of the boundary to obtain a complete Liouville manifold $\tilde W_c$.

We define $\tilde H_T:\tilde W\to \R$ by putting
\begin{itemize}
\item $\tilde H_T=H_T$ outside the handle and the neighborhood $\nu_W^{T_0}(S)$
\item inside $\nu_W^{T_0}(S)$ and on the handle we choose the function as indicated in Figure~\ref{fig:ham_levels}. 
\end{itemize}

\begin{figure}[htp]
\def\svgwidth{0.75\textwidth}%
\begingroup\endlinechar=-1
\resizebox{0.75\textwidth}{!}{%
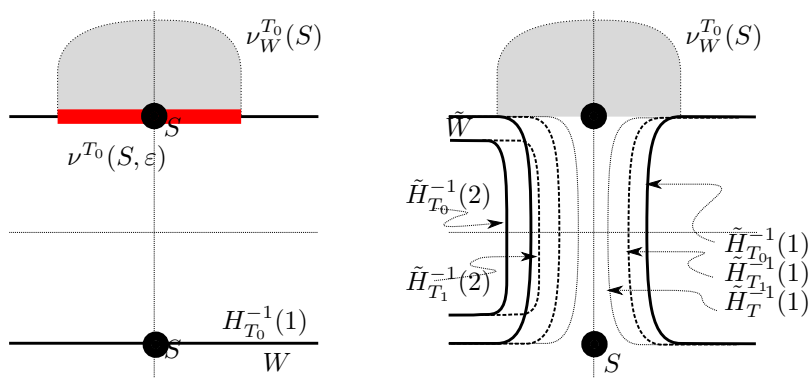%
}\endgroup
\caption{Level sets of the Hamiltonian $H_T$}
\label{fig:ham_levels}
\end{figure}

Note that for $T>T_1$ we have $\tilde H_{T}(x)\geq \tilde H_{T_1}(x)$ for all $x\in \tilde W$.
Now extend $\tilde H_T$ to $\tilde W_c$ as a function that is linear at infinity with slope $T$ with respect to the contact form $\tilde \alpha_T$; choose this extension such that $\tilde H_{T}(x)\geq \tilde H_{T_1}(x)$ for all $x\in \tilde W_c$.
\end{proof}

\subsection*{Acknowledgements}
This note grew out of some questions asked at the AIM workshop on ``Contact topology in higher dimensions'' in May 2012, Palo Alto. 
OvK would like to thank American Institute for Mathematics for their hospitality.
We also thank Patrick Massot and Urs Frauenfelder for helpful comments and suggestions.
MK and OvK are supported by the NRF Grant 2012-011755 funded by the Korean government.

\address{
Myeonggi Kwon and Otto van Koert\\
Department of Mathematical Sciences and Research Institute of Mathematics,\\
Seoul National University\\
Building 27, room 402, San 56-1,\\
Sillim-dong, Gwanak-gu, Seoul\\
South Korea\\
\email{kwunmk04@snu.ac.kr\\
okoert@snu.ac.kr}}

\end{document}